\documentclass[11pt,a4paper]{article}
\usepackage{titlesec}
\usepackage{a4wide}
\usepackage{graphicx}
\usepackage{float}
\usepackage{amssymb}
\usepackage{amsmath}
\usepackage{amsthm}
\usepackage{appendix}
\usepackage{color}
\usepackage{array}
\usepackage{eucal}
\usepackage{mathrsfs}
\usepackage{esint}
\usepackage{lmodern}
\usepackage{tikz}
\usepackage{bbm}
\usepackage{hyperref}
\usepackage{geometry}
\usepackage{changepage}
\changepage{0pt}{}{}{}{}{0pt}{}{0pt}{10pt}
\usepackage[numbers]{natbib}
\hypersetup{
	pdfpagemode=UseThumbs,
	pdftoolbar=true,        % show Acrobat's toolbar?
	pdfmenubar=true,        % show Acrobats menu?
	pdffitwindow=false,     % window fit to page when opened
	pdfstartview={Fit},    % fits the width of the page to the window
	pdftitle={Generalized impedance boundary conditions with vanishing or sign-changing impedance},    % title
	pdfauthor={L. Bourgeois, L. Chesnel},     % author
	pdfsubject={},  % subject of the document
	pdfcreator={L. Bourgeois, L. Chesnel},   % creator of the document
	pdfproducer={L. Bourgeois, L. Chesnel}, % producer of the document
	pdfkeywords={}, % list of keywords
	pdfnewwindow=true,      % links in new window
colorlinks=true,       % false: boxed links; true: colored links
linkcolor=magenta,          % color of internal links
citecolor=red,        % color of links to bibliography
filecolor=cyan,      % color of file links
urlcolor=blue           % color of external links
}
\usetikzlibrary{shapes.misc}

\newcommand{\dsp}{\displaystyle}

\newcommand{\eps}{\varepsilon}
\newcommand{\om}{\omega}
\newcommand{\Om}{\Omega}
\newcommand{\mrm}[1]{\mathrm{#1}}

\newcommand{\Cplx}{\mathbb{C}}
\newcommand{\N}{\mathbb{N}}
\newcommand{\R}{\mathbb{R}}

\newcommand{\mL}{\mrm{L}}
\newcommand{\mH}{\mrm{H}}
\newcommand{\mV}{\mrm{V}}

\newcommand{\mX}{\mrm{X}}

\newcommand{\mW}{\mrm{W}}

\usepackage{bbm}

\newtheorem{theorem}{Theorem}[section]
\newtheorem{lemma}[theorem]{Lemma}
\newtheorem{remark}[theorem]{Remark}

\newtheorem{proposition}[theorem]{Proposition}

\newcommand {\be}{\begin{equation}}
\newcommand {\ee}{\end{equation}}

\usepackage{dsfont}
\everymath{\displaystyle}
\begin{document}
\begin{center}
{\sc \bf\fontsize{20}{20}\selectfont  	
Generalized impedance boundary conditions \\[6pt]
with vanishing or sign-changing impedance}
\end{center}
\begin{center}
	\textsc{Laurent Bourgeois}$^1$, \textsc{Lucas Chesnel}$^2$\\[16pt]
	\begin{minipage}{0.9\textwidth}
		{\small
$^1$  POEMS, CNRS, Inria, ENSTA Paris, Institut Polytechnique de Paris, 91120 Palaiseau, France;\\
$^2$ Inria, ENSTA Paris, Institut Polytechnique de Paris, 91120 Palaiseau, France.\\[10pt]
			E-mails:
			\texttt{Laurent.Bourgeois@ensta-paris.fr}, \texttt{Lucas.Chesnel@inria.fr}.\\[-14pt]
			\begin{center}
				(\today)
			\end{center}
		}
	\end{minipage}
\end{center}
\textbf{Abstract:} We consider a Laplace type problem with a generalized impedance boundary condition of the form $\partial_\nu u=-\partial_x(g\partial_xu)$ on a flat part $\Gamma$ of the boundary. Here $\nu$ is the outward unit normal vector to $\partial\Om$, $g$ is the impedance function and $x$ is the coordinate along $\Gamma$.  Such problems appear for example in the modelling of small perturbations of the boundary. In the literature, the cases $g=1$ or $g=-1$ have been investigated. In this work, we address situations where $\Gamma$ contains the origin and $g(x)=\mathds{1}_{x>0}(x)x^\alpha$ or $g(x)=-\mrm{sign}(x)|x|^\alpha$ with $\alpha\ge0$. 
In other words, we study cases where $g$ vanishes at the origin and changes its sign. The main message is that the well-posedness in the Fredholm sense of the corresponding problems depends on the value of $\alpha$. For $\alpha\in[0,1)$, we show that the associated operators are Fredholm of index zero while it is not the case when $\alpha=1$. The proof of the first results is based on the reformulation as 1D problems combined with the derivation of compact embedding results for the functional spaces involved in the analysis. The proof of the second results relies on the computation of singularities and the construction of Weyl's sequences. We also discuss the equivalence between the strong and weak formulations, which is not straightforward. Finally, we provide simple numerical experiments which seem to corroborate the theorems.\\[6pt]

\section{Introduction}

Generalized Impedance Boundary Conditions (GIBCs) are often used in the context of asymptotic analysis for partial differential equations to obtain simplified models. Imagine for example that one is interesting in the scattering of an electromagnetic wave by an inclusion of perfectly conducting material coated with a thin dielectric layer of variable thickness. One can show that the solution of the corresponding problem is well approximated by the solution of a scattering problem for the inclusion alone supplemented with an \textit{ad hoc} second-order GIBC. In this model, the complexity of the initial geometry is incorporated in the boundary condition. This can be useful in particular to reduce computational costs because it allows one to avoid meshing the thin layer, see e.g. \cite{aslanyurek_haddar_sahinturk}. Note that this approach can also be exploited to solve the inverse problem consisting in finding information on the obstacle from the measurement of scattered fields  \cite{BoCH11,bourgeois_chaulet_haddar}. For modelling aspects and derivation of GIBCs in electromagnetism, we refer the reader to \cite{HaJN05}. Similarly, in aeroacoustics, the so-called Ingard-Myers boundary conditions are used to model the presence of a liner located on the surface of a duct \cite{luneville_mercier}. They also consist of a second-order GIBC.\\

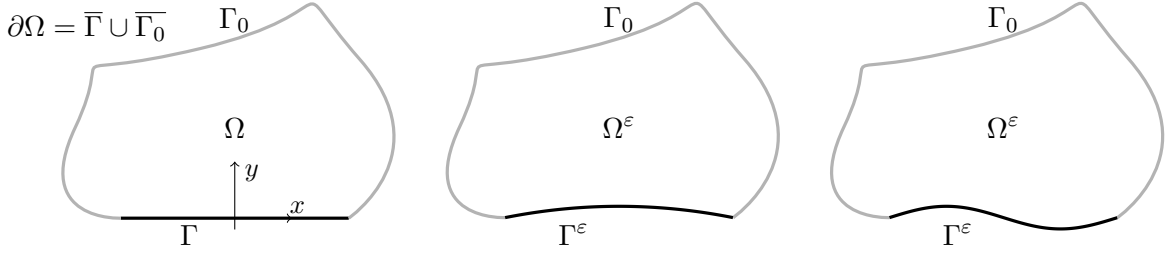
\begin{figure}[!ht]
\centering
\begin{tikzpicture}[scale=1.5]
\draw[very thick,black] (-1,0)--(1,0);
\draw[very thick,gray!60] plot [smooth, tension=2] coordinates {(-1,0)(-1.4,0.8)(-0.3,1.5)(1.1,1.4)(1,0)};
\node at (-0.4,-0.15){$\Gamma$};
\node at (0,1.75){$\Gamma_0$};
\node at (0,0.8){$\Om$};
\node at (-1.3,1.7){$\partial\Om=\overline{\Gamma}\cup\overline{\Gamma_0}$};
\begin{scope}[shift={(-3.1,-0.2)}]
\draw[->] (3,0.2)--(3.6,0.2);
\draw[->] (3.1,0.1)--(3.1,0.7);
\node at (3.65,0.3){\small $x$};
\node at (3.25,0.6){\small $y$};
\end{scope}
\end{tikzpicture}\hspace{-0.4cm}\begin{tikzpicture}[scale=1.5]
\draw[very thick,gray!60] plot [smooth, tension=2] coordinates {(-1,0)(-1.4,0.8)(-0.3,1.5)(1.1,1.4)(1,0)};
\node at (-0.4,-0.15){$\Gamma^\eps$};
\node at (0,1.75){$\Gamma_0$};
\node at (0,0.8){$\Om^\eps$};
\draw[black,very thick] plot[domain=-1:1,samples=100] (\x,{0.1*(1-\x*\x)});
\end{tikzpicture}\hspace{-0.4cm}\begin{tikzpicture}[scale=1.5]
\draw[very thick,gray!60] plot [smooth, tension=2] coordinates {(-1,0)(-1.4,0.8)(-0.3,1.5)(1.1,1.4)(1,0)};
\node at (-0.4,-0.15){$\Gamma^\eps$};
\node at (0,1.75){$\Gamma_0$};
\node at (0,0.8){$\Om^\eps$};
\draw[black,very thick] plot[domain=-1:1,samples=100] (\x,{-0.1*(sin(180*\x)});
\end{tikzpicture}
\caption{Left: domain $\Om$. Center and right: examples of perturbed domains $\Om^\eps.$\label{Domain}}
\end{figure}

In order to describe the content of this article, let us present in more details another situation where GIBCs arise. Let $\Om\subset\R^2$ be an open, connected, bounded set with a Lipschitz continuous boundary $\partial\Om$. Additionally, assume that there holds $\Gamma\subset\partial\Om$ with $\Gamma:=(-1,1)\times\{0\}$ and  set $\Gamma_0:=\partial\Om\setminus\overline{\Gamma}$ (see Figure \ref{Domain} left). Now let us perturb slightly, in a smooth way, the boundary of $\Om$. To proceed, introduce some smooth profile function $g$ supported in $[-1,1]$ satisfying $g(-1)=g(1)=0$ and for $\eps>0$ small, define the domain $\Om^\eps$ such that 
\[
\partial\Om^\eps=\overline{\Gamma_0}\cup\Gamma^\eps\quad\mbox{ with }\quad\Gamma^\eps:=\{(x,\eps g(x))\,|\,x\in(-1,1)\}.
\]
Examples of such geometries are given in Figure \ref{Domain} center and right. In $\Om^\eps$, we study the model problem
\be \label{fortEps} 
\begin{array}{|rccl}
-\Delta u^\eps + u^\eps &=& f&  \mbox{in } \Omega^\eps\\[3pt]
\partial_{\nu^\eps} u^\eps&=&0 & \mbox{on } \partial\Om^\eps,
\end{array}	
\ee
where $f$ is a given source term which vanishes in a neighbourhood of $\Gamma$ and $\nu^\eps$ stands for the outward unit normal vector to $\partial\Om^\eps$. Let us recall how to obtain formally an asymptotic expansion of $u^\eps$ with respect to $\eps$ small. Consider the ansatz 
\begin{equation}\label{Expansion}
u^{\eps}=u_0+\eps u_1+\dots
\end{equation}
where $u_0$, $u_1$ are functions to determine and where the dots correspond to higher order terms. On $\Gamma^\eps$, we have the expansions
\begin{equation}\label{ExpansionNormal}
\nu^\eps = \cfrac{1}{\sqrt{1+\eps^2(g'(x))^2}}\,\left( \begin{array}{c}
\eps g'(x)\\
-1
\end{array}
\right)=\left( \begin{array}{c}
0\\
-1
\end{array}
\right)+\eps\left( \begin{array}{c}
g'(x)\\
0
\end{array}
\right)+\dots
\end{equation}
\begin{equation}\label{TaylorExpansion}
\begin{array}{rcl}
\nabla u^\eps(x,\eps g(x))&=&\nabla u^\eps(x,0)+\eps g(x)\left( \begin{array}{c}
\partial^2_{xy}u_\eps(x,0)\\[2pt]
\partial^2_{yy}u_\eps(x,0)
\end{array}
\right)+\dots\\[15pt]
&=&\nabla u^\eps(x,0)+\eps g(x)\left( \begin{array}{c}
\partial^2_{xy}u_\eps(x,0)\\[2pt]
-\partial^2_{xx}u_\eps(x,0)+u_\eps(x,0)
\end{array}
\right)+\dots\,.
\end{array}
\end{equation}
Now we insert (\ref{Expansion}) in (\ref{fortEps}) and exploit  (\ref{ExpansionNormal}), (\ref{TaylorExpansion}). Collecting the terms of orders $\eps^0$, $\eps^1$, we find that $u_0$, $u_1$ satisfy respectively the problems
\be \label{fort0} 
\begin{array}{|rccl}
-\Delta u_0 + u_0 &=& f&  \mbox{in } \Omega\\[3pt]
\partial_{\nu} u_0&=&0 & \mbox{on } \partial\Om
\end{array}\qquad\qquad \begin{array}{|rccl}
-\Delta u_1 + u_1 &=& 0&  \mbox{in } \Omega\\[3pt]
\partial_{\nu} u_1&=&0 & \mbox{on } \Gamma_0\\[3pt]
\partial_{\nu} u_1&=&-\partial_x(g(x)\partial_x u_0)+g(x)u_0 & \mbox{on } \Gamma.
\end{array}
\ee
Then by rectifying the boundary of $\Om^\eps$ using ``almost identical'' diffeomorphisms to transform the perturbed domain into the original geometry $\Om$ (see e.g. \cite[Chap.\,7,\,\S6.5]{Kato95}), under additional assumptions of regularity for $f$, one can prove the estimate, for $\eps$ small enough,
\[
\|u^\eps-(u_0+\eps u_1)\|_{\mH^1(\Om\setminus\overline{\om})} \le C\eps^2,
\]
where $C$ is a constant independent of $\eps$ and $\om$ is a neighbourhood of $\Gamma$. In practice, as mentioned above, instead of computing successively $u_0$, $u_1$, $\dots$, very often one prefers to work with a model problem, whose dependence with respect to $\eps$ is rather explicit, which provides via a simple calculation an approximation of $u^\eps$ up to a given order. In our case, to approach $u_0+\eps u_1$ in one shot, one can consider the problem 
\be \label{fortModel} 
\begin{array}{|rccl}
-\Delta u + u &=& f&  \mbox{in } \Omega\\[3pt]
\partial_{\nu} u&=&0 & \mbox{on } \Gamma_0\\[3pt]
\partial_{\nu} u&=&-\eps\partial_x(g(x)\partial_x u)+\eps g(x)u & \mbox{on } \Gamma.
\end{array}
\ee
The condition on $\Gamma$ appearing in (\ref{fortModel}), which makes the analysis of this problem not straightforward, is a particular instance of the so-called Ventcel boundary conditions, which are themselves a subclass of GIBCs. Ventcel boundary conditions are second order differential conditions which have been named after the pioneering works of Feller and Ventcel \cite{Fell52,Vent56,Fell57,Vent59}. Since then, they have been many studies concerning the Laplace operator with Ventcel boundary conditions \cite{FGGOR03,AMPR03,DaKa12,CaDK13,DKL16,CLNV19,bonnaillie_dambrine_herau_vial,chamaillard_chaulet_haddar}. For investigations in non smooth domains, one can consult \cite{Lemr85,PoSl13,NiLM17}. The value of $\eps>0$ as well as the term $\eps gu$ plays no major role in the well-posedness in the Fredholm sense of (\ref{fortModel}) and to simplify, we shall work with the condition
\begin{equation}\label{InitialCondition}
\partial_\nu u =-\partial_x(g(x) \partial_x u)\quad\mbox{ on }\Gamma,
\end{equation}
where $\nu$ is the outward unit normal vector to $\partial\Om$. Up to now, in most of the articles mentioned above only the cases $g=-1$ or $g=1$ on the whole boundary have been considered. However, for certain problems, as the one which led us to (\ref{fortModel}), it may be relevant to consider some $g$ which vanish or whose sign is not constant on $\Gamma$.
For instance, such situation arises in a water wave problem where the bottom shape small variation of the ocean with respect to a reference flat depth is approached by the boundary condition (\ref{InitialCondition}), like in \cite{bourgeois_mercier_terrine}. In this case the function $g$ coincides with this shape variation and is likely to vanish or change sign. Then it is natural to wonder what can be said concerning the well-posedness of the corresponding problem in that situations. This is precisely the goal of the present article to address this question. Note that a variable $g$ is allowed in \cite{GrVi20} but it does not vanish. Besides, let us mention that degenerate elliptic problems are studied in \cite{NCPDC20} in the context of modelling of resonant waves in 2D plasma. These problems share similarities with ours but are nonetheless different.\\
\newline
Below, we will focus our attention on two model problems. Let us divide $\Gamma$ into the two segments 
\[
\Gamma_-:=(-1,0)\times\{0\}\qquad\mbox{ and }\qquad\Gamma_+:=(0,1)\times\{0\}. 
\]
To study the case of a $g$ which vanishes at a point (the origin), we will work on the variational problem whose strong formulation is
\be \label{fort} 
\begin{array}{|rccl}
-\Delta u + u &=& f&  \mbox{in } \Omega\\[3pt]
\partial_\nu u&=&0 & \mbox{on } \partial\Om\setminus\overline{\Gamma_+}\\[3pt]
\partial_\nu u &=& s\,\partial_x (x^\alpha\, \partial_x u) & \mbox{on } \Gamma_+,
\end{array}	
\ee
where $s=1$ or $s=-1$ and $\alpha \geq 0$ is a real number.
Moreover here and in the rest of the article, $f$ is a given element of $\mL^2(\Om)$. 

To address the situation where $g$ changes sign in (\ref{InitialCondition}), we will study the variational problem whose strong formulation is
\be\label{fort_bis}
\begin{array}{|rccl}
-\Delta u+u &=& f&  \mbox{in } \Omega\\
\partial_\nu u &=& 0&  \mbox{on } \Gamma_0\\
\partial_\nu u&=& -\partial_x (|x|^\alpha\, \partial_x u) & \mbox{on } 
\Gamma_-\\
\partial_\nu u &=& \partial_x (x^\alpha\, \partial_x u) & \mbox{on } \Gamma_+,
\end{array}
\ee
with again $\alpha\ge0$. 
In (\ref{fort}) and (\ref{fort_bis}), the coefficient $\alpha$
can be viewed as a simple polynomial rate which characterizes how fast the impedance vanishes at the origin.
The most relevant case is probably $\alpha=1$, which in (\ref{fort_bis}) corresponds to a smooth change of sign. The situation $\alpha=0$, i.e. a jump of the impedance function which may seem less realistic because one needs smoothness in the asymptotic expansion leading to (\ref{fortModel}), is also considered in our paper. We will see that the value of $\alpha$ plays a crucial role in the results of well-posedness.\\
\newline
The outline is as follows. We start by presenting the problems and the main results in Section \ref{SectionMainResults}. In Section \ref{Section1D}, we prove a series of results concerning weighted Sobolev spaces in 1D which will be used in the analysis. In Section \ref{SectionAlphaMoins1}, we establish the well-posedness in the Fredholm sense of the operator, denoted by $A$, see (\ref{DefOpA}), associated with the variational formulation leading to (\ref{fort}) in the case $\alpha\in[0,1)$. Then we prove that for $\alpha=1$, $A$ is not of Fredholm type (see Section \ref{SectionWithSingus}). In Section \ref{SectionVanishing}, we study the operator associated with the variational formulation leading to (\ref{fort_bis}). In Section \ref{SectionWeakStrong}, we analyse the equivalence between strong and weak formulations. Finally, we present the results of simple numerical experiments which seem to corroborate our theorems before giving a few concluding remarks in Section \ref{SectionConcludingRemarks}.

\section{Main results}\label{SectionMainResults}

\subsection{Definition of the problems}

To consider the case of an impedance vanishing at the origin, we introduce the space 
\[
\mV_\alpha(\Om):=\{v \in \mH^1(\Omega)\,|\,x^{\alpha/2} \partial_x v \in \mL^2(\Gamma_+)\}\]		
and study the weak formulation:
\be\label{faible}
\begin{array}{|l}
\mbox{Find }u\in\mV_\alpha(\Om)\mbox{ such that for all }v \in \mV_\alpha(\Om)\\[4pt]
\int_\Omega \nabla u\cdot \nabla \overline{v}+u\overline{v}\,dxdy + s\int_{\Gamma_+}x^\alpha\, \partial_x u\, \partial_x \overline{v}\,dx=\int_\Omega f\overline{v}\,dxdy. 
\end{array}	
\ee 
Here the term $x^\alpha$ is our impedance function. We denote by $a(\cdot,\cdot)$ the sesquilinear form appearing in the left hand side of (\ref{faible}) and by $\ell(\cdot)$ the antilinear form of the right hand side. For the space $\mV_\alpha(\Om)$, we have the following result.
\begin{proposition}\label{hilbert}
For all $\alpha \geq 0$,  $\mV_\alpha(\Om)$ is a Hilbert space when equipped with the inner product
\[(u,v)_{\mV_\alpha(\Om)}=\int_\Omega \nabla u\cdot \nabla \overline{v}+u\overline{v}\,dxdy + \int_{\Gamma_+}x^\alpha\, \partial_x u\, \partial_x \overline{v}\,dx.\]
\end{proposition}
	\begin{proof}
	Let us consider a Cauchy sequence $(u_n)_{n \in \mathbb{N}}$ in $\mV_\alpha(\Om)$.
	Clearly, since $\mH^1(\Omega)$ and $\mL^2(\Gamma_+)$ are Hilbert spaces, there exist $u \in \mH^1(\Omega)$ such that
	$u_n \rightarrow u$ in $\mH^1(\Omega)$ and $w \in \mL^2(\Gamma_+)$ such that $x^{\alpha/2}\partial_x u_n|_{\Gamma_+} \rightarrow w$ in $\mL^2(\Gamma_+)$. Since the trace mapping is continuous from $\mH^1(\Omega)$ to $\mH^{1/2}(\Gamma_+)$, there holds $u_n|_{\Gamma_+} \rightarrow u|_{\Gamma_+}$ in $\mH^{1/2}(\Gamma_+)$, which implies that
	$u_n|_{\Gamma_+} \rightarrow u|_{\Gamma_+}$ in $\mathcal{D}'(\Gamma_+)$, the space of distributions on $\Gamma_+$. From this, we infer that $\partial_x u_n|_{\Gamma_+} \rightarrow \partial_x u|_{\Gamma_+}$ in $\mathcal{D}'(\Gamma_+)$. Using that $x^{\alpha/2} \in \mathscr{C}^\infty(\Gamma_+)$, we deduce $x^{\alpha/2}\partial_x u_n|_{\Gamma_+} \rightarrow  x^{\alpha/2}\partial_x u|_{\Gamma_+}$ in $\mathcal{D}'(\Gamma_+)$. We conclude that $w=x^{\alpha/2}\partial_x u|_{\Gamma_+}$, which completes the proof.
\end{proof}	
\begin{remark}
This result is also true for $\alpha<0$ but we will not consider this case in the following.
\end{remark}
It is readily seen that the forms $a(\cdot,\cdot)$ and $\ell(\cdot)$ are respectively continuous on $\mV_\alpha(\Om)\times\mV_\alpha(\Om)$ and $\mV_\alpha(\Om)$. With the help of the Riesz representation theorem, hence we can define the bounded operator $A: \mV_\alpha(\Om) \rightarrow \mV_\alpha(\Om)$ and the element $L \in \mV_\alpha(\Om)$ such that
\begin{equation}\label{DefOpA}
(Au,v)_{\mV_\alpha(\Om)}=a(u,v)\qquad\mbox{ and }\qquad (L,v)_{\mV_\alpha(\Om)}=\ell(v), \qquad \forall u,v \in \mV_\alpha(\Om).\end{equation}
With this notation, the weak formulation (\ref{faible}) is equivalent to find $u \in \mV_\alpha(\Om)$ such that $Au=L$.\\
\newline
To investigate the case of an impedance which is both vanishing at the origin and whose sign is non constant, we introduce the space, for $\alpha \geq 0$,
\[
\mW_\alpha(\Om):=\{v \in \mH^1(\Omega)\,|\,|x|^{\alpha/2} \partial_x v \in \mL^2(\Gamma_-),\,x^{\alpha/2} \partial_x v \in \mL^2(\Gamma_+)\},
\]	
and consider the weak formulation
\be\label{faible bis}
\begin{array}{|l}
\mbox{Find }u\in\mW_\alpha(\Om)\mbox{ such that for all }v \in \mW_\alpha(\Om)\\[4pt]
\int_\Omega \nabla u\cdot \nabla \overline{v}+u\overline{v}\,dxdy  -\int_{\Gamma_-}|x|^\alpha\, \partial_x u\, \partial_x \overline{v}\,dx+ \int_{\Gamma_+}x^\alpha\, \partial_x u\, \partial_x \overline{v}\,dx=\int_\Omega f\overline{v}\,dxdy. 
\end{array}	
\ee	
We denote by $b(\cdot,\cdot)$	 the sesquilinear form appearing in the left hand side of (\ref{faible bis}). By adapting the proof of Proposition \ref{hilbert}, one shows that $\mW_\alpha(\Om)$ is a Hilbert space for all $\alpha \geq 0$ when equipped with the inner product
\[(u,v)_{\mW_\alpha(\Om)}=\int_\Omega \nabla u\cdot \nabla \overline{v}+u\overline{v}\,dxdy + \int_{\Gamma_-}|x|^\alpha\, \partial_x u\, \partial_x \overline{v}\,dx + \int_{\Gamma_+}x^\alpha\, \partial_x u\, \partial_x \overline{v}\,dx.\]
With the Riesz representation theorem, we define the bounded operator $B: \mW_\alpha(\Om) \rightarrow \mW_\alpha(\Om)$ and the element $\tilde{L} \in \mW_\alpha(\Om)$ such that
\begin{equation}\label{DefOpB}
(Bu,v)_{\mW_\alpha(\Om)}=b(u,v)\qquad\mbox{ and }\qquad (\tilde{L},v)_{\mW_\alpha(\Om)}=\ell(v), \qquad \forall u,v \in \mW_\alpha(\Om).
\end{equation}
With this notation, the weak formulation (\ref{faible bis}) is equivalent to find $u \in \mW_\alpha(\Om)$ such that $Bu=\tilde{L}$.

\subsection{Statement of the results}

We start with the variational formulation (\ref{faible}). Concerning the relation with the strong problem (\ref{fort}), we have the following result:
\begin{theorem}\label{faible_fort}
For $s=\pm 1$, for any $\alpha \geq 0$, if $u$ satisfies (\ref{faible}), then it solves (\ref{fort}) where the GIBC on $\Gamma_+$ holds in $\mathcal{D}'(\Gamma_+)$ (distributional sense).
\end{theorem}	

Next we consider the well-posedness of (\ref{faible}). In the good sign case $s=1$, we have the following theorem, which is a direct application of the Lax-Milgram lemma since the bilinear form $a(\cdot,\cdot)$ is coercive (the proof is omitted because it is straightforward).
\begin{theorem}\label{good_sign}
For $s=1$, the weak formulation (\ref{faible}) has a unique solution for all $\alpha \geq 0$. 
\end{theorem}
The bad sign case $s=-1$ is more delicate to address because then the coercivity of $a(\cdot,\cdot)$ in $\mV_\alpha(\Om)\times\mV_\alpha(\Om)$ is not clear. The main achievement of the present paper is to emphasize that the well-posedness (in the Fredolm sense) of (\ref{faible})
depends on the parameter $\alpha$.
We prove a positive result for $\alpha \in [0,1)$ and a negative result for $\alpha=1$. 
\begin{theorem}\label{casgeneral}
For $s=-1$ and $\alpha \in [0,1)$, the operator $A: \mV_\alpha(\Om) \rightarrow \mV_\alpha(\Om)$ is Fredholm of index zero. As a consequence, Problem (\ref{faible}) admits a solution when $A$ is injective.
\end{theorem}
\begin{theorem}\label{cas1}
For $s=-1$ and $\alpha=1$, the operator $A: \mV_\alpha(\Om) \rightarrow \mV_\alpha(\Om)$ is not of Fredholm type.
\end{theorem}

\noindent Now we present the results for the variational formulation (\ref{faible bis}) with an impedance which is both vanishing and sign-changing. They are very similar to the ones of Theorems \ref{faible_fort}, \ref{casgeneral} and \ref{cas1}.  
\begin{theorem}\label{faible_fort_bis}
For any $\alpha \geq 0$, if $u$ satisfies (\ref{faible bis}), then it solves (\ref{fort_bis}) where the GIBC on $\Gamma_\pm$ holds in $\mathcal{D}'(\Gamma_\pm)$.
\end{theorem}	
\begin{theorem}\label{casgeneral_bis}
For $\alpha \in [0,1)$, the operator $B: \mW_\alpha(\Om) \rightarrow \mW_\alpha(\Om)$ is Fredholm of index zero. As a consequence, Problem (\ref{faible bis}) admits a solution when $B$ is injective.
\end{theorem}
\begin{theorem}\label{cas1_bis}
For $\alpha=1$, the operator $B: \mW_\alpha(\Om) \rightarrow \mW_\alpha(\Om)$ is not of Fredholm type.
\end{theorem}
\begin{remark}
Comparing all these results, we can say that the well-posedness of Problem (\ref{faible bis}) is more affected by the fact that the impedance is vanishing than by the fact that its sign is non constant. 
\end{remark}
Among the important results of this paper, let us also mention Theorem \ref{faible_fort_ter}, in which we establish the equivalence between the weak and the  strong formulation of a slightly more realistic problem than those presented in this section.  

In order to prove some of the theorems above, we need to establish results for weighted Sobolev spaces in 1D. This is the subject of the next section.

\section{Weighted Sobolev spaces in dimension 1}\label{Section1D}
Denote by $I$ the interval $(0,1)$. For $\alpha \geq 0$, define the space
\be\label{V_alpha} 
\mX_\alpha(I):=\{\varphi\in \mL^2(I)\,|\,x^{\alpha/2} d_x \varphi \in \mL^2(I)\}.
\ee 
Equipped with its natural inner product
\[(\varphi,\psi)_{\mX_\alpha(I)}=\int_I \varphi\,\overline{\psi}\,dx + \int_{I}x^\alpha\, d_x \varphi\, d_x \overline{\psi}\,dx,
\]
it is clear that $\mX_\alpha(I)$ is a Hilbert space.
We remark that $\mX_0(I)=\mH^1(I)$, where $\mH^1(I)$ stands for the classical Sobolev space, and $\mX_\alpha(I) \subset \mX_\beta(I)$ for $\alpha \leq \beta$. Additionally, the functions of $\mX_\alpha(I)$ belong to $\mH^1(\eps,1)$ for any $\eps>0$, and so are continuous on $(0,1]$.

\begin{proposition}\label{prop-H12}
The space $\mX_1(I)$ is continuously embedded in $\mH^{1/2}(I)$.
\end{proposition}
\begin{remark}
A consequence of this result is that all the spaces $\mX_\alpha(I)$, $\alpha\le1$, are continuously embedded in $\mH^{1/2}(I)$.
\end{remark}
\begin{proof}
Let $\varphi$ be an element of $\mX_1(I)$. Consider the quarter of disk $\Omega_+:=\{(x,y) \in \mathbb{R}^2,\,\,x> 0,\,y >0,\,x^2+y^2 <1\}$ and define the function $u$ such that $u(x,y)=\varphi((x^2+y^2)^{1/2})$ in $\Om_+$. We observe that $u \in \mH^1(\Omega_+)$ with
\[\|u\|_{\mH^1(\Omega_+)}^2=\frac{\pi}{2}\int_0^1\left(|\varphi|^2+|d_r \varphi|^2\right)rdr \leq \frac{\pi}{2} \int_0^1\left(|\varphi|^2+r|d_r \varphi|^2\right)dr=\frac{\pi}{2}\|\varphi\|^2_{\mX_{1}(I)}.\]
By using the continuity of the trace mapping from $\mH^1(\Omega_+)$ to $\mH^{1/2}(\Gamma_+)$, we obtain that there exists a constant $C>0$ such that
 \[\|\varphi\|^2_{\mH^{1/2}(I)}=\|u|_{\Gamma_+}\|_{\mH^{1/2}(\Gamma_+)}^2 \leq C\,\|u\|_{\mH^1(\Omega_+)}^2=C\,\frac{\pi}{2}\|\varphi\|^2_{\mX_{1}(I)}.
\]
This ends the proof.
\end{proof}
In our analysis below, we will need more refined results of regularity for the functions of $\mX_\alpha(I)$. 
\begin{proposition}\label{prop-Holder}
For $\alpha<1$, $\mX_\alpha(I)$ is continuously embedded in the H\"older space $\mathscr{C}^{0,\frac{1-\alpha}{2}}(I)$. 
\end{proposition} 
\begin{remark}\label{RemarkComp}
This implies that for $\alpha<1$, $\mX_\alpha(I)$ is compactly embedded in $\mathscr{C}^{0}(\overline{I})$ endowed with the norm $\|\cdot\|_{\mL^{\infty}(I)}:=\sup_{I}|\cdot|$.
\end{remark}
\begin{proof}
Consider some $\varphi\in\mX_\alpha(I)$. For $0<x\le y <1$, we can write
\[
\varphi(y)-\varphi(x)=\int_x^y d_t\varphi\,dt=\int_x^y t^{-\alpha/2}\,t^{\alpha/2}d_t\varphi\,dt.
\]
Using the Cauchy-Schwarz inequality in $\mL^2(x,y)$, this implies
\begin{equation}\label{EstimHolder0}
|\varphi(y)-\varphi(x)| \le \left(\int_x^y t^{-\alpha}\,dt\right)^{1/2}\,\|\varphi\|_{\mX_\alpha(I)} = \left(\cfrac{y^{1-\alpha}-x^{1-\alpha}}{1-\alpha}\right)^{1/2}\,\|\varphi\|_{\mX_\alpha(I)}.
\end{equation}
Now a simple analysis guarantees that for $0<x\le y <1$, there holds
\begin{equation}\label{EstimHolder}
y^{1-\alpha}-x^{1-\alpha} \le (y-x)^{1-\alpha}. 
\end{equation}
Indeed, (\ref{EstimHolder}) is equivalent to have $g(t)\ge0$ with $g(t)=(t-1)^{\eps}-t^{\eps}+1$, $\eps:=1-\alpha$ and $t=y/x\ge1$. Since $g'(t)=\eps((t-1)^{\eps-1}-t^{\eps-1})=\eps t^{\eps-1}((t/(t-1))^{1-\eps}-1)$, we deduce that $g'$ is positive on $(1,+\infty)$. As a consequence, $g$ is increasing on $(1,+\infty)$. Since $g(1)=0$, we infer that $g$ is non negative on  $[1,+\infty)$. Thus (\ref{EstimHolder}) is valid and inserting this estimate into (\ref{EstimHolder0}) gives the result of the proposition.
\end{proof}

We are now in a position to establish the main result of this section.
\begin{proposition}\label{prop-compacite}
For $\alpha<1$, the space $\mX_\alpha(I)$ is compactly embedded in $\mH^{1/2}(I)$.
\end{proposition} 
\begin{proof}
Let us consider a bounded sequence $(\varphi_n)$ of elements of $\mX_\alpha(I)$. We wish to prove that there exists a subsequence of $(\varphi_n)$ which converges in $\mH^{1/2}(I)$. According to Proposition \ref{prop-Holder}, $\mX_\alpha(I)$ is continuously embedded in $\mathscr{C}^{0,\frac{1-\alpha}{2}}(I)$. This implies in particular that there holds $\|\varphi_n\|_{\mL^{\infty}(I)} \le C\,\|\varphi_n\|_{\mX_\alpha(I)}$ and so $|\varphi_n(0)|\le C$. Here and below, $C>0$ is a constant which may change from one line to another but which remains independent of $n$. For $n\in\N$, let us decompose $\varphi_n$ as 
\[
\varphi_n=\tilde{\varphi}_n+\varphi_n(0),
\]
so with $\tilde{\varphi}_n:=\varphi_n-\varphi_n(0)$. According to the Bolzano-Weierstrass theorem, we can extract a subsequence of $(\varphi_n)$, still denoted $(\varphi_n)$, such that $(\varphi_n(0))$ converges in $\Cplx$. Therefore it suffices to prove that we can extract from $(\tilde{\varphi}_n)$ a subsequence which converges in $\mH^{1/2}(I)$.\\
\newline 
Since $(\tilde{\varphi}_n)$ is bounded in $\mX_\alpha(I)$, up to a subsequence, it converges weakly in $\mX_\alpha(I)$ to some $\tilde{\varphi}$. Additionally, since there holds $\tilde{\varphi}_n(0)=0$ for all $n\in\N$ and since $\mX_\alpha(I)$ is compactly embedded in $\mathscr{C}^{0}(\overline{I})$ (Remark \ref{RemarkComp}), we have $\tilde{\varphi}(0)=0$. Introduce some function $\zeta\in\mathscr{C}^{\infty}(\R)$ such that $\zeta(t)=1$ for $t\le1$, $\zeta(t)=0$ for $t\ge2$ and $0\le\zeta(t)\le1$ in $\R$. Then for $k\in\N\setminus\{0\}$, define $\zeta_k$ such that $\zeta_k(t)=\zeta(t/k)$. The triangle inequality gives
\begin{equation}\label{MainEstim}
\|\tilde{\varphi}_n-\tilde{\varphi}\|_{\mH^{1/2}(I)}\le \|\zeta_k(\tilde{\varphi}_n-\tilde{\varphi})\|_{\mH^{1/2}(I)}+\|(1-\zeta_k)(\tilde{\varphi}_n-\tilde{\varphi})\|_{\mH^{1/2}(I)}.
\end{equation}
Let us study each of the two terms of the right hand side of (\ref{MainEstim}) separately. We start with the first one. Consider again the quarter of disk $\Omega_+=\{(x,y) \in \mathbb{R}^2,\,\,x> 0,\,y >0,\,x^2+y^2 <1\}$ and define the function $u\in \mH^1(\Omega_+)$ such that 
\begin{equation}\label{Def1DTo2D}
u(x,y)=\zeta_k(r)(\tilde{\varphi}_n(r)-\tilde{\varphi}(r))\quad\mbox{ in }\Om_+\qquad\mbox{ with }\quad r=(x^2+y^2)^{1/2}.
\end{equation}
Using the continuity of the trace mapping from $\mH^1(\Omega_+)$ to $\mH^{1/2}(\Gamma_+)$ and the fact that $u$ vanishes on the curved part of $\partial\Omega_+$, which allows us to exploit a Poincar\'e inequality, we can write
\begin{equation}\label{Estim2Compactbis}
\|\zeta_k(\tilde{\varphi}_n-\tilde{\varphi})\|_{\mH^{1/2}(I)} \le C\,\| u\|_{\mH^1(\Omega_+)} \le  C\,\|\nabla u\|_{\mL^2(\Omega_+)}.
\end{equation}
But there holds
\begin{equation}\label{Estim2Compact}
\begin{array}{rcl}
\dsp\|\nabla u\|_{\mL^2(\Omega_+)}^2&=&\frac{\pi}{2}\int_0^1|d_r (\zeta_k(\tilde{\varphi}_n-\tilde{\varphi}))|^2\,rdr \\[10pt]
 & \le & \pi\int_0^1\zeta_k^2\,|d_r(\tilde{\varphi}_n-\tilde{\varphi})|^2\,rdr+\pi\int_0^1|\tilde{\varphi}_n-\tilde{\varphi}|^2(d_r\zeta_k)^2\,rdr\\[10pt]
  & \le & \pi\int_0^{2/k}\hspace{-0.1cm}r^{1-\alpha}r^{\alpha}|d_r(\tilde{\varphi}_n-\tilde{\varphi})|^2\,dr+C\,k^{-2}\pi\int_{1/k}^{2/k}|\tilde{\varphi}_n-\tilde{\varphi}|^2\,rdr \\[10pt]
  & \le  & C\,\left(k^{\alpha-1}+k^{-2}\int_{1/k}^{2/k}r^{2-\alpha}\,dr\right)\,\|\tilde{\varphi}_n-\tilde{\varphi}\|^2_{\mX_\alpha(I)} \le C\,k^{\alpha-1}.
\end{array}
\end{equation}
Note that to obtain the first inequality of the last line, we used estimate   (\ref{EstimHolder0}). We emphasize that in (\ref{Estim2Compact}), the constant $C>0$ is independent of $n$. Gathering (\ref{Estim2Compactbis}) and (\ref{Estim2Compact}), we deduce that for any $\eps>0$, we can take $k\in\N\setminus \{0\}$ large enough such that there holds 
\begin{equation}\label{Estim1Bis}
\|\zeta_k(\tilde{\varphi}_n-\tilde{\varphi})\|_{\mH^{1/2}(I)} \le \eps/2
\end{equation}
for all $n\in\N$. Let us fix one such $k$. The fact that $(\tilde{\varphi}_n)$ is bounded in $\mX_\alpha(I)$ implies that $(\tilde{\varphi}_n)$ is bounded in $\mH^1(1/k,1)$. Since this space is compactly embedded in $\mH^{1/2}(1/k,1)$, see e.g. \cite[Chap.\,1,\,Theorem\,16.3]{LiMa68}, we can extract a subsequence of $(\varphi_n)$ such that $(\tilde{\varphi}_n)$ converges strongly to $\tilde{\varphi}$ in $\mH^{1/2}(1/k,1)$. The sequence  $((1-\zeta_k)(\tilde{\varphi}_n-\tilde{\varphi}))$ is also bounded in $\mH^1(1/k,1)$ and so we can extract a subsequence which converges in $\mH^{1/2}(1/k,1)$ to some $\hat{\varphi}$. Then, starting with the triangular inequality, we can write
\[
\begin{array}{rcl}
\|\hat{\varphi}\|_{\mL^2(1/k,1)} &\le& \|\hat{\varphi}-(1-\zeta_k)(\tilde{\varphi}_n-\tilde{\varphi})\|_{\mL^2(1/k,1)}+\|(1-\zeta_k)(\tilde{\varphi}_n-\tilde{\varphi})\|_{\mL^2(1/k,1)} \\[5pt]
 &\le& \|\hat{\varphi}-(1-\zeta_k)(\tilde{\varphi}_n-\tilde{\varphi})\|_{\mH^{1/2}(1/k,1)}+\|\tilde{\varphi}_n-\tilde{\varphi}\|_{\mL^2(1/k,1)} \\[5pt]
  &\le& \|\hat{\varphi}-(1-\zeta_k)(\tilde{\varphi}_n-\tilde{\varphi})\|_{\mH^{1/2}(1/k,1)}+\|\tilde{\varphi}_n-\tilde{\varphi}\|_{\mH^{1/2}(1/k,1)}.
\end{array}
\]
Since the right hand side above can be made as small as we wish, we deduce that $\hat{\varphi}\equiv0$. This ensures that for $n$ large enough, there holds
\begin{equation}\label{Estim2Bis}
\|(1-\zeta_k)(\tilde{\varphi}_n-\tilde{\varphi})\|_{\mH^{1/2}(I)} \le \eps/2.
\end{equation}
Using (\ref{Estim1Bis}) and (\ref{Estim2Bis}) in (\ref{MainEstim}), we deduce that $(\tilde{\varphi}_n)$ converges, up to a subsequence, to $\tilde{\varphi}$ in $\mH^{1/2}(I)$, which was the searched result. 
\end{proof}

In our study, we will also need a density result, which requires the following lemma.
\begin{lemma}\label{poincare}
Assume that $\alpha\in[0,1]$. There is a constant $C>0$ (depending on $\alpha$) such that 
\[
\int_0^1 |\varphi|^2\,dx \le C\,\int_0^1 x^\alpha |d_x \varphi|^2\,dx,\qquad\forall\varphi\in\mX_\alpha(I)\mbox{ satisfying }\varphi(1)=0.
\]
\end{lemma}
\begin{proof}
Assume on the contrary that there exists a sequence $(\varphi_n)$ in $\mX_\alpha(I)$ such that for all $n\ne1$, 
\begin{equation}\label{ConditionContradiction}
\int_0^1 |\varphi_n|^2\,dx=1,\qquad 1 \ge n\,\int_0^1 x^\alpha |d_x \varphi|^2\,dx,\qquad \varphi_n(1)=0.
\end{equation}
Then $(\varphi_n)$ is bounded in $\mX_\alpha(I)$. But $\mX_\alpha(I)$ is continuously embedded in $\mX_1(I)$, which is itself continuously embedded in $\mH^{1/2}(I)$ according to Proposition \ref{prop-H12}. Since $\mH^{1/2}(I)$ is compactly embedded in $\mL^2(I)$ (see e.g. \cite[Chap.\,1,\,Theorem\,16.3]{LiMa68}), we deduce that $\mX_\alpha(I)$ is compactly embedded in $\mL^2(I)$. This ensures that there exists a subsequence of $(\varphi_n)$, still denoted $(\varphi_n)$, which converges to some $\varphi$ in $\mL^2(I)$. From the second relation of (\ref{ConditionContradiction}), we infer that $(\varphi_n)$ is a Cauchy sequence in $\mX_\alpha(I)$ and thus converges to some function in $\mX_\alpha(I)$. But uniqueness of the limit  in $\mL^2(I)$ ensures that this function is $\varphi$. Moreover, from (\ref{ConditionContradiction}), we find $d_x\varphi=0$ in $I$ and $\varphi(1)=0$, which implies $\varphi\equiv0$. This contradicts that $\|\varphi\|_{\mL^2(I)}=1$ and completes the proof.
\end{proof}

\begin{proposition}\label{density}
For $\alpha\in[0,1]$, the space $\mathscr{C}^\infty(\overline{I})$ is dense in $\mX_\alpha(I)$. 
\end{proposition}
\begin{proof}
For $\varphi\in \mX_\alpha(I)$, we define the sequence of functions $(\varphi_n)$ in $I$ such that for all $n \geq 1$,
\be\label{hn} \varphi_n(x)=\begin{array}{|cll}
\displaystyle \varphi(1/n) & {\rm for} & x \in [0,1/n] \\[3pt]
\displaystyle \varphi(x) & {\rm for} & x \in [1/n,1]. 
\end{array}
\ee
We observe that $\varphi_n \in \mX_0(I)$. According to Lemma \ref{poincare}, there is a constant $C>0$ such that
\[
\|\varphi_n-\varphi\|^2_{\mX_\alpha(I)} \le C\,\int_0^1 x^\alpha |d_x \varphi_n- d_x \varphi|^2\,dx.
\]
But we have 
\[
\int_0^1 x^\alpha |d_x \varphi_n- d_x \varphi|^2\,dx=\int_0^{1/n} x^\alpha |d_x \varphi|^2\,dx \rightarrow 0 \quad {\rm when} \quad n \rightarrow +\infty.
\]
This proves the density of $\mX_0(I)$ in $\mX_\alpha(I)$.
Since in addition $\mathscr{C}^\infty(\overline{I})$ is dense in $\mX_0(I)=\mH^1(I)$ and $\mX_0(I)$ is continuously embedded in $\mX_\alpha(I)$, we obtain that 
$\mathscr{C}^\infty(\overline{I})$ is dense in $\mX_\alpha(I)$. 
\end{proof}

\section{The vanishing case: a positive result for $s=-1$ and $\alpha \in [0,1)$}\label{SectionAlphaMoins1}

This section is devoted to the proof of Theorem \ref{casgeneral}. As already mentioned, for $s=-1$ the Lax-Milgram theorem is not directly applicable to study Problem (\ref{faible}). To circumvent this issue, we adapt the approach proposed in \cite{bonnaillie_dambrine_herau_vial,chamaillard_chaulet_haddar} to deal with the case $\alpha=0$. The main idea, roughly speaking, consists in showing that in (\ref{faible}), the volumic terms are only a compact perturbation of the operator on the boundary. To proceed, we rewrite the 2D   weak formulation (\ref{faible}) as a 1D problem and use the compactness result of Proposition \ref{prop-compacite}.\\
\newline
To state the next proposition, we need to introduce some material. 
In what follows, for $\gamma \subset \partial \Omega$ an open subpart of the boundary of $\Omega$, on the one hand $\mH^{-1/2}(\gamma)$ denotes the dual space of $\mH^{1/2}(\gamma)$.
It coincides with the distributions in $\mH^{-1/2}(\partial \Omega)$ which are supported in $\overline{\gamma}$.
On the other hand, $\tilde{\mH}^{1/2}(\gamma)$ denotes the subspace of functions in $\mH^{1/2}(\partial \Omega)$ which are supported in $\overline{\gamma}$. The dual space of $\tilde{\mH}^{1/2}(\gamma)$ is denoted $\tilde{\mH}^{-1/2}(\gamma)$. It coincides with the restrictions to $\gamma$ of the distributions in $\mH^{-1/2}(\partial \Omega)$.
Let us warn the reader that different notations exist in the literature for the spaces above. For example in \cite{LiMa68}, $\tilde{\mH}^{1/2}(\gamma)$ and its dual space $\tilde{\mH}^{-1/2}(\gamma)$ are respectively denoted $\mH_{00}^{1/2}(\gamma)$ and $\mH^{-1/2}(\gamma)$. 

Denote by $T$ the Dirichlet-to-Neumann operator such that
\begin{equation}\label{DTN}
\begin{array}{rccl}
T: &\mH^{1/2}(\Gamma_+) &\rightarrow &  \mH^{-1/2}(\Gamma_+) \\ 
&\varphi  &\mapsto&   T\varphi=\partial_\nu u_{\varphi},
\end{array}
\end{equation}
where $u_{\varphi}$ is the unique element of $\mH^1(\Omega)$ satisfying
\be\label{uh}
\begin{array}{|rccl}
		-\Delta u_{\varphi}+u_\varphi &=& 0&  \mbox{in } \Omega\\[2pt]
		\partial_\nu u_{\varphi} &=& 0  & \mbox{on } \partial\Om\setminus\overline{\Gamma_+}\\[2pt]
				u_{\varphi} &=&\varphi & \mbox{on } \Gamma_+.
	\end{array}
	\ee
Classically one shows that $T$ is continuous from 
$\mH^{1/2}(\Gamma_+)$ to $\mH^{-1/2}(\Gamma_+)$. Additionally, set $g:=-\partial_\nu U \in \mH^{-1/2}(\Gamma_+)$ where $U$ is the unique function of $\mH^1(\Omega)$ satisfying
\be\label{U}
\begin{array}{|rccl}
-\Delta U+U &=& f&  \mbox{in } \Omega\\[2pt]
\partial_\nu U &=& 0  & \mbox{on } \partial\Om\setminus\overline{\Gamma_+}\\[2pt]
U &=&0 & \mbox{on } \Gamma_+.
\end{array}
\ee

\begin{proposition}\label{equivalence}
Assume that $s=\pm 1$ and $\alpha \in [0,1]$. If the function $u \in \mV_\alpha(\Om)$ satisfies (\ref{faible}) then $\varphi:=u|_{\Gamma_+}$ solves the problem
\be\label{1d}
\begin{array}{|l}
\mbox{Find }\varphi\in \mX_\alpha(I)\mbox{ such that for all }\psi\in \mX_\alpha(I)\\[6pt]
s\int_{\Gamma_+} x^\alpha d_x\varphi\,d_x \overline{\psi}\,dx + \langle T\varphi,\overline{\psi}\rangle_{\mH^{-1/2}(\Gamma_+),\mH^{1/2}(\Gamma_+)} = \langle  g,\overline{\psi} \rangle_{\mH^{-1/2}(\Gamma_+),\mH^{1/2}(\Gamma_+)}.
\end{array}
\ee
Conversely, if $\varphi\in \mX_\alpha(I)$ satisfies (\ref{1d}), then $u:=u_{\varphi}+U \in \mV_\alpha(\Om)$ solves (\ref{faible}). 
\end{proposition}
\begin{proof}
Assume that $u \in \mV_\alpha(\Om)$ solves (\ref{faible}). Then $\varphi:=u|_{\Gamma_+}$ belongs to $\mX_{\alpha}(I)$ and for all $v \in \mV_\alpha(\Om)$, we have
\begin{equation}\label{faiblebis}
\int_\Omega \nabla u\cdot \nabla \overline{v}+u\overline{v}\,dxdy + s\int_{\Gamma_+}x^\alpha\, \partial_x u\, \partial_x \overline{v}\,dx=\int_\Omega f\,\overline{v}\,dxdy.\end{equation}
The function $U\in\mH^1(\Omega)$ defined as the solution of (\ref{U}) satisfies, for all $v \in \mV_\alpha(\Om) \subset \mH^1(\Omega)$,
\[\int_\Omega \nabla U \cdot \nabla \overline{v}+U\overline{v}\,dxdy-\langle \partial_\nu U,\overline{v} \rangle_{\mH^{-1/2}(\Gamma_+),\mH^{1/2}(\Gamma_+)}=\int_\Omega f\,\overline{v}\,dxdy.\]
By taking the difference of these two identities, we obtain, for all $v \in \mV_\alpha(\Om)$,
\be\label{eq1}
\int_\Omega \nabla (u-U)\cdot \nabla \overline{v}+(u-U)\overline{v}\,dxdy + s\int_{\Gamma_+}x^\alpha\, \partial_x u\, \partial_x \overline{v}\,dx=\langle g,\overline{v} \rangle_{\mH^{-1/2}(\Gamma_+),\mH^{1/2}(\Gamma_+)}.
\ee
On the other hand, by choosing $v\in \mathscr{C}_0^\infty(\Omega)$ in (\ref{faiblebis}), we get 
\begin{equation}\label{EqnVol}
-\Delta u+u=f\mbox{ in }\Omega.
\end{equation}
Now consider some $\phi\in \tilde{\mH}^{1/2}(\partial\Om\setminus\overline{\Gamma_+})$, extend it by zero to $\Gamma_+$ and introduce a function $v\in\mH^1(\Om)$ such that $v|_{\partial\Om}=\phi$. Then clearly $v$ belongs to $\mV_\alpha(\Om)$. Multiplying (\ref{EqnVol}) by this $v$, integrating by parts and taking the difference with (\ref{faiblebis}), we find $\langle \partial_\nu u,\overline{\phi} \rangle_{\mH^{-1/2}(\partial\Om),\mH^{1/2}(\partial\Om)}=0$. Since this is valid for all $\phi\in \tilde{\mH}^{1/2}(\partial\Om\setminus\overline{\Gamma_+})$, we obtain $\partial_\nu u=0$ on 
$\partial\Om\setminus\overline{\Gamma_+}$. As a result, we get that
the function $u-U \in \mH^1(\Omega)$ solves (\ref{uh}) with $\varphi=u|_{\Gamma_+}$. This implies that for all $v \in \mV_\alpha(\Om)$, there holds
\be\int_\Omega \nabla (u-U) \cdot \nabla \overline{v}+(u-U)\overline{v}\,dxdy= \langle T u,\overline{v} \rangle_{\mH^{-1/2}(\Gamma_+),\mH^{1/2}(\Gamma_+)}.\label{eq2}\ee
Gathering (\ref{eq1}) and (\ref{eq2}), we deduce that for all $v \in \mV_\alpha(\Om)$, we have
\begin{equation}\label{IdenVaria}
s\int_{\Gamma_+}x^\alpha\, d_x \varphi\, \partial_x \overline{v}\,dx + \langle T\varphi,\overline{v} \rangle_{\mH^{-1/2}(\Gamma_+),\mH^{1/2}(\Gamma_+)} =\langle g,\overline{v} \rangle_{\mH^{-1/2}(\Gamma_+),\mH^{1/2}(\Gamma_+)}.
\end{equation}
Next, pick some $\phi \in\mathscr{C}^\infty(\overline{\Gamma_+})$, extend it to $\partial\Om\setminus\overline{\Gamma_+}$ to create an element $\tilde{\phi}\in\mH^{1/2}(\partial\Om)$ and consider some function $v\in\mH^1(\Om)$ such that $v|_{\partial\Om}=\tilde{\phi}$. Obviously such a $v$ belongs to $\mV_\alpha(\Om)$. Inserting it in (\ref{IdenVaria}), we obtain
\[
s\int_{\Gamma_+}x^\alpha\, d_x\varphi\, d_x \overline{\phi}\,dx + \langle T \varphi,\overline{\phi} \rangle_{\mH^{-1/2}(\Gamma_+),\mH^{1/2}(\Gamma_+)} =\langle g,\overline{\phi} \rangle_{\mH^{-1/2}(\Gamma_+),\mH^{1/2}(\Gamma_+)}.
\]
Since this is true for all $\phi \in\mathscr{C}^\infty(\overline{\Gamma_+})$, using the density of $\mathscr{C}^\infty(\overline{\Gamma_+})$ in $\mX_\alpha(I)$ which is established in Proposition \ref{density}, we conclude that $\varphi=u|_{\Gamma_+}$ solves (\ref{1d}). \\
\newline
Now let us show the second part of the statement. Assume that $\varphi\in \mX_\alpha(I)$ satisfies (\ref{1d}). Denote respectively by $u_{\varphi}$, $U$ the solutions of (\ref{uh}), (\ref{U}). For all $v \in \mV_\alpha(\Om)$, $v|_{\Gamma_+}$ is an element of $\mX_\alpha(I)$. As a consequence, for all $v \in \mV_\alpha(\Om)$, we have
\[
s\int_{\Gamma_+}x^\alpha\, \partial_x u_{\varphi}\, \partial_x \overline{v}\,dx + \langle \partial_\nu u_{\varphi},\overline{v} \rangle_{\mH^{-1/2}(\Gamma_+),\mH^{1/2}(\Gamma_+)} =-\langle \partial_\nu U,\overline{v} \rangle_{\mH^{-1/2}(\Gamma_+),\mH^{1/2}(\Gamma_+)}.
\]
Using additionally that $U=0$ on $\Gamma_+$ implies $\partial_xU=0$ on $\Gamma_+$, we get for all $v \in \mV_\alpha(\Om)$, 
\[
s\int_{\Gamma_+}x^\alpha\, \partial_x (u_{\varphi}+U)\, \partial_x \overline{v}\,dx + \langle \partial_\nu (u_{\varphi}+U),\overline{v} \rangle_{\mH^{-1/2}(\partial \Omega),\mH^{1/2}(\partial \Omega)} =0.
\]
Finally, since 
\[\langle \partial_\nu (u_{\varphi}+U),\overline{v} \rangle_{\mH^{-1/2}(\partial \Omega),\mH^{1/2}(\partial \Omega)}= \int_\Omega \nabla (u_{\varphi}+U)\cdot \nabla \overline{v}+(u_\varphi+U)\overline{v}\,dxdy -\int_\Omega f\,\overline{v}\,dxdy,
\]
we deduce that $u=u_{\varphi}+U \in \mV_\alpha(\Om)$ satisfies Problem (\ref{faible}).
 \end{proof}	 
We are now in a position to show Theorem \ref{casgeneral}.
\begin{proof}[Proof of Theorem \ref{casgeneral}]
With the Riesz representation theorem, define the continuous operators $S,\,K: \mX_\alpha(I) \rightarrow \mX_\alpha(I)$ such that for all $\varphi,\psi \in \mX_\alpha(I)$
\[
(S\varphi,\psi)_{\mX_\alpha(I)}=\int_{\Gamma_+} x^\alpha d_x\varphi\,d_x \overline{\psi}+\varphi\overline{\psi}\,dx
\]
\[
(K\varphi,\psi)_{\mX_\alpha(I)}=\langle T\varphi,\overline{\psi}\rangle_{\mH^{-1/2}(\Gamma_+),\mH^{1/2}(\Gamma_+)}-\int_{\Gamma_+} \varphi\overline{\psi}.
\]
Clearly $S$ is an isomorphism. On the other hand, in Proposition \ref{prop-compacite} we showed that the embedding of $\mX_\alpha(I)$ in $\mH^{1/2}(\Gamma_+)$ is compact when $\alpha \in [0,1)$. Using this result and the fact that the operator $T: \mH^{1/2}(\Gamma_+)\rightarrow\mH^{-1/2}(\Gamma_+)$ defined in (\ref{DTN}) is continuous, one proves that $K: \mX_\alpha(I) \rightarrow \mX_\alpha(I)$ is compact. This guarantees that $S+K: \mX_\alpha(I) \rightarrow \mX_\alpha(I)$ is Fredholm of index zero. In particular, if only the null function solves (\ref{1d}) with $g\equiv0$, then it admits a solution. The equivalence between
problems (\ref{faible}) and (\ref{1d}) given by Proposition \ref{equivalence} completes the proof. 
\end{proof}

\section{The vanishing case: a negative result for $s=-1$ and $\alpha=1$}\label{SectionWithSingus}

In the previous section, we proved that the problem (\ref{faible}) in the bad sign case $s=-1$ is well-posed in the Fredholm sense for all $\alpha\in[0,1)$. Here we still consider (\ref{faible}) with $s=-1$ but now with $\alpha=1$. In that situation, we show Theorem \ref{cas1}, namely that the operator $A: \mV_1(\Om) \rightarrow \mV_1(\Om)$ associated with (\ref{faible}) and defined via (\ref{DefOpA}) is not of Fredholm type.\\
\newline
For $s=-1$ and $\alpha=1$, the variational formulation of our problem writes:
find $u \in \mV_1(\Om)$ such that for all $v \in \mV_1(\Om)$,
	\be	\label{faible_1}
	\int_\Omega \nabla u\cdot \nabla \overline{v}+u\overline{v}\,dxdy -\int_{\Gamma_+}x\, \partial_x u\, \partial_x \overline{v}\,dx=\int_\Omega f\overline{v}\,dxdy.
	\ee		
If $u$ solves (\ref{faible_1}), then $u$ satisfies 
the strong problem:
\be \label{fort_1}
\begin{array}{|rccl}
-\Delta u +u&=& f&  \mbox{in } \Omega\\[3pt]
\partial_\nu u&=&0 & \mbox{on } \partial\Om\setminus\overline{\Gamma_+}\\[3pt]
\partial_\nu u &=& -\partial_x (x\, \partial_x u) & \mbox{on } \Gamma_+ .
\end{array}	
\ee 
We will see that the Fredholm property for $A$ is lost due to the existence of strongly oscillating singularities at the origin. We start by computing these singularities. The latter are defined as the functions which solve the problem corresponding to the principal part of (\ref{fort_1}) in a neighbourhood of point $O:=(0,0)$ with $f\equiv0$ and which are of the form $\mathfrak{s}(x,y)=r^\lambda\varphi(\theta)$, where $\lambda\in\Cplx$ and $\varphi$ are to be determined. Here $(r,\theta)$ are the polar coordinates  centered at $O$. From the equation $\Delta \mathfrak{s}=0$, we find first $d^2_{\theta\theta}\varphi(\theta)+\lambda^2\varphi(\theta)=0$ in $(0,\pi)$. Then the boundary conditions of (\ref{fort_1}) on $\Gamma_-$ and $\Gamma_+$ lead to the relations
\[
d_{\theta}\varphi(\pi)=0 \qquad\mbox{ and }\qquad d_\theta\varphi(0)=\lambda^2\varphi(0),
\]
respectively.
Thus, we get
\be\label{system} 
\begin{array}{|rccl}
		d^2_{\theta\theta}\varphi(\theta)+\lambda^2\varphi(\theta) &=& 0&  \mbox{in } (0,\pi)\\
		d_\theta \varphi(\pi) &=& 0& \\
		d_\theta\varphi(0)&=& \lambda^2\varphi(0). & 
	\end{array}
\ee
Exploiting lines 1 and 2 of (\ref{system}), we find that $\varphi$ must be of the form $\varphi(\theta)=C\cos(\lambda(\theta-\pi))$ where $C$ is a constant. From the third line of (\ref{system}), we deduce that this system admits a non zero solution if and only if $\lambda$ solves 
\begin{equation}\label{RelationDispersion}
\sin(\lambda\pi)=\lambda \cos(\lambda\pi).
\end{equation}
Of particular importance are the singular exponents $\lambda \neq 0$ which are non zero and purely imaginary, i.e. of the form $\lambda=i\tau$ with $\tau\in\mathbb{R}\setminus\{0\}$. Indeed, in this case, we have 
\begin{equation}\label{defBHsingu}
\mathfrak{s}(x,y)=r^{i\tau}\varphi(\theta)=e^{i\tau\ln r}\varphi(\theta),
\end{equation}
which corresponds to a function which oscillates an infinite number of times without decay when $r$ tends to zero (see a rough approximation, because the mesh is fixed, in Figure \ref{Fig_sm1_alpha_1}). Such singularities are just outside of $\mV_1(\Om)$ (and of $\mH^1(\Om)$) in the sense that they are not in $\mV_1(\Om)$ while $r^{\eps}\mathfrak{s}$ belongs to $\mV_1(\Om)$ for all $\eps>0$. In general such singularities are absent and appear only for particular problems. Among them, let us mention certain problems with sign-changing dielectric constants in presence of inclusions of metals or metamaterials with non smooth shapes in electromagnetism (see \cite{BoCCSub,BCCC16,BoCR23}). In this context, the strongly oscillating singularities are sometimes called black hole waves. They are met also in the study of the Laplace operator with singular Robin boundary conditions \cite{NaPo18} or in singular geometries with cusps \cite{NaPT18}.\\
\newline
For our problem, one finds that $\lambda=i\tau$, $\tau\in\mathbb{R}\setminus\{0\}$, solves (\ref{RelationDispersion}) if  and only if 
\begin{equation}\label{EqnOscillating}
\tanh(\tau\pi)=\tau.
\end{equation}
This equation has exactly two solutions which are opposite sign. The idea is to construct a Weyl sequence from these particular singularities in order to prove that the problem (\ref{faible_1}) is not of Fredhom type.
\begin{remark}
If the condition on $\Gamma_+$ was $\partial_\nu u=+\partial_x(x\partial_x u)$ (with the good sign) and not $\partial_\nu u=-\partial_x(x\partial_x u)$ as in (\ref{fort_1}), we would find that the equation characterizing the singular exponents is
\[
\sin(\lambda\pi)=-\lambda \cos(\lambda\pi).
\]
Observe that there is no solution of this equation in $\mathbb{R} i\setminus\{0\}$. This is coherent with the fact that the corresponding problem (\ref{faible}) is well-posed for $s=1$ and $\alpha=1$ (Theorem \ref{good_sign}).
\end{remark}

In our analysis below, we need the two following density results.
\begin{lemma}\label{infini}
For $\alpha \in [0,1]$, the space $\mathscr{C}^\infty(\overline{\Omega})$ is dense in $\mV_\alpha(\Om)$.
\end{lemma}
\begin{proof}
We work in three steps. First, we show the result for $\alpha=0$. Then we prove that $\mV_0(\Om)$ is dense in $\mV_\alpha(\Om)$. Finally, we  conclude.\\
\newline
Let us establish the density of $\mathscr{C}^\infty(\overline{\Omega})$ in
$\mV_0(\Om)=\{v \in \mH^1(\Omega),\,v|_{\Gamma_+} \in \mH^1(\Gamma_+)\}$.
We take $u \in \mV_0(\Om)$ and denote $\varphi:=u|_{\Gamma_+} \in \mH^1(\Gamma_+)$. Introduce $\tilde{\varphi}\in\mH^1(\partial \Omega)$ an extension of $\varphi$. There exists $U\in\mH^{3/2}(\Omega)$ such that $U|_{\partial\Om}=\tilde{\varphi}$. Now observe that the function $\hat{u}:=u-U \in \mH^1(\Omega)$ is such that $\hat{u}|_{\Gamma_+}=0$. As a consequence, it can be approximated in $\mH^1(\Om)$ by a sequence $(\hat{u}_n)$ of elements of $\mathscr{C}_0^\infty(\overline{\Omega}\setminus\Gamma_+)$. Introduce also $(U_n)$ a sequence of functions of $\mathscr{C}^\infty(\overline{\Omega})$ which converges to $U$ in $\mH^{3/2}(\Omega)$. Then $(\hat{u}_n+U_n)$ is a sequence in $\mathscr{C}^\infty(\overline{\Omega})$ which converges to $u$ in $\mH^1(\Omega)$ and such that $(\hat{u}_n+U_n)|_{\Gamma_+}$ converges to $\varphi$ in $\mH^1(\Gamma_+)$. This proves that $(\hat{u}_n+U_n)$ converges to $u$ in $\mV_0(\Om)$ and completes the first step.\\
\newline
Let us now prove the density of $\mV_0(\Om)$ in $\mV_\alpha(\Om)$ for $\alpha \in [0,1]$, taking inspiration from \cite{luneville_mercier}.
Consider $u \in \mV_\alpha(\Om)$ and denote $\varphi:=u|_{\Gamma_+} \in \mX_\alpha(I)$.
We introduce the function $\varphi_n$ given by (\ref{hn}), which satisfies $\varphi_n \rightarrow \varphi$ in $\mX_\alpha(I)$.
From Proposition \ref{prop-H12}, since $\alpha \in [0,1]$, we have that $\varphi_n \rightarrow \varphi$ in $\mH^{1/2}(\Gamma_+)$.
Introduce $u_n \in \mH^1(\Omega)$ the function satisfying $u_n|_{\Gamma_+}=\varphi_n$ and 
\[(u_n,v)_{\mH^1(\Omega)}=(u,v)_{\mH^1(\Omega)},\qquad \forall v \in \mH^1(\Omega),\,v|_{\Gamma_+}=0.\]
We have $u_n \in \mV_0(\Om)$. Additionally, $e_n:=u-u_n$ is such that $e_n \in \mH^1(\Omega)$, $e_n|_{\Gamma_+}=\varphi-\varphi_n$ and
\[
(e_n,v)_{\mH^1(\Omega)}=0,\qquad \forall v \in \mH^1(\Omega),\,v|_{\Gamma_+}=0.
\]
Hence there exists a constant $C>0$ independent of $n$ such that
\[\|e_n\|_{\mH^1(\Omega)} \leq C\, \|\varphi-\varphi_n\|_{\mH^{1/2}(\Gamma_+)}.\]
This guarantees that $(u_n)$ converges to $u$ in $\mV_0(\Om)$, which completes the second step.

To conclude, it suffices to use the two previous steps and to remark that the convergence in $\mV_0(\Om)$ implies the convergence in $\mV_\alpha(\Om)$.
\end{proof}	

\begin{lemma}\label{infini0}
The space $\mathscr{C}^\infty_0(\overline{\Om}\setminus\{O\})$ is dense in $\mV_1(\Om)$. 
\end{lemma}
\begin{proof}
We first show that any $w$ of $\mathscr{C}^\infty(\overline{\Om})$ can be approximated by a sequence of functions of $\mV^{\star}_1(\Om):=\{w\in \mV_1(\Om),\,\,w=0\mbox{ in a neighbourhood of }O\}$. To proceed, we work as in the proof of Lemma 1.2.2 of \cite{Claeys08}. For $\eps>0$, define the function $\tau_\eps$ such that 
\[
\tau_\eps(r)=\begin{array}{|ccl}
\displaystyle 0 & {\rm for} & r \leq \eps \\
\displaystyle 1-\frac{\log(2r)}{\log(2\eps)} & {\rm for} & \eps < r <1/2 \\
\displaystyle 1 & {\rm for} & r \geq 1/2.
\end{array}
\]

We have $\tau_\eps w \in \mV^{\star}_1(\Om)$ and 
\[
\begin{aligned}
\|w-\tau_\eps w\|_{\mV_1(\Om)} &\leq  \|1-\tau_\eps\|_{\mL^2(\Omega)} \|\nabla w\|_{\mL^\infty(\Omega)^2}+\|\nabla\tau_\eps\|_{\mL^2(\Omega)^2} \| w\|_{\mL^\infty(\Om)} \\[3pt]
& +\|\sqrt{x}(1-\tau_\eps)\|_{\mL^2(\Gamma_+)} \|\partial_x w\|_{\mL^\infty(\Gamma_+)}+\|\sqrt{x}\partial_x\tau_\eps\|_{\mL^2(\Gamma_+)} \| w\|_{\mL^\infty(\Gamma_+)} .
\end{aligned}
\]
For $\eps$ small enough, there holds
\[\|1-\tau_\eps\|^2_{\mL^2(\Omega)}=\pi \int_0^{1/2}(1-\tau_\eps)^2r\,dr=\pi \left(\int_0^\eps r\,dr + \int_\eps^{1/2}\frac{(\log(2r))^2}{(\log(2\eps))^2}r\,dr \right) \leq \frac{C}{|\log \eps|^2},\]
where $C$ is independent of $\eps$.
Similar 
explicit computations show that for $\eps$ small enough, we have
\[
\|1-\tau_\eps\|_{\mL^2(\Omega)}+\|\sqrt{x}\,(1-\tau_\eps)\|_{\mL^2(\Gamma_+)} \leq \frac{C}{|\log\eps|},\quad \|\nabla\tau_\eps\|_{\mL^2(\Omega)^2}+\|\sqrt{x}\,\partial_x\tau_\eps\|_{\mL^2(\Gamma_+)} \leq \frac{C}{\sqrt{|\log \eps|}},
\]
where $C$ is independent of $\eps$. This ensures that $(\tau_\eps w)$ converges to $w$ in $\mV_1(\Om)$ as $\eps$ goes to zero.\\
\newline
To conclude, take some $v\in \mV_1(\Om)$ and pick $\eps>0$. From Lemma \ref{infini}, there is $\varphi \in \mathscr{C}^\infty(\overline{\Omega})$ such that $\|v-\varphi\|_{\mV_1(\Om)} \leq \eps/3$. According to the first part of the proof, there exists some $v^\star \in \mV^{\star}_1(\Om)$ such that $\|\varphi-v^\star\|_{\mV_1(\Om)} \leq \eps/3$.
 On the other hand, it is clear that any function of $\mV^{\star}_1(\Om)$ can be approximated in $\mV_1(\Om)$ by a sequence of functions of $\mathscr{C}^\infty_0(\overline{\Omega}\setminus\{O\})$ (adapt the first step of the proof of Lemma \ref{infini}). 
Hence there exists $\psi \in \mathscr{C}^\infty_0(\overline{\Omega}\setminus\{O\})$ such that $\|v^\star-\psi\|_{\mV_1(\Om)} \leq \eps/3$.
Finally, we can write
\[\|v-\psi\|_{\mV_1(\Om)} \leq \|v-\varphi\|_{\mV_1(\Om)} + \|\varphi-v^\star\|_{\mV_1(\Om)} + \|v^\star-\psi\|_{\mV_1(\Om)} \leq \eps,\] 
which completes the proof
\end{proof}
We are now in a position to establish Theorem \ref{cas1}.
\begin{proof}[Proof of Theorem \ref{cas1}]
Let us work by contradiction. Assume that $A: \mV_1(\Om) \rightarrow \mV_1(\Om)$ is of Fredholm type. In this case, since it is selfadjoint, because it is symmetric and bounded, it is Fredholm of index zero. Then, define the operator $\tilde{A}: \mV_1(\Om) \rightarrow \mV_1(\Om)$ such that 
\[
(\tilde{A}u,v)_{\mV_1(\Om)}=a(u,v)+i\int_{\Omega} u\,\overline{v}\,dxdy,\qquad \forall u,v\in \mV_1(\Om),
\]
Using that $\mH^1(\Om)$ is compactly embedded in $\mL^2(\Om)$, one can shows that $\tilde{A}-A$ is compact. This guarantees that $\tilde{A}$ is Fredholm of index zero. Since $\tilde{A}$ is injective, we deduce that $\tilde{A}$ is an isomorphism. Therefore, there is a constant $C>0$ such that we have
\begin{equation}\label{ContinuityInverse}
\|u\|_{\mV_1(\Om)} \le C\,\|\tilde{A}u\|_{\mV_1(\Om)},\qquad\forall u\in \mV_1(\Om).
\end{equation}
Let us show that this is not true. For $n\in \N\setminus \{0\}$, define the function $\mathfrak{s}_n$ such that
\begin{equation}\label{defsn}
\mathfrak{s}_n(x,y)=r^{i\tau+1/n}\cos(i\tau(\theta-\pi))
\end{equation}
where $\tau$ stands for the positive root of Equation (\ref{EqnOscillating}). Due to the regularizing term $r^{1/n}$ in the definition of $\mathfrak{s}_n$, one can check that $\mathfrak{s}_n$ belongs to $\mV_1(\Om)$ for all $n\in\mathbb{N} \setminus \{0\}$. However we have 
\begin{equation}\label{calculus1}
\begin{aligned}
\|\mathfrak{s}_n\|^2_{\mV_1(\Om)} &\geq \int_{\Gamma_+} x|\partial_x \mathfrak{s}_n|^2\,dx \\
& = |\cos(i\tau\pi)|^2\,\left|i\tau+1/n\right|^2\,\int_0^{1} x^{2/n-1}\,dx \geq |\tau\cos(i\tau\pi)|^2\,\frac{n}{2}\underset{n \rightarrow +\infty}{\longrightarrow}+\infty.
\end{aligned}
\end{equation}
On the other hand, using Lemma \ref{infini0}, we can write
\begin{equation}\label{defNorm}
\|\tilde{A}\mathfrak{s}_n\|_{\mV_1(\Om)}=\underset{v\in \mathscr{C}^\infty_0(\overline{\Omega}\setminus\{O\})\setminus\{0\}}{\sup}\frac{|(\tilde{A}\mathfrak{s}_n,v)_{\mV_1(\Om)}|}{\|v\|_{\mV_1(\Om)}}.
\end{equation}
Let us work on the right hand side of (\ref{defNorm}). For $v\in \mathscr{C}^\infty_0(\overline{\Omega}\setminus\{O\})$, we have
\begin{equation}\label{Decompo}
\begin{array}{rcl}
(\tilde{A}\mathfrak{s}_n,v)_{\mV_1(\Om)}&\hspace{-0.2cm}=&\hspace{-0.2cm}\dsp\int_{\Omega} \nabla \mathfrak{s}_n \cdot \nabla \overline{v}\,dxdy-\int_{\Gamma_+}x\,\partial_x \mathfrak{s}_n\,\partial_x \overline{v}\,dx + (1+i) \int_{\Omega} \mathfrak{s}_n\overline{v}\,dxdy \\
&\hspace{-0.2cm}= &\hspace{-0.2cm}\dsp-\int_{\Omega} \Delta \mathfrak{s}_n \overline{v}\,dxdy +\int_{\partial\Om} \partial_\nu \mathfrak{s}_n \overline{v}\,dx + 
\int_{\Gamma_+} \partial_x(x\,\partial_x \mathfrak{s}_n)\overline{v}\,dx \\[6pt]
 & & -\partial_x \mathfrak{s}_n(P)\,\overline{v}(P)+ (1+i) \int_{\Omega} \mathfrak{s}_n\overline{v}\,dxdy 
\end{array}
\end{equation}
where $P:=(1,0)$. Let $\delta\in(0,1)$ be sufficiently small so that $\Om$ contains the half disk $D^+(O,\delta):=\{(x,y)\,|\,0<(x^2+y^2)^{1/2}<\delta,\,y>0\}$. Introduce some cut-off function $\zeta\in\mathscr{C}^{\infty}(\overline{\Om})$ which depends only on the radial coordinate such that $\zeta=1$ for $r\le\delta/2$ and $\zeta=0$ for $r\ge\delta$. Let us rewrite (\ref{Decompo}) as $(\tilde{A}\mathfrak{s}_n,v)_{\mV_1(\Om)}=I_1^n+I_2^n+I_3^n$ with 
\begin{equation}\label{DefTerms}
\begin{array}{rcl}
I_1^n &:= & -\int_{\Omega} (1-\zeta)\Delta \mathfrak{s}_n \overline{v}\,dxdy+\int_{\partial\Om} (1-\zeta)\partial_\nu \mathfrak{s}_n \overline{v}\,dx +\int_{\Gamma_+} (1-\zeta)\partial_x(x\,\partial_x \mathfrak{s}_n)\overline{v}\,dx\\[6pt]
& & -\partial_x \mathfrak{s}_n|_{\Gamma_+}(P)\,\overline{v}(P)+ (1+i) \int_{\Omega} \mathfrak{s}_n\overline{v}\,dxdy \\[12pt]
I_2^n &:=& \int_{\Gamma_+}(\partial_\nu \mathfrak{s}_n + \partial_x(x\,\partial_x \mathfrak{s}_n))\,\zeta\,\overline{v}\,dx\\[12pt]
I_3^n &:=& \dsp-\int_{\Omega} \Delta \mathfrak{s}_n (\zeta \overline{v})\,dxdy.
\end{array}
\end{equation}
Above we used that $\partial_\nu \mathfrak{s}_n=0$ on $\Gamma_-$ according to the definition (\ref{defsn}) of $\mathfrak{s}_n$. Let us study each of the terms of (\ref{DefTerms}) separately.\\
\newline
We start with $I_1^n$. Since $(\mathfrak{s}_n)$ is bounded in $\mL^{\infty}(\Omega)$, and so in $\mL^{2}(\Omega)$, and that its derivatives are also bounded in $\mL^{\infty}(\Omega\setminus D^+(O,\delta/2))$, we have
\begin{equation}\label{estim1}
|I_1^n| \le c\,\|v\|_{\mV_1(\Om)}
\end{equation}
where $c>0$ is a constant which may change from one line to another below but which remains independent of $n$. Now we deal with the term $I_2^n$ in (\ref{DefTerms}). On $\Gamma_+$, we have 
\[
\partial_\nu \mathfrak{s}_n + \partial_x(x\partial_x\mathfrak{s}_n)= \left(-i\tau\sin(i\tau\pi) + (i\tau+1/n)^2\cos(i\tau\pi)\right)x^{i\tau+1/n-1},
\]
which, using the fact that $\lambda=i\tau$ satisfies the relation (\ref{RelationDispersion}), yields
\[
\partial_\nu \mathfrak{s}_n + \partial_x(x\partial_x\mathfrak{s}_n)=\left(\frac{2i\tau}{n}+\frac{1}{n^2}\right)\cos(i\tau\pi)x^{i\tau+1/n-1}.
\]
Integrating by parts, we find
\[
\int_{\Gamma_+}(\partial_\nu \mathfrak{s}_n + \partial_x(x\,\partial_x \mathfrak{s}_n))\,\zeta\overline{v}\,dx=-\int_{0}^{1}\int_{0}^x(\partial_\nu \mathfrak{s}_n + \partial_t(t\,\partial_t\mathfrak{s}_n))\,dt\,\partial_x(\zeta\overline{v})\,dx.
\]
Writing
\[
\int_{0}^x(\partial_\nu \mathfrak{s}_n + \partial_t(x\,\partial_t\mathfrak{s}_n))\,dt\,\partial_x(\zeta\overline{v})=\frac{1}{i\tau+\frac{1}{n}}\left(\frac{2i\tau}{n}+\frac{1}{n^2}\right)\cos(i\tau\pi)x^{i\tau+1/n-1/2}\,x^{1/2}\partial_x(\zeta\overline{v})
\]
and using the Cauchy-Schwarz inequality,
we obtain
\[
\left|\int_{\Gamma_+}(\partial_\nu \mathfrak{s}_n + \partial_x(x\,\partial_x \mathfrak{s}_n))\,\zeta\overline{v}\,dx\right| \leq 
\left|\frac{1}{i\tau+\frac{1}{n}}\right| \left|\frac{2i\tau}{n}+\frac{1}{n^2} \right| |\cos(i\tau\pi)| \sqrt{\int_0^1 x^{2/n-1}\,dx}\sqrt{\int_0^1 x|\partial_x(\zeta\, v)|^2\,dx} ,
\]
which is enough to conclude that
\begin{equation}\label{estim2}
|I_2^n|=\left|\int_{\Gamma_+}(\partial_\nu \mathfrak{s}_n + \partial_x(x\,\partial_x\mathfrak{s}_n))\,\zeta\overline{v}\,dx\right| \le c\,\|v\|_{\mV_1(\Om)}.
\end{equation}
Finally we work on the term $I_3^n$ in (\ref{DefTerms}). Using that
\[\mathfrak{s}_n(x,y)=r^{1/n} \mathfrak{s}(x,y),\qquad \mathfrak{s}(x,y)=r^{i\tau}\cos(i\tau(\theta-\pi))\]
and the fact that $\Delta \mathfrak{s}=0$, we find
\[
\Delta \mathfrak{s}_n=\left(\frac{2i\tau}{n}+\frac{1}{n^2}\right)\cos(i\tau(\theta-\pi))r^{i\tau+1/n-2}.
\]
This allows us to write
\[
\begin{array}{rcl}
\dsp\int_{\Omega}\Delta \mathfrak{s}_n \,(\zeta\overline{v})\,dxdy&=&\dsp\int_{0}^\pi\int_{0}^{1} \left(\frac{2i\tau}{n}+\frac{1}{n^2}\right)\cos(i\tau(\theta-\pi))r^{i\tau+1/n-2} \,(\zeta\overline{v})\,rdrd\theta \\[10pt]
&=&-\dsp \left(\frac{2i\tau}{n}+\frac{1}{n^2}\right)\int_{0}^\pi \cos(i\tau(\theta-\pi)) \left(\int_{0}^{1} \frac{1}{i\tau+1/n}r^{i\tau+1/n} \,\partial_r(\zeta\overline{v})\,dr\right)d\theta \\
&=&-\dsp \frac{1}{i\tau+1/n} \left(\frac{2i\tau}{n}+\frac{1}{n^2}\right)\int_{0}^\pi\int_{0}^{1}r^{i\tau+1/n-1}\cos(i\tau(\theta-\pi)) \,\partial_r(\zeta\overline{v})\,rdrd\theta.
\end{array}
\]
Using again the Cauchy-Schwarz inequality in $\mL^2(D^+(O,1))$, we deduce
\begin{equation}\label{estim3}
|I_3^n|=\left|\dsp\int_{\Omega}\Delta \mathfrak{s}_n \,(\zeta\overline{v})\,dxdy\right| \leq c\,\|v\|_{\mV_1(\Om)}.
\end{equation}
Finally, gathering estimates  (\ref{estim1}), (\ref{estim2}), (\ref{estim3}) into (\ref{Decompo}), from (\ref{defNorm}), we obtain 
\begin{equation}\label{Prop2}
\|\tilde{A}\mathfrak{s}_n\|_{\mV_1(\Om)} \leq c.
\end{equation}
Properties (\ref{calculus1}) and (\ref{Prop2}) are in contradiction with (\ref{ContinuityInverse}). This ends the proof.
\end{proof}

\section{Problem with a vanishing and sign-changing impedance}\label{SectionVanishing}

The goal of the present section is to study the problem (\ref{faible bis}) which involves an impedance which is both vanishing at zero and whose sign is not constant. In Section \ref{SectionAlphaMoins1}, we showed that the boundary term in the problem (\ref{faible}), with an impedance which vanishes at zero but whose sign is constant, plays the role of the principal part and the volumic component is only a compact perturbation of it when $\alpha\in[0,1)$. Therefore, by considering here a sign-changing impedance, we strongly affect the principal part of the operator. For this reason, the question of the Fredholmness of $B$ compared to the one of $A$ is a priori not clear. We will see however that we have similar results for $A$ and $B$ when $\alpha\in[0,1]$.

\subsection{Proof of Theorem \ref{casgeneral_bis} -- case $\alpha\in[0,1)$}
In order to prove Theorem \ref{casgeneral_bis}, we will use, like in Proposition \ref{equivalence}, an equivalence between the formulation (\ref{faible bis}) in 2D and a variational problem in 1D. Set $I_+:=(0,1)$, $I_-:=(-1,0)$, $\mathbb{I}:=(-1,1)$ and define the two spaces 
\[
\mX_\alpha(I_\pm):=\{\phi_\pm \in \mL^2(I_\pm),\,|x|^{\alpha/2} d_x \phi_\pm \in \mL^2(I_\pm)\}.
\]	
We equip them with their natural inner products
\[(
\phi_\pm,\psi_\pm)_{\mX_\alpha(I_\pm)}=\int_{I_\pm} \phi_\pm\,\overline{\psi_\pm}\,dx + \int_{I_\pm}|x|^\alpha\, d_x \phi_\pm\, d_x \overline{\psi_\pm}\,dx.
\]
For $\alpha \in [0,1)$, we also define the space
\[
\mathcal{X}_\alpha(\mathbb{I}):=\{\Phi=(\phi_-,\phi_+) \in \mX_\alpha(I_-) \times \mX_\alpha(I_+),\, \phi_-(0)=\phi_+(0)\}.
\]
It is a Hilbert space when endowed with the inner product, for $\Phi=(\phi_-,\phi_+)$, $\Psi=(\psi_-,\psi_+)$,
\[
(\Phi,\Psi)_{\mathcal{X}_\alpha(\mathbb{I})}=(\phi_-,\psi_-)_{\mX_\alpha(I_-)} + (\phi_+,\psi_+)_{\mX_\alpha(I_+)}.
\]
Note that the definition of $\mathcal{X}_\alpha(\mathbb{I})$ relies on the result of Proposition \ref{prop-Holder} which guarantees that
$\mX_\alpha(I_\pm) \subset \mathscr{C}^0(\overline{I_\pm})$ as soon as $\alpha \in [0,1)$. Minor adaptations of the proof of Proposition \ref{density} allow one to show that $\mathscr{C}^\infty(\overline{\mathbb{I}})$ is dense in $\mathcal{X}_\alpha(\mathbb{I})$ for all $\alpha\in[0,1]$. 
\newline
To state the result of equivalence, we need to define a few objects. Recall that $\Gamma=(-1,1)\times\{0\}$. Denote by $\Theta$ the Dirichlet-to-Neumann operator such that
\begin{equation}\label{DTN_bis}
\begin{array}{rccl}
\Theta: &\mH^{1/2}(\Gamma) &\rightarrow &  \mH^{-1/2}(\Gamma) \\ 
&\Phi  &\mapsto&   \Theta\Phi=\partial_\nu u_{\Phi},
\end{array}
\end{equation}
where $u_{\Phi}$ is the unique element of $\mH^1(\Omega)$ satisfying
\be\label{uh_bis}
\begin{array}{|rccl}
		-\Delta u_{\Phi}+u_\Phi &=& 0&  \mbox{in } \Omega\\[2pt]
		\partial_\nu u_{\Phi} &=& 0  & \mbox{on } \partial\Om\setminus\overline{\Gamma}\\[2pt]
				u_{\Phi} &=&\Phi & \mbox{on } \Gamma.
	\end{array}
	\ee
Classically, one shows that $\Theta$ is continuous from 
$\mH^{1/2}(\Gamma)$ to $\mH^{-1/2}(\Gamma)$. Additionally, set $G:=-\partial_\nu \mathcal{U} \in \mH^{-1/2}(\Gamma)$ where $\mathcal{U}$ is the unique function of $\mH^1(\Omega)$ satisfying
\be\label{R}
\begin{array}{|rccl}
-\Delta \mathcal{U}+\mathcal{U} &=& f&  \mbox{in } \Omega\\[2pt]
\partial_\nu \mathcal{U} &=& 0  & \mbox{on } \partial\Om\setminus\overline{\Gamma}\\[2pt]
\mathcal{U} &=&0 & \mbox{on } \Gamma.
\end{array}
\ee

\begin{proposition}\label{equivalence_bis}
Assume that $\alpha \in [0,1)$. If the function $u \in \mW_\alpha(\Om)$ satisfies (\ref{faible bis}) then $\Phi=(\phi_-,\phi_+):=(u|_{\Gamma_-},u|_{\Gamma_+})$ solves the problem
\be \label{1d_bis}
\hspace{-0.3cm}\begin{array}{|l}
\mbox{Find }\Phi\in\mathcal{X}_\alpha(\mathbb{I})\mbox{ such that for all }\Psi=(\psi_-,\psi_+) \in\mathcal{X}_\alpha(\mathbb{I})\\[6pt]
-\int_{\Gamma_-}\hspace{-0.15cm} |x|^\alpha d_x \phi_-\,d_x \overline{\psi_-}\,dx+\int_{\Gamma_+} \hspace{-0.25cm}x^\alpha d_x \phi_+\,d_x \overline{\psi_+}\,dx+ \langle \Theta \Phi,\overline{\Psi}\rangle_{\mH^{-1/2}(\Gamma),\mH^{1/2}(\Gamma)} = \langle  G,\overline{\Psi} \rangle_{\mH^{-1/2}(\Gamma),\mH^{1/2}(\Gamma)}.
\end{array}\hspace{-0.9cm}
\ee
Conversely, if $\Phi\in\mathcal{X}_\alpha(\mathbb{I})$ satisfies (\ref{1d_bis}), then $u=u_\Phi+\mathcal{U} \in \mW_\alpha(\Om)$ solves (\ref{faible bis}). 
\end{proposition}
\begin{proof}
The proof is very similar to the one of Proposition \ref{equivalence}. However we detail it for the sake of completeness. Assume first that $u \in \mW_\alpha(\Om)$ solves (\ref{faible bis}). Then $\Phi=(\phi_-,\phi_+):=(u|_{\Gamma_-},u|_{\Gamma_+})$ belongs to $\mathcal{X}_\alpha(\mathbb{I})$. Indeed, then clearly there holds $(\phi_-,\phi_+) \in \mX_\alpha(I_-) \times \mX_\alpha(I_+)$. Additionally, from Proposition \ref{prop-Holder}, we know that $\phi_{\pm}\in\mathscr{C}(\overline{I_\pm})$. Since $u|_{\partial\Om}\in\mH^{1/2}(\partial\Om)$, we must have $\phi_-(0)=\phi_+(0)$ (as mentioned in \cite[Chap.\,33]{tartar}). This can be shown by using the characterization of the space $\mH^{1/2}(\partial\Om)$ with the help of the double integral on $\partial\Om$ 
(see for example \cite{LiMa68}). On the other hand, for all $v\in\mW_\alpha(\Om)$, there holds
\[
\int_\Omega \nabla u\cdot \nabla \overline{v}+u\overline{v}\,dxdy  -\int_{\Gamma_-}|x|^\alpha\, \partial_x u\, \partial_x \overline{v}\,dx+ \int_{\Gamma_+}x^\alpha\, \partial_x u\, \partial_x \overline{v}\,dx=\int_\Omega f\overline{v}\,dxdy.
\]
The function $\mathcal{U}\in\mH^1(\Om)$ defined as the solution of (\ref{R}), satisfies, for all $v\in\mW_\alpha(\Om)\subset\mH^1(\Om)$, 
\[
\int_\Omega \nabla \mathcal{U}\cdot \nabla \overline{v}+\mathcal{U}\overline{v}\,dxdy- \langle \partial_\nu \mathcal{U} ,\overline{v}\rangle_{\mH^{-1/2}(\Gamma),\mH^{1/2}(\Gamma)}=\int_\Omega f\,\overline{v}\,dxdy,\qquad \forall v \in \mW_\alpha(\Om).   \]
By taking the difference of the two above identities, we obtain, for all $ v \in \mW_\alpha(\Om)$, 
\begin{equation}\label{estim_bis_1}
\int_\Omega \nabla (u-\mathcal{U})\cdot \nabla \overline{v}+(u-\mathcal{U})\overline{v}\,dxdy - \int_{\Gamma_-}|x|^\alpha\, \partial_x u\, \partial_x \overline{v}\,dx+\int_{\Gamma_+}x^\alpha\, \partial_x u\, \partial_x \overline{v}\,dx= \langle G,\overline{v}\rangle_{\mH^{-1/2}(\Gamma),\mH^{1/2}(\Gamma)}.
\end{equation}
Besides, by working as in (\ref{EqnVol}), one finds that $u$ satisfies $-\Delta u+u=f$ in $\Om$ and $\partial_\nu u=0$ on $\partial\Om\setminus\overline{\Gamma}$. As a result, we get that
the function $u-U \in \mH^1(\Omega)$ solves (\ref{uh_bis}) with $\Phi=(\phi_-,\phi_+)=(u|_{\Gamma_-},u|_{\Gamma_+})$. We infer that for all $v \in \mW_\alpha(\Om)$, we have
\begin{equation}\label{estim_bis_2}
\int_\Omega \nabla (u-\mathcal{U})\cdot \nabla \overline{v}+(u-\mathcal{U})\overline{v}\,dxdy=\langle \Theta\Phi,\overline{v}\rangle_{\mH^{-1/2}(\Gamma),\mH^{1/2}(\Gamma)}.
\end{equation}
Gathering (\ref{estim_bis_1}) and (\ref{estim_bis_2}), we get, for all $v \in \mW_\alpha(\Om)$,
\begin{equation}\label{estim_bis_3}
-\int_{\Gamma_-}|x|^\alpha\, d_x \phi_-\, d_x \overline{v}\,dx+ \int_{\Gamma_+}x^\alpha\, d_x \phi_+\, d_x \overline{v}\,dx + \langle \Theta\Phi,\overline{v}\rangle_{\mH^{-1/2}(\Gamma),\mH^{1/2}(\Gamma)} = \langle G,\overline{v}\rangle_{\mH^{-1/2}(\Gamma),\mH^{1/2}(\Gamma)}.
\end{equation}
Next, pick some $\Psi \in\mathscr{C}^\infty(\overline{\Gamma})$, extend it to $\partial\Om\setminus\overline{\Gamma}$ to create an element $\tilde{\Psi}\in\mH^{1/2}(\partial\Om)$ and consider some function $v\in\mH^1(\Om)$ such that $v|_{\partial\Om}=\tilde{\Psi}$. Obviously such a $v$ belongs to $\mW_\alpha(\Om)$. Inserting it in (\ref{estim_bis_3}) gives
\[
-\int_{\Gamma_-}|x|^\alpha\, d_x \phi_-\, d_x \overline{\psi_-}\,dx+ \int_{\Gamma_+}x^\alpha\, d_x \phi_+\, d_x \overline{\psi_+}\,dx + \langle \Theta\Phi,\overline{\Psi}\rangle_{\mH^{-1/2}(\Gamma),\mH^{1/2}(\Gamma)} = \langle G,\overline{\Psi}\rangle_{\mH^{-1/2}(\Gamma),\mH^{1/2}(\Gamma)}.
\]
Since this is true for all $\Psi \in\mathscr{C}^\infty(\overline{\Gamma})$, using the density of $\mathscr{C}^\infty(\overline{\Gamma})$ in $\mathcal{X}_\alpha(\mathbb{I})$, we conclude that $\Phi=(u|_{\Gamma_-},u|_{\Gamma_+})$ solves (\ref{1d_bis}). This ends the first part of the proof.\\
\newline
Now, assume that $\Phi\in \mathcal{X}_\alpha(\mathbb{I})$ satisfies (\ref{1d_bis}). Denote respectively by $u_\Phi$, $\mathcal{U}$ the solutions of (\ref{uh_bis}), (\ref{R}). For all $v \in \mW_\alpha(\Om)$, $(v|_{\Gamma_-},v|_{\Gamma_+})$ belongs to $\mathcal{X}_\alpha(\mathbb{I})$. Therefore, for all $v \in \mW_\alpha(\Om)$, we have
\[
-\int_{\Gamma_-}|x|^\alpha\, \partial_x u_\Phi\, \partial_x \overline{v}\,dx+\int_{\Gamma_+}x^\alpha\, \partial_x u_\Phi\, \partial_x \overline{v}\,dx + \langle \partial_\nu u_\Phi,\overline{v} \rangle_{\mH^{-1/2}(\Gamma),\mH^{1/2}(\Gamma)} =-\langle \partial_\nu \mathcal{U},\overline{v} \rangle_{\mH^{-1/2}(\Gamma),\mH^{1/2}(\Gamma)}.
\]
Using that $\mathcal{U}=0$ on $\Gamma$ implies $\partial_x\mathcal{U}=0$ on $\Gamma$, we get, for all $v \in \mW_\alpha(\Om)$,
\[
-\int_{\Gamma_-}|x|^\alpha\, \partial_x (u_\Phi+\mathcal{U})\, \partial_x \overline{v}\,dx+\int_{\Gamma_+}x^\alpha\, \partial_x (u_\Phi+\mathcal{U})\, \partial_x \overline{v}\,dx 
 + \langle \partial_\nu (u_\Phi+\mathcal{U}),\overline{v} \rangle_{\mH^{-1/2}(\partial \Omega),\mH^{1/2}(\partial \Omega)} =0.
\]
Finally, since there holds,
\[\langle \partial_\nu (u_\Phi+\mathcal{U}),\overline{v} \rangle_{\mH^{-1/2}(\partial \Omega),\mH^{1/2}(\partial \Omega)}= \int_\Omega \nabla (u_\Phi+\mathcal{U})\cdot \nabla \overline{v}+(u_\Phi+\mathcal{U})\overline{v}\,dxdy -\int_\Omega f\,\overline{v}\,dxdy,\]
we deduce that $u=u_\Phi+\mathcal{U} \in \mW_\alpha(\Om)$ satisfies Problem (\ref{faible bis}).
\end{proof}
We are now able to prove Theorem \ref{casgeneral_bis}.
\begin{proof}[Proof of Theorem \ref{casgeneral_bis}]
Due to the change of sign of the impedance $x\mapsto \mrm{sign}(x)|x|^\alpha$, we can not directly apply the Lax-Milgram theorem to prove that the 1D variational problem (\ref{1d_bis}) is well-posed. Instead we will adapt the T-coercivity approach presented in \cite{BoCZ10,ChCi11}. Let us define the operator 
\[ 
\begin{array}{rrcl}
\mathbb{T}:& \mathcal{X}_\alpha(\mathbb{I}) &\rightarrow & \mathcal{X}_\alpha(\mathbb{I}) \\ 
&\Phi=(\phi_-,\phi_+)  &\mapsto&   (-\phi_-+2\phi_+(0),\phi_+).
\end{array}
\]
Let us comment a bit on this choice. The ``$-\phi_-$'' in the definition of $\mathbb{T}\Phi$ on $I_-$ will allow us to recover some positivity. However, to ensure that $\mathbb{T}\Phi$ belongs to $\mathcal{X}_\alpha(\mathbb{I})$, we have to compensate for the jump of trace at zero. This explains the presence of the ``$+2\phi_+(0)$''. Using in particular Proposition \ref{prop-Holder}, we can show that $\mathbb{T}$ is continuous. Additionally, we have $\mathbb{T} \circ \mathbb{T}=\mrm{Id}$, where $\mrm{Id} : \mathcal{X}_\alpha(\mathbb{I}) \rightarrow \mathcal{X}_\alpha(\mathbb{I})$ stands for the identity of $\mathcal{X}_\alpha(\mathbb{I})$. This guarantees that $\mathbb{T}$ is an isomorphism. Now, with the Riesz representation theorem, define the continuous operators $\mathsf{S},\mathsf{K}: \mathcal{X}_\alpha(\mathbb{I}) \rightarrow \mathcal{X}_\alpha(\mathbb{I})$ such that for all $\Phi,\Psi \in \mathcal{X}_\alpha(\mathbb{I})$,
\[
(\mathsf{S}\Phi,\Psi)_{\mathcal{X}_\alpha(\mathbb{I})}= -\int_{I_-} |x|^\alpha d_x \phi_-\,d_x \overline{\psi_-}\,dx + \int_{I_+} x^\alpha d_x \phi_+\,d_x \overline{\psi_+}\,dx+ \int_{\mathbb{I}} \mathbb{T}\Phi\,\overline{\Psi}\,dx\]
\[
(\mathsf{K}\Phi,\Psi)_{\mathcal{X}_\alpha(\mathbb{I})}= - \int_{\mathbb{I}} \mathbb{T}\Phi\,\overline{\Psi}\,dx+\langle \Theta\Phi,\overline{\Psi}\rangle_{\mH^{-1/2}(\Gamma),\mH^{1/2}(\Gamma)}.
\]
We have
\[
((\mathsf{S}\circ\mathbb{T})\Phi,\Psi)_{\mathcal{X}_\alpha(\mathbb{I})} =(\Phi,\Psi)_{\mathcal{X}_\alpha(\mathbb{I})},\qquad \forall \Phi,\Psi \in \mathcal{X}_\alpha(\mathbb{I}).
\]
This proves that $\mathsf{S}: \mathcal{X}_\alpha(\mathbb{I}) \rightarrow \mathcal{X}_\alpha(\mathbb{I})$ is an isomorphism whose inverse is $\mathbb{T}$. On the other hand, exploiting that $\mathbb{T}: \mathcal{X}_\alpha(\mathbb{I}) \rightarrow \mathcal{X}_\alpha(\mathbb{I})$ is continuous, that $\Theta:\mH^{1/2}(\Gamma)\to\mH^{-1/2}(\Gamma)$ is continuous and that $\mathcal{X}_\alpha(\mathbb{I})$ is compactly embedded in $\mH^{1/2}(\Gamma)$ (consequence of Proposition \ref{prop-compacite}), we can show that $\mathsf{K}: \mathcal{X}_\alpha(\mathbb{I}) \rightarrow \mathcal{X}_\alpha(\mathbb{I})$ is a compact operator. We deduce that $\mathsf{S}+\mathsf{K}$ is of Fredholm type. Since it is symmetric, it is of index zero. In particular, if only the null function solves (\ref{1d_bis}) with $G\equiv0$, then it admits a solution. The equivalence between
problems (\ref{faible bis}) and (\ref{1d_bis}) given by Proposition \ref{equivalence_bis} completes the proof. 
\end{proof}

\subsection{Proof of Theorem \ref{cas1_bis} -- case $\alpha=1$}

In the previous paragraph, we considered the problem (\ref{faible bis}) for $\alpha\in[0,1)$. Here we address the case $\alpha=1$. In that situation, (\ref{faible bis}) simply writes
\be\label{faible_1_bis}
\begin{array}{|l}
\mbox{Find }u\in\mW_1(\Om)\mbox{ such that for all }v \in \mW_1(\Om)\\[4pt]
\int_\Omega \nabla u\cdot \nabla \overline{v}+u\overline{v}\,dxdy + \int_{\Gamma}x\, \partial_x u\, \partial_x \overline{v}\,dx=\int_\Omega f\overline{v}\,dxdy.
\end{array}	
\ee	
If $u$ solves (\ref{faible_1_bis}), then $u$
satisfies the strong problem:
\be\label{fort_1_bis}
\begin{array}{|rccl}
-\Delta u+u &=& f&  \mbox{in } \Omega\\
\partial_\nu u &=& \partial_x (x\, \partial_x u) & \mbox{on } \Gamma \\
\partial_\nu		u &=& 0&  \mbox{on } \Gamma_0.
\end{array}
\ee
Similarly to what has been done in Section \ref{SectionWithSingus}, let us compute the singularities associated to the principal part of   (\ref{fort_1_bis}) at the origin and which are of the form  $\mathfrak{s}(x,y)=r^\lambda\varphi(\theta)$. This times, we find that $(\lambda,\varphi)$ must satisfy the spectral problem 
\be \label{system_bis} 
	\begin{array}{|rccl}
		d^2_{\theta\theta}\varphi(\theta)+\lambda^2\varphi(\theta) &=& 0&  \mbox{in } (0,\pi)\\
		-d_\theta \varphi(0) &=& \lambda^2\varphi(0) & \\
		-d_\theta\varphi(\pi)&=& \lambda^2\varphi(\pi). & 
	\end{array}
	\ee
For $\lambda \neq 0$, let us look for $\varphi$ in the form $\varphi(\theta)=A\,\cos(\lambda\theta)+B\,\sin(\lambda\theta)$ where $A$, $B$ are some constants. The second relation of (\ref{system_bis}) gives $-B=\lambda A$ and so $\varphi(\theta)=A\,(\cos(\lambda\theta)-\lambda\sin(\lambda\theta))$. From the third equation of (\ref{system_bis}), we find that there is a non zero solution if and only if $\lambda \neq 0$ satisfies
\[
\sin(\lambda \pi)+\lambda\cos(\lambda\pi)=\lambda\,(\cos(\lambda \pi)-\lambda\sin(\lambda\pi))\qquad\Longleftrightarrow\qquad (1+\lambda^2)\sin(\lambda\pi)=0.
\]
Thus, among the singular exponents, we find the values $\lambda=\pm i$.
Then mimicking the proof of Theorem \ref{cas1} with $\tau$ replaced by $1$, we can then establish Theorem \ref{cas1_bis} by working by contradiction. 
In the process, we need the density of $\mathscr{C}^\infty(\overline{\Omega})$ in $\mW_\alpha(\Om)$ for $\alpha \in [0,1]$ and 
more precisely the density of $\mathscr{C}^\infty_0(\overline{\Omega}\setminus\{O\})$ in $\mW_1(\Om)$. The proofs of these results follow the same steps as the ones of Lemma \ref{infini} and Lemma \ref{infini0} respectively.

\section{Relationship between the strong and weak formulations}\label{SectionWeakStrong}

In this section, we discuss the equivalence between strong and weak formulations. As we will see, this is not an obvious point, the reason being that it is not clear in which sense the boundary conditions in the strong problems should be imposed. We start by showing the simple result of Theorem \ref{faible_fort}.\\
\newline
\textit{Proof of Theorem \ref{faible_fort}.}
	Assume that $u \in \mV_\alpha(\Om)$ satisfies (\ref{faible}). Choosing first $v\in \mathscr{C}^\infty_0(\Omega)$ in the variational formulation, we get $-\Delta u+u=f$ in $\mathcal{D}'(\Omega)$. Then working as after (\ref{EqnVol}), we obtain $\partial_\nu u=0$ on $\partial\Om\setminus\overline{\Gamma_+}$. Finally, let us choose $v \in \mV_\alpha(\Om)$ in (\ref{faible}).
Using the Green formula
\be\label{IPP}\int_\Omega \nabla u\cdot \nabla \overline{v}\,dxdy=-\int_\Omega \Delta u\,\overline{v}\,dxdy + \langle\partial_\nu u,\overline{v} \rangle_{\mH^{-1/2}(\partial \Omega),\mH^{1/2}(\partial \Omega)},\ee
we find 
\begin{equation}\label{DtN1D2D}
\langle\partial_\nu u,\overline{v} \rangle_{\mH^{-1/2}(\Gamma_+),\mH^{1/2}(\Gamma_+)} + s\int_{\Gamma_+}x^\alpha\, \partial_x u\, \partial_x \overline{v}\,dx=0.
\end{equation}
Now, consider some $\varphi\in\mathscr{C}^{\infty}_0(\Gamma_+)$, extend it by zero to $\partial\Om$ and introduce some function $v\in\mH^1(\Om)$ such that $v|_{\partial\Om}=\varphi$. Then clearly $v$ belongs to $\mV_\alpha(\Om)$. Inserting it into (\ref{DtN1D2D}) gives
\[
\langle\partial_\nu u,\overline{\varphi} \rangle_{\mH^{-1/2}(\Gamma_+),\mH^{1/2}(\Gamma_+)} + s\int_{\Gamma_+}x^\alpha\, \partial_x u\, \partial_x \overline{\varphi}\,dx=0.
\]
Since this is true for all $\varphi\in\mathscr{C}^{\infty}_0(\Gamma_+)$, we obtain $\partial_\nu u=s \partial_x(x^\alpha\partial_x u)$ in the distributional sense on $\Gamma_+$. This shows that $u$ solves (\ref{fort}).\hfill \qed \\
\newline
We establish similarly the result of Theorem \ref{faible_fort_bis} and for this reason, we omit the proof.\\
\newline
At this stage, it is natural to wonder if the converses of Theorems \ref{faible_fort}, \ref{faible_fort_bis} are true. Namely, do solutions of (\ref{fort}), (\ref{fort_bis}) satisfy (\ref{faible}), (\ref{faible bis}) respectively? It turns out that the answer to this question is no in general because additional conditions are needed at $\partial \Gamma_+$. In other words, it is necessary to enrich the strong problems. To set ideas, let us consider (\ref{fort}). We will assume that $\partial\Om$ also contains the flat segment $\tilde{\Gamma}_+:=(-\eta,1+\eta)\times\{0\}$ for a certain $\eta>0$, and assume that the GIBC is imposed not only in $\mathcal{D}'(\Gamma_+)$ but also in $\mathcal{D}'(\tilde{\Gamma}_+)$. Nevertheless, this is still not satisfactory because at the point $P=(1,0)$, the function $x\mapsto x^{\alpha}\mathds{1}_{\Gamma_+}(x)$ has a jump which prevents from defining $\mathds{1}_{\Gamma_+}\partial_xu$ as an element of $\mathcal{D}'(\tilde{\Gamma}_+)$. This drawback comes from the fact that to simplify the presentation, we have decided to work on problems  (\ref{faible}) and (\ref{faible bis}) which are rather academic. A more natural one is 
\be\label{faible_ter}
\begin{array}{|l}
\mbox{Find }u\in\tilde{\mV}_\alpha(\Om)\mbox{ such that for all }v \in \tilde{\mV}_\alpha(\Om)\\[4pt]
\int_\Omega \nabla u\cdot \nabla \overline{v}+u\overline{v}\,dxdy + s\int_{\Gamma_+}x^\alpha(1-x)^\alpha\, \partial_x u\, \partial_x \overline{v}\,dx=\int_\Omega f\overline{v}\,dxdy
\end{array}	
\ee 
with 
\[
\tilde{\mV}_\alpha(\Om):=\{v \in \mH^1(\Omega)\,|\,x^{\alpha/2}(1-x)^{\alpha/2} \partial_x v \in \mL^2(\Gamma_+)\}.
\]		
It is linked to the strong formulation
\be \label{fort_ter} 
\begin{array}{|rccl}
-\Delta u + u &=& f&  \mbox{in } \Omega\\[3pt]
\partial_\nu u&=&0 & \mbox{on } \partial\Om\setminus\overline{\tilde{\Gamma}_+}\\[3pt]
\partial_\nu u &=& s\,\partial_x (\mathds{1}_{\Gamma_+}(x)\,x^\alpha(1-x)^\alpha\, \partial_x u) & \mbox{on } \tilde{\Gamma}_+,
\end{array}	
\ee
where the GIBC in (\ref{fort_ter}) is written in $\mathcal{D}'(\tilde{\Gamma}_+)$. Before proceeding, let us justify that this GIBC is well-defined in $\mathcal{D}'(\tilde{\Gamma}_+)$. This will be a consequence of the two following lemmas.\\
\newline
For $\alpha \geq 0$, define $g_\alpha$ such that $g_\alpha(x)=\mathds{1}_{\Gamma_+}(x)\,x^\alpha(1-x)^\alpha$.
\begin{lemma}\label{galpha}
The function $g_\alpha$ belongs to $\mH^{1/2}(\tilde{\Gamma}_+)$ if and only if $\alpha>0$.
 \end{lemma}
\begin{proof}
Assume that $\alpha>0$. Define the rectangle $\mathcal{R}:=(-\eta,1+\eta)\times(0,1)$ and introduce $u$ such that $u(x,y)=r^\alpha\cos(\theta/2)$. We observe that $u \in \mH^1(\mathcal{R})$. Indeed, we have
\[
\|u\|_{\mH^1(\mathcal{R})}^2 \le \int_0^\pi\int_0^{1+\eta}\left(|u|^2+ |\partial_r u|^2+r^{-2}|\partial_\theta u|^2\right)r\,drd\theta<+\infty.
\]
Define also $v$ such that $v(x,y)=u(1-x,y)$. Clearly there also holds $v\in\mH^1(\mathcal{R})$. Additionally, we remark that $u\in\mathscr{C}^\infty(\overline{\Om}\setminus\{O\})$ and $v\in\mathscr{C}^\infty(\overline{\Om}\setminus\{P\})$. From this, we deduce that $uv$ belongs to $\mH^1(\mathcal{R})$. By continuity of the trace mapping from $\mH^1(\mathcal{R})$ to $\mH^{1/2}(\tilde{\Gamma}_+)$, since $(uv)|_{\tilde{\Gamma}_+}=g_\alpha$, we infer that $g_\alpha$ belongs to $\mH^{1/2}(\tilde{\Gamma}_+)$.\\
When $\alpha=0$, it is clear that the discontinuous function $g_0$ does not belong to $\mH^{1/2}(\tilde{\Gamma}_+)$ (see again \cite[Chap.\,33]{tartar}). 
\end{proof}
\begin{lemma}
\label{galphav}
For $\alpha>0$ and $v \in \tilde{\mH}^{-1/2}(\tilde{\Gamma}_+)$, $g_\alpha v$ belongs to $\mathcal{D}'(\tilde{\Gamma}_+)$.
\end{lemma}
\begin{proof}
We can define $g_\alpha v$ as the distribution such that for $\phi \in \mathscr{C}^\infty_0(\tilde{\Gamma}_+)$, 
\[\langle g_\alpha\, v, \phi\rangle:=\langle v, g_\alpha\, \phi \rangle\]
provided that $g_\alpha \phi$ belongs to $\tilde{\mH}^{1/2}(\tilde{\Gamma}_+)$. But this can be shown exactly as in the proof of Lemma \ref{galpha}. More precisely, by setting $w(x,y)=u(x,y)v(x,y)\phi(x)$, we find
\[
\|w\|_{\mH^1(\mathcal{R})} \le C\,(\|\phi\|_{\mL^\infty(\tilde{\Gamma}_+)} + \|d_x\phi\|_{\mL^\infty(\tilde{\Gamma}_+)})
\]
where here and below $C>0$ is a constant which is independent of $\phi$. By continuity of the trace mapping from $\mH^1(\mathcal{R})$ to $\mH^{1/2}(\tilde{\Gamma}_+)$, since $w|_{\tilde{\Gamma}_+}=g_\alpha\phi$, we infer that 
\[
\|g_\alpha \phi\|_{\mH^{1/2}(\tilde{\Gamma}_+)}\le C\,(\|\phi\|_{\mL^\infty(\tilde{\Gamma}_+)} + \|d_x\phi\|_{\mL^\infty(\tilde{\Gamma}_+)}).
\]
Introduce the points $P_1:=(-\eta,0)$, $P_2:=(1+\eta,0)$ corresponding to the ends of $\tilde{\Gamma}_+$, and $d(x):=d((x,0),\{P_1,P_2\})$ ($d$ is the distance function to the boundary of $\tilde{\Gamma}_+$).
Classically, the space $\tilde{\mH}^{1/2}(\tilde{\Gamma}_+)$ can be endowed with the norm $\|\cdot\|_{\tilde{\mH}^{1/2}(\tilde{\Gamma}_+)}$ such that for $\psi \in \tilde{\mH}^{1/2}(\tilde{\Gamma}_+)$
\[\|\psi\|_{\tilde{\mH}^{1/2}(\tilde{\Gamma}_+)}=\left(\|\psi\|^2_{\mH^{1/2}(\tilde{\Gamma}_+)}+ \left\|d^{-1/2} \psi\right\|^2_{\mL^2(\tilde{\Gamma}_+)}\right)^{1/2}.\]
Hence for all $\phi \in \mathscr{C}^\infty_0(\tilde{\Gamma}_+)$ which is compactly supported in $K \subset \tilde{\Gamma}_+$, there exists a constant $C_K$ depending on $K$ such that
\[
\|g_\alpha \phi\|_{\tilde{\mH}^{1/2}(\tilde{\Gamma}_+)}\le C_K \|g_\alpha \phi\|_{\mH^{1/2}(\tilde{\Gamma}_+)} \le C_K\,(\|\phi\|_{\mL^\infty(\tilde{\Gamma}_+)} + \|d_x\phi\|_{\mL^\infty(\tilde{\Gamma}_+)}).
\]
This yields 
\[
|\langle g_\alpha\, v, \phi\rangle| \leq C_K\,(\|\phi\|_{\mL^\infty(\tilde{\Gamma}_+)} + \|d_x\phi\|_{\mL^\infty(\tilde{\Gamma}_+)}),
\]
which proves that $g_\alpha v$ is an element of $\mathcal{D}'(\tilde{\Gamma}_+)$.
\end{proof}
\begin{proposition}\label{PropositionDistrib}
For $\alpha>0$, if $u\in \tilde{\mV}_\alpha(\Om)$, then $\partial_x(\mathds{1}_{\Gamma_+}(x)\,x^\alpha(1-x)^\alpha\partial_x u)$ belongs to $\mathcal{D}'(\tilde{\Gamma}_+)$.
\end{proposition}
\begin{proof}
If $u\in \tilde{\mV}_\alpha(\Om)$ then we have $u\in\mH^1(\Om)$ and so $u|_{\tilde{\Gamma}_+} \in \mH^{1/2}(\tilde{\Gamma}_+)$. This implies that $\partial_x u|_{\tilde{\Gamma}_+} \in \tilde{\mH}^{-1/2}(\tilde{\Gamma}_+)$. Then 
Lemma \ref{galphav} guarantees that $g_\alpha \partial_x u|_{\tilde{\Gamma}_+}$ is well defined in $\mathcal{D}'(\tilde{\Gamma}_+)$. As a consequence, $\partial_x(g_\alpha \partial_x u)$ belongs to $\mathcal{D}'(\tilde{\Gamma}_+)$.
\end{proof}
Now we can state the main result concerning the relationship between problems (\ref{faible_ter}) and (\ref{fort_ter}).
\begin{theorem}\label{faible_fort_ter}~\\
For any $\alpha \geq 0$, if $u$ satisfies (\ref{faible_ter}), then it solves (\ref{fort_ter}).\\
For $\alpha \in (0,1]$,	if $u \in \tilde{\mV}_\alpha(\Om)$ satisfies (\ref{fort_ter}), then it solves (\ref{faible_ter}).
\end{theorem}	

\begin{proof}
The first item can be established by working as in the proof of Theorem \ref{faible_fort} with the help of Proposition \ref{PropositionDistrib}. Let us focus our attention on the second point.	\\
Assume that $u \in \tilde{\mV}_\alpha(\Om)$ satisfies (\ref{fort_ter}). Introducing again the points $P_1=(-\eta,0)$, $P_2=(1+\eta,0)$ corresponding to the ends of $\tilde{\Gamma}_+$, multiplying the first equation of (\ref{fort_ter}) by $v \in \mathscr{C}^\infty_0(\overline{\Omega}\setminus\{P_1,P_2\})$ and using the Green formula (\ref{IPP}),
we get
\be\label{interm2}
\int_\Omega \nabla u\cdot \nabla \overline{v}+u\overline{v}\,dxdy - \langle\partial_\nu u,\overline{v} \rangle_{\mH^{-1/2}(\tilde{\Gamma}_+),\mH^{1/2}(\tilde{\Gamma}_+)}=\int_\Omega f\,\overline{v}\,dxdy.\ee
Since $u \in \tilde{\mV}_\alpha(\Om)$, Proposition \ref{PropositionDistrib} guarantees that $\partial_x(g_\alpha \partial_x u)$ belongs to $\mathcal{D}'(\tilde{\Gamma}_+)$. Exploiting that $v|_{\tilde{\Gamma}_+} \in \mathscr{C}^\infty_0(\tilde{\Gamma}_+)$, we can write 
\[\langle \partial_x( g_\alpha \partial_x u),\overline{v} \rangle_{\mathcal{D}'(\tilde{\Gamma}_+),\mathscr{C}^\infty_0(\tilde{\Gamma}_+)}= -\langle g_\alpha \partial_x u,\partial_x \overline{v}\rangle_{\mathcal{D}'(\tilde{\Gamma}_+),\mathscr{C}^\infty_0(\tilde{\Gamma}_+)}=-\int_{\Gamma_+} x^\alpha \partial_x u\,\partial_x \overline{v}\,dx.\]
On the other hand, using that
\[s\langle \partial_x (g_\alpha\partial_x u),\overline{v} \rangle_{\mathcal{D}'(\tilde{\Gamma}_+),\mathscr{C}^\infty_0(\tilde{\Gamma}_+)}= \langle \partial_\nu u,\overline{v} \rangle_{\mathcal{D}'(\tilde{\Gamma}_+),\mathscr{C}^\infty_0(\tilde{\Gamma}_+)}=\langle \partial_\nu u,\overline{v}\rangle_{\mH^{-1/2}(\tilde{\Gamma}_+),\mH^{1/2}(\tilde{\Gamma}_+)},\]	
we obtain	
\[\langle \partial_\nu u,\overline{v}\rangle_{\mH^{-1/2}(\tilde{\Gamma}_+),\mH^{1/2}(\tilde{\Gamma}_+)}=-s\int_{\Gamma_+} x^\alpha \partial_x u\,\partial_x \overline{v}\,dx. \]
With (\ref{interm2}), this gives
\[\int_\Omega \nabla u\cdot \nabla \overline{v}+u\overline{v}\,dxdy + s \int_{\Gamma_+} x^\alpha(1-x)^\alpha\partial_x u\, \partial_x \overline{v}\,dx=\int_\Omega f\,\overline{v}\,dxdy,\qquad \forall v \in\mathscr{C}^\infty_0(\overline{\Om}\setminus\{P_1,P_2\}).
\]
Finally, we can conclude by using the density of $\mathscr{C}^\infty_0(\overline{\Om}\setminus\{P_1,P_2\})$ in $\tilde{\mV}_\alpha(\Om)$. This latter result can be shown by exploiting Lemma \ref{infini} and using the fact that $\mH^1$ functions vanishing in a neighbourhood of $\{P_1, P_2\}$ are dense in $\mH^1(\Om)$ (see e.g. Lemma 1.2.2 in \cite{Claeys08}).
\end{proof}

It is natural to wonder if the second point of Theorem \ref{faible_fort_ter} also holds for $\alpha=0$. Note that for $\alpha=0$, there holds $\tilde{\mV}_0(\Om)=\mV_0(\Om)$ and that (\ref{faible}), (\ref{faible_ter}) are the same. Let us define the space
\[
\mV^2_0(\Om):=\{v \in \mH^1(\Omega),\,v|_{\Gamma_+} \in \mH^2(\Gamma_+)\}.\]
We emphasize that $\mV^2_0(\Om)$ is smaller than $\mV_0(\Om)$ because in the latter space, we only impose $v|_{\Gamma_+} \in \mH^1(\Gamma_+)$. We have the following result.
\begin{proposition}
A function $u\in\mV^2_0(\Om)$ solves (\ref{faible_ter}) with $\alpha=0$ if and only if it satisfies 
\be  \label{fort_regulier}
	\begin{array}{|rccl}
-\Delta u +u&=& f&  \mbox{in } \Omega\\[3pt]
\partial_\nu u&=&0 & \mbox{on } \partial\Om\setminus\overline{\Gamma_+}\\[3pt]
\partial_\nu u &=& s\,\partial_{xx} u & \mbox{on } \Gamma_+\\[3pt]
\partial_x u|_{\Gamma_+} &=&0 & \mbox{at }\{O,P\}\mbox{ with }P=(1,0).
	\end{array}	
	\ee
\end{proposition} 
\begin{proof}
Assume that $u$ solves (\ref{faible_ter}) with $\alpha=0$ and belongs to $\mV^2_0(\Om)$.
From Theorem \ref{faible_fort}, it is clear that $u$ satisfies all the equations of (\ref{fort}) with $\alpha=0$.
On the other hand, for all $v \in \mV_0(\Om)$, we have
\[\langle\partial_\nu u,\overline{v} \rangle_{\mH^{-1/2}(\Gamma_+),\mH^{1/2}(\Gamma_+)} + s\int_{\Gamma_+}\partial_x u\, \partial_x \overline{v}\,dx=0.\]
Choosing $v$ such that $v|_{\Gamma_+}=\phi$ where $\phi$ is a given element of $\mathscr{C}^\infty(\overline{\Gamma_+})$, we get
\[\langle\partial_\nu u,\overline{\phi} \rangle_{\mH^{-1/2}(\Gamma_+),\mH^{1/2}(\Gamma_+)} + s\int_{\Gamma_+}\partial_x u\, \partial_x \overline{\phi}\,dx=0.\]
Using an integration by parts and the fact that $u|_{\Gamma_+} \in \mH^2(\Gamma_+)$, for all 
$\phi \in\mathscr{C}^\infty(\overline{\Gamma_+})$, we obtain
\[\langle\partial_\nu u,\overline{\phi} \rangle_{\mH^{-1/2}(\Gamma_+),\mH^{1/2}(\Gamma_+)} - s\int_{\Gamma_+}\partial^2_{xx} u\,\overline{\phi}\,dx - s\partial_x u|_{\Gamma_+}(O)\overline{\phi|_{\Gamma_+}(O)}+ s\partial_x u|_{\Gamma_+}(P)\overline{\phi|_{\Gamma_+}(P)}=0.\]
This yields $\partial_x u|_{\Gamma_+}(O)=\partial_x u|_{\Gamma_+}(P)=0$.\\
The converse statement follows the same lines.
\end{proof}
\begin{remark}
We do not know if Theorem \ref{faible_fort_ter} holds for $\alpha=0$ but we conjecture that additional conditions at $O$, $P$ are required for a solution to  (\ref{fort}) to always satisfy (\ref{faible_ter}). Unfortunately, the additional conditions $\partial_x u|_{\Gamma_+}(O)=\partial_x u|_{\Gamma_+}(P)=0$ have no meaning in general for $u \in \mV_0(\Om)$. 
\end{remark}

\section{Numerical illustrations}\label{Section_Num}

In this section, we provide numerical results concerning the discretization  of the problems (\ref{faible}) and (\ref{faible bis}). More precisely, we solve (\ref{faible}) and (\ref{faible bis}) using a standard P2 finite element method and display the numerical solutions for different meshes of the geometry. Admittedly, we have no guarantee that when the problem at the continuous level is well-posed, the numerical solution converges to the exact solution when the mesh is refined. The question of the numerical approximations of (\ref{faible}) and (\ref{faible bis}), when they are well-posed, remains to be studied. Here we simply wish to present what is observed. For the numerical analysis of problems similar to (\ref{faible}) with $s=-1$ and $\alpha=0$, we refer the reader to \cite{KaCDQ15,Inve22,CaGP23}.\\
\newline
Practically, we work in the rectangle $\Om=(-1,1)\times(0,1)$, use the library \texttt{Freefem++} \cite{Hech12} to compute the numerical solutions and display the results with \texttt{Paraview}\footnote{\texttt{Paraview}, \url{http://www.paraview.org/}.}. For the source term, we choose $f(x,y)=\cos(x)$.\\
\newline
In Figures \ref{Fig_sm1_alpha_0p5}--\ref{Fig_sm1_alpha_1p5}, we present the numerical solutions associated with the problem (\ref{faible}) for $s=-1$ and different values of $\alpha$. In Figures \ref{Fig_sm1_alpha_0p5}, \ref{Fig_sm1_alpha_0p95}, we take respectively $\alpha=0.5$ and $\alpha=0.95$. The meshes used in the experiments appear in Figure \ref{Fig_meshes}. In Figures \ref{Fig_sm1_alpha_0p5}, \ref{Fig_sm1_alpha_0p95}, when the mesh is refined, we observe that the numerical solution seems to converge. This is quite in agreement with Theorem \ref{casgeneral} with guarantees that (\ref{faible}) is well-posed in the Fredholm sense when $\alpha\in[0,1)$. 

\begin{figure}[!ht]
\centering
\includegraphics[width=3.2cm]{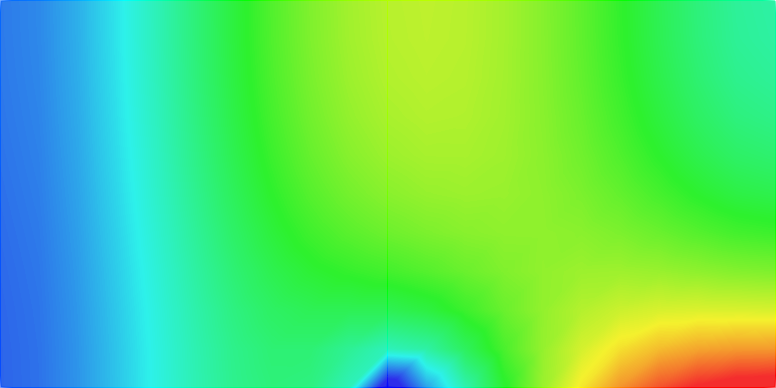}\,\includegraphics[width=3.2cm]{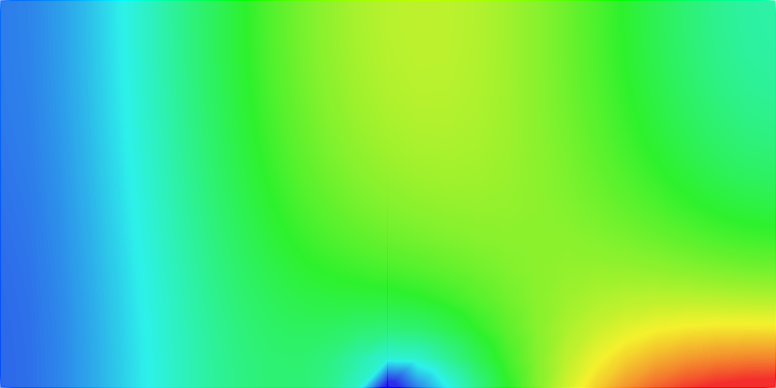}\,\includegraphics[width=3.2cm]{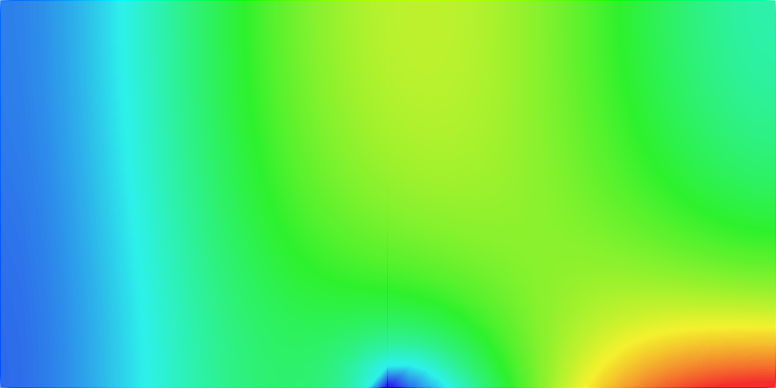}\,\includegraphics[width=3.2cm]{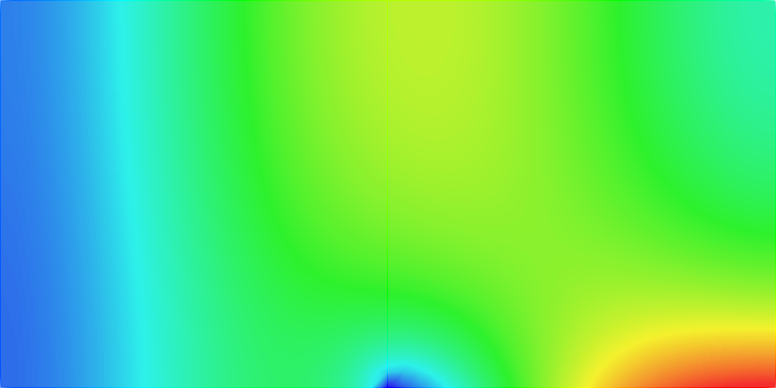}\,\includegraphics[width=3.2cm]{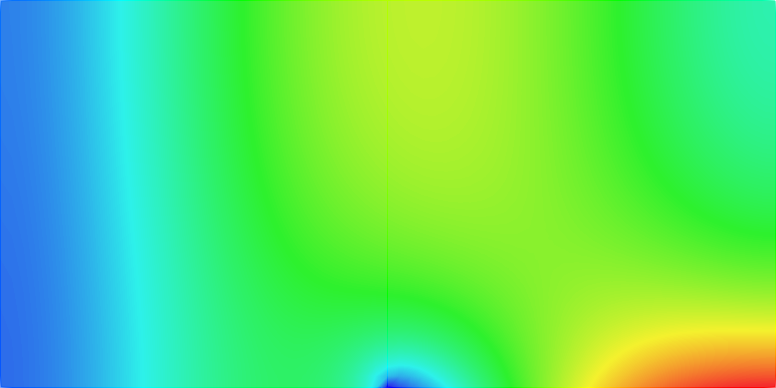}\\[3pt]
\includegraphics[width=3.2cm]{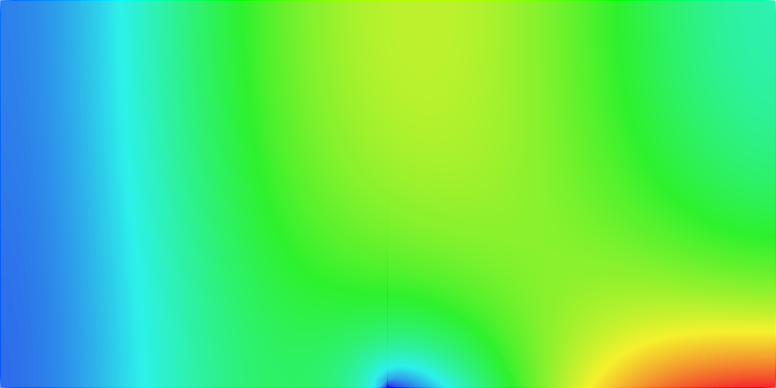}\,\includegraphics[width=3.2cm]{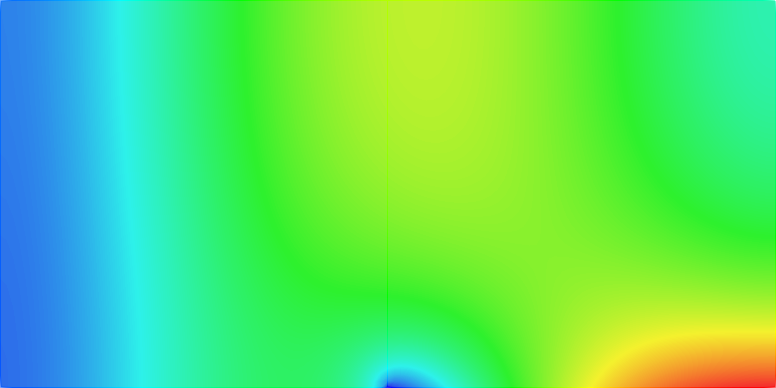}\,\includegraphics[width=3.2cm]{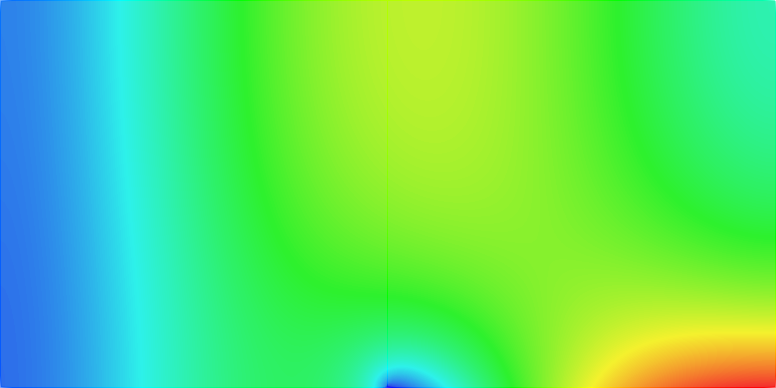}\,\includegraphics[width=3.2cm]{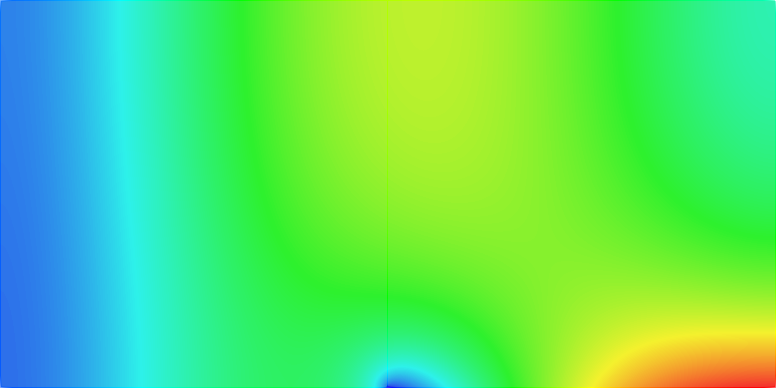}\,\includegraphics[width=3.2cm]{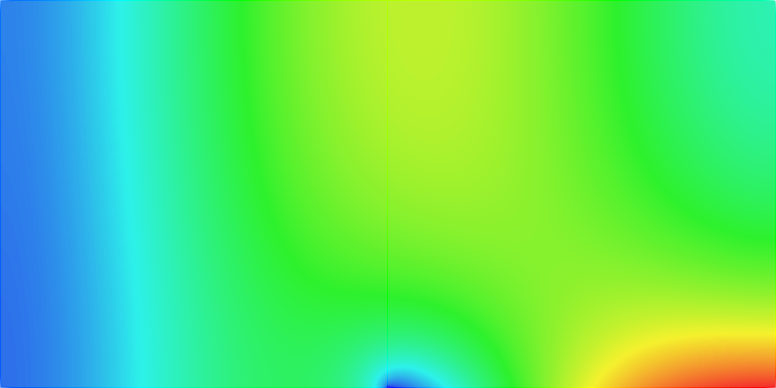}
\caption{Numerical resolution of (\ref{faible}) for ten different meshes with $s=-1$ and $\alpha=0.5$.\label{Fig_sm1_alpha_0p5}}
\end{figure}

\begin{figure}[!ht]
\centering
\includegraphics[width=3.2cm]{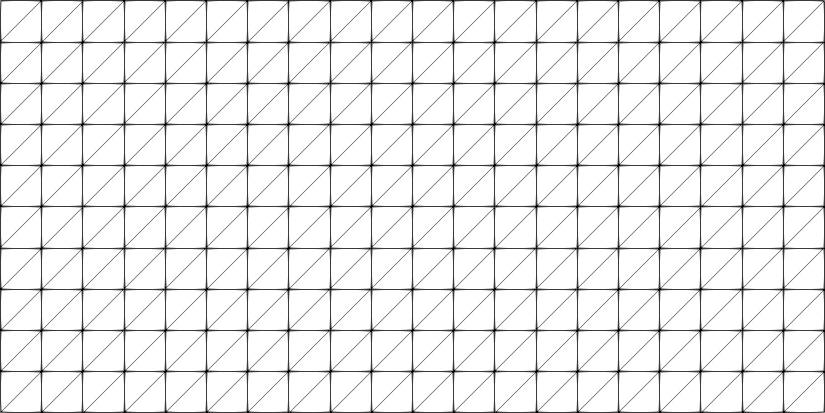}\,\includegraphics[width=3.2cm]{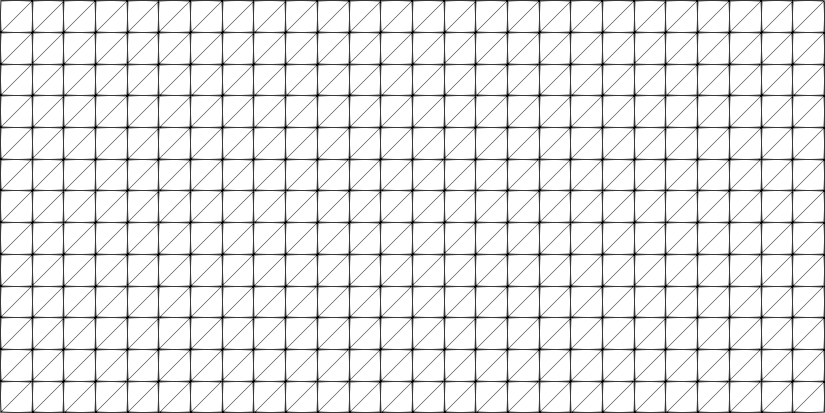}\,\includegraphics[width=3.2cm]{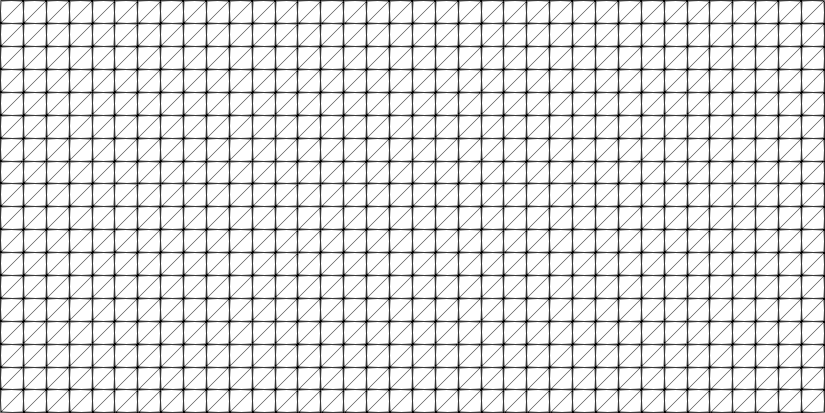}\,\includegraphics[width=3.2cm]{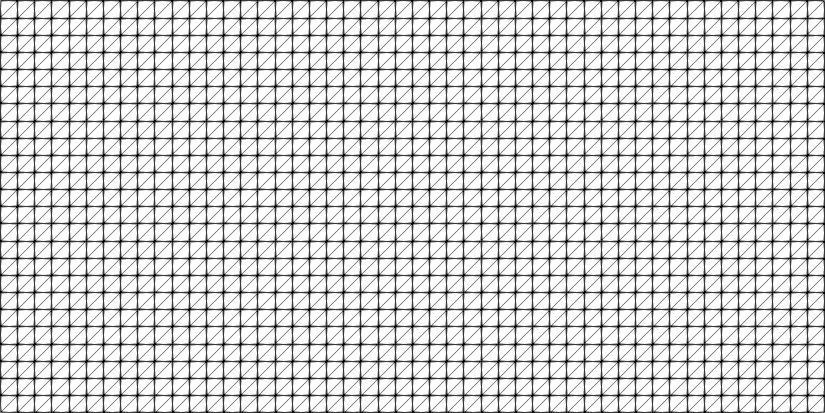}\,\includegraphics[width=3.2cm]{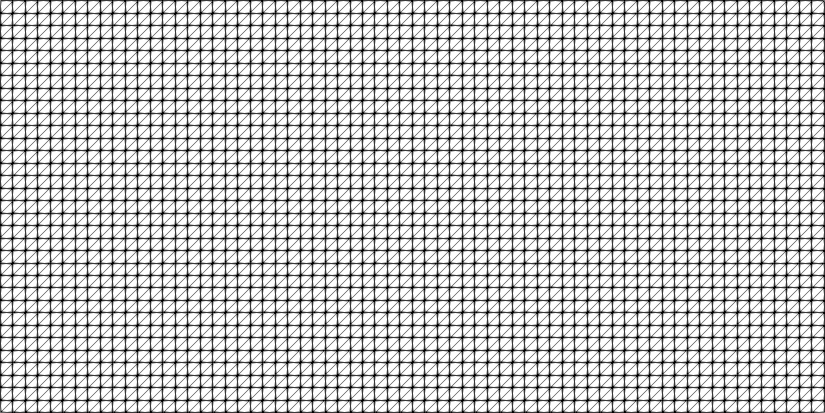}\\[3pt]
\includegraphics[width=3.2cm]{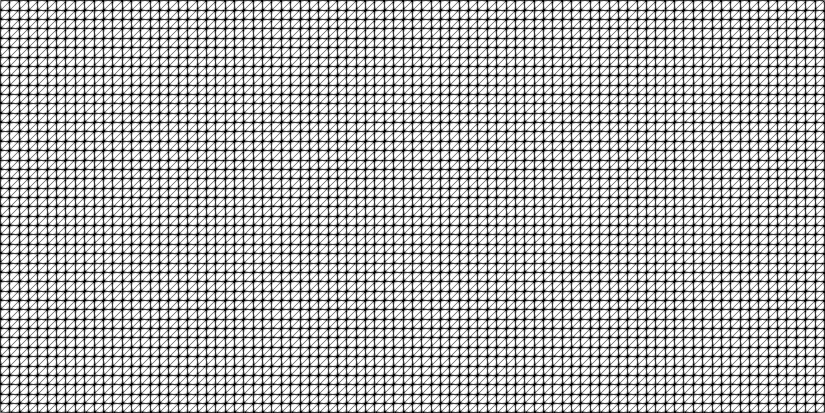}\,\includegraphics[width=3.2cm]{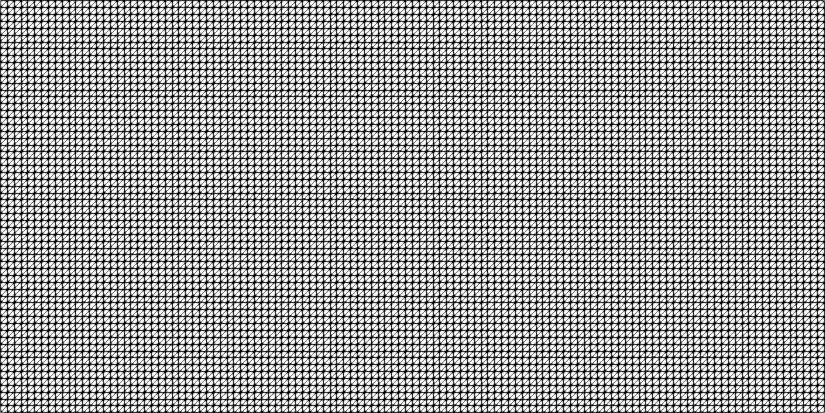}\,\includegraphics[width=3.2cm]{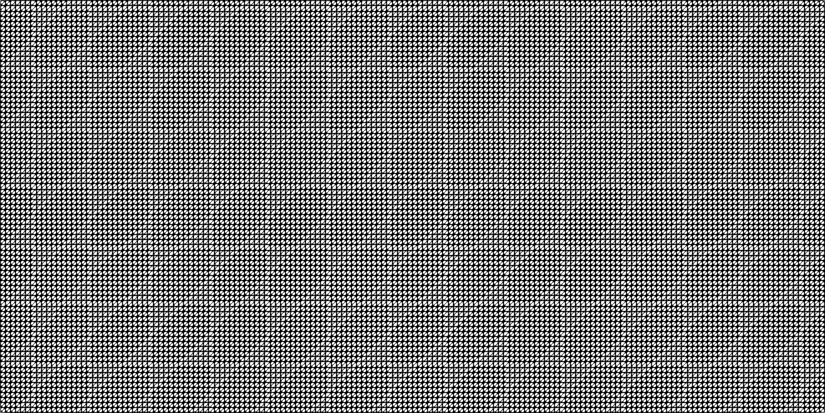}\,\includegraphics[width=3.2cm]{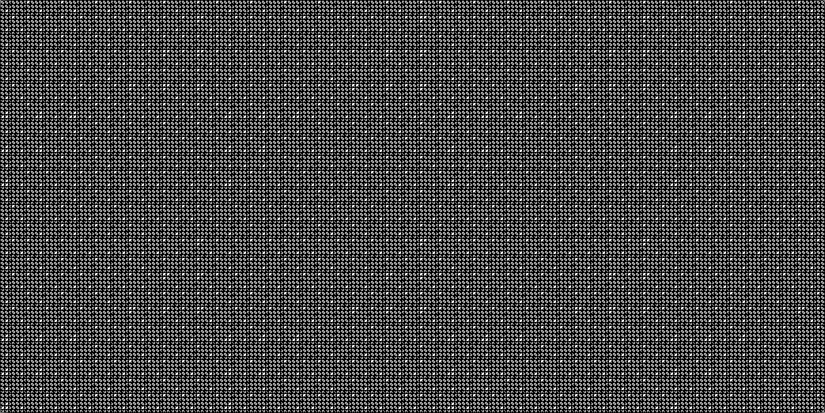}\,\includegraphics[width=3.2cm]{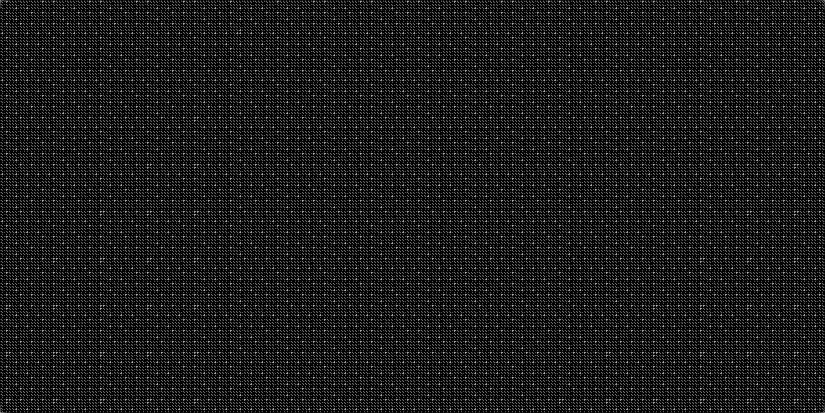}
\caption{Meshes used in the experiments.\label{Fig_meshes}}
\end{figure}

\begin{figure}[!ht]
\centering
\includegraphics[width=3.2cm]{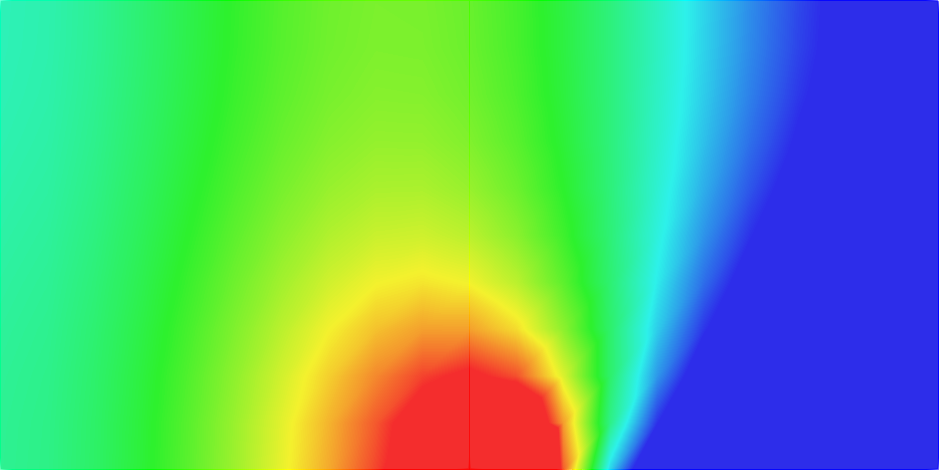}\,\includegraphics[width=3.2cm]{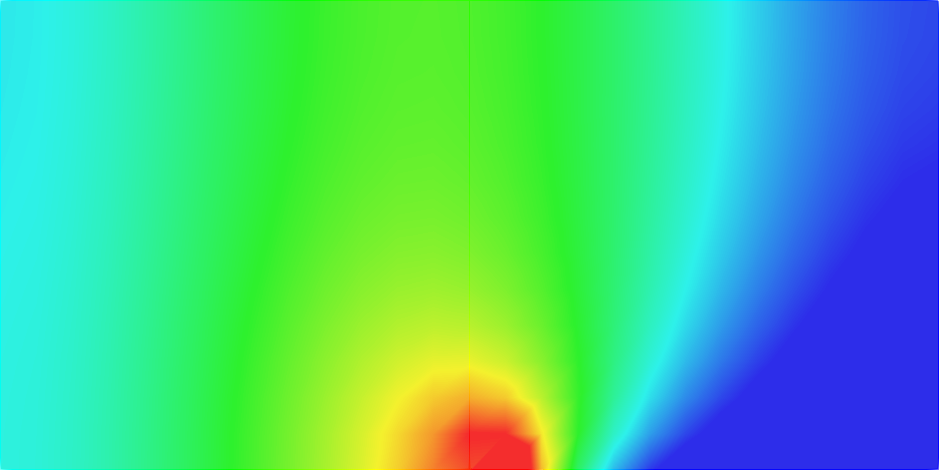}\,\includegraphics[width=3.2cm]{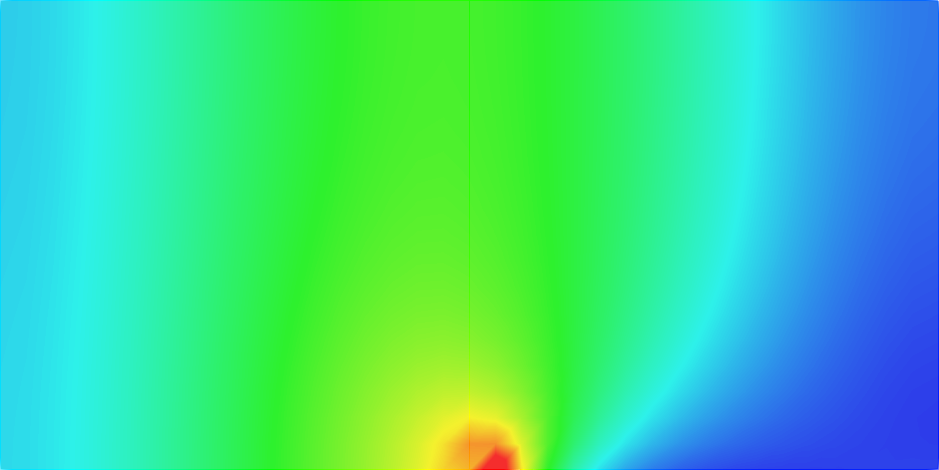}\,\includegraphics[width=3.2cm]{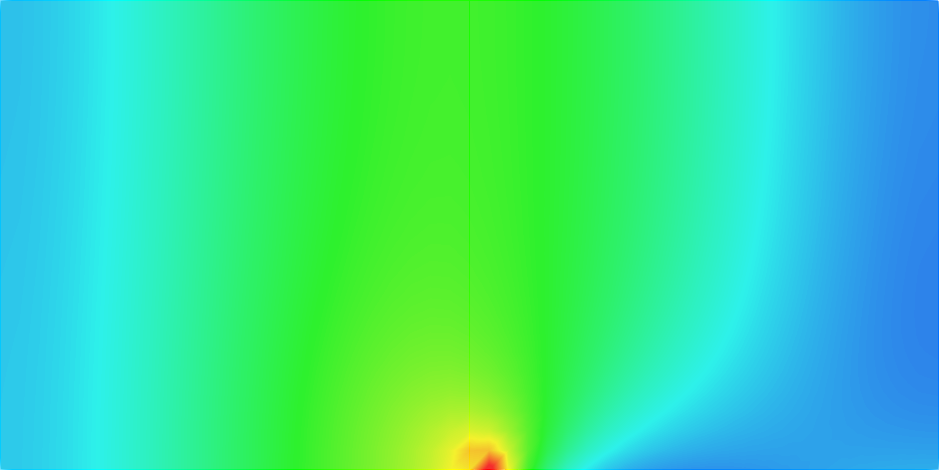}\,\includegraphics[width=3.2cm]{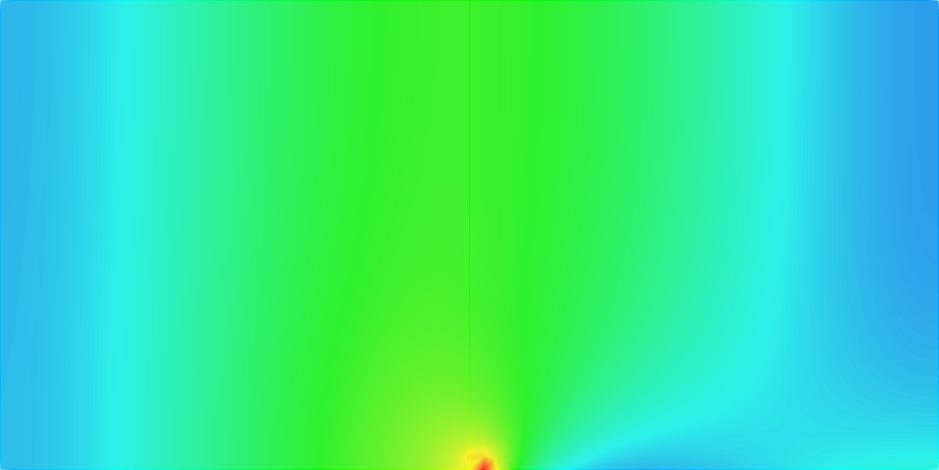}\\[3pt]
\includegraphics[width=3.2cm]{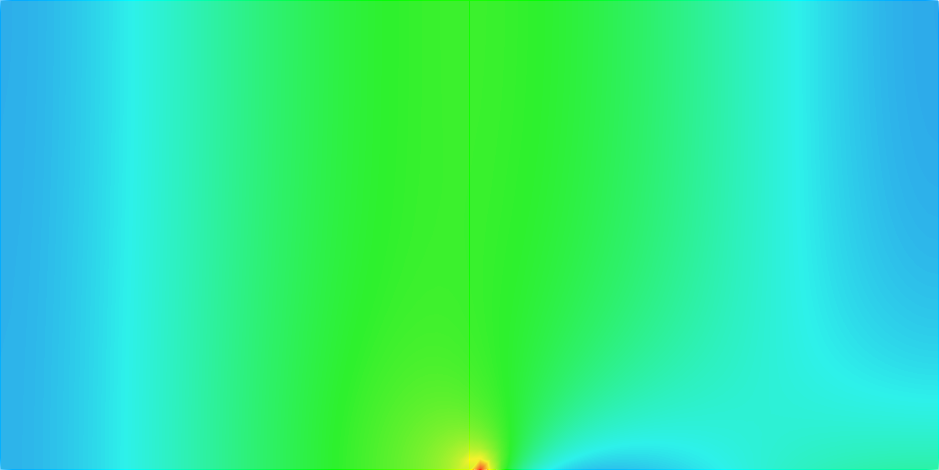}\,\includegraphics[width=3.2cm]{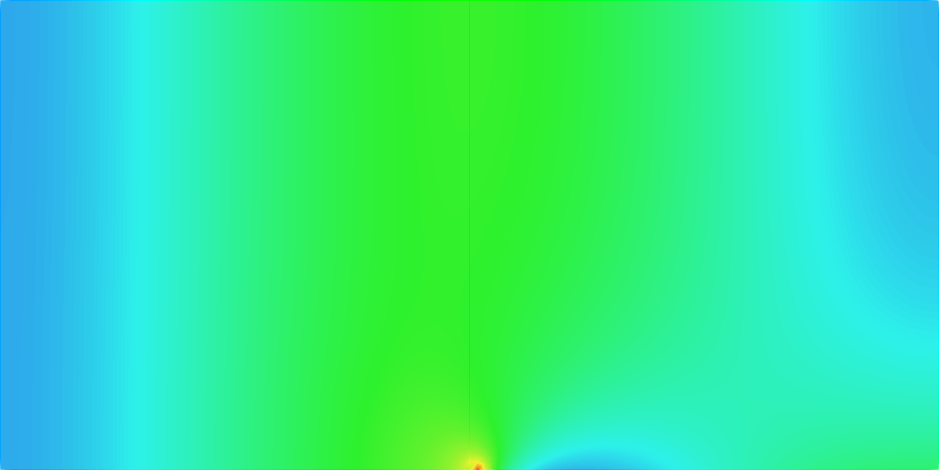}\,\includegraphics[width=3.2cm]{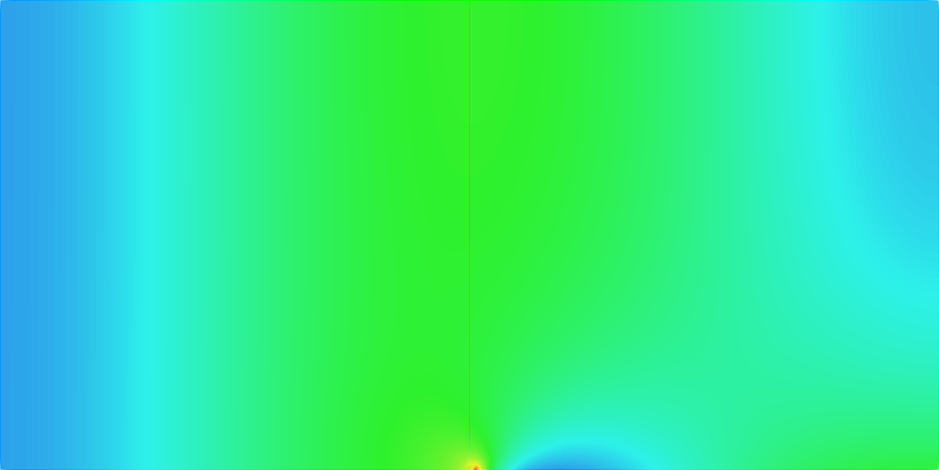}\,\includegraphics[width=3.2cm]{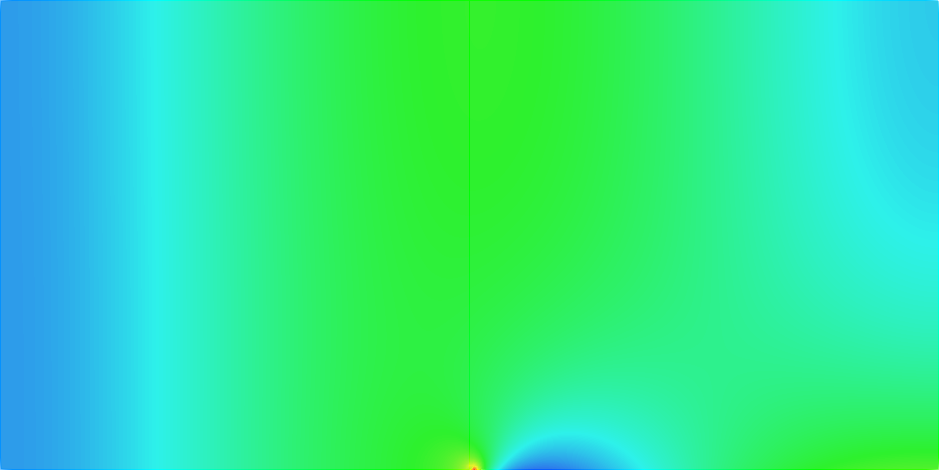}\,\includegraphics[width=3.2cm]{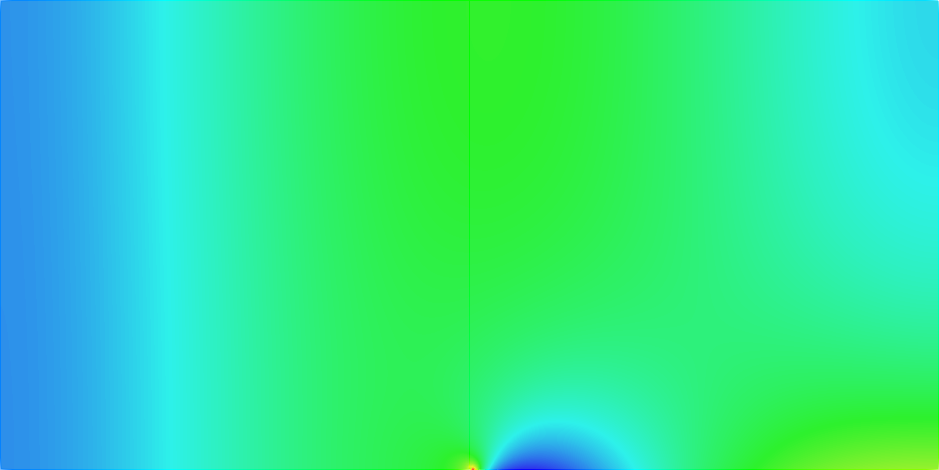}
\caption{Numerical resolution of (\ref{faible}) for ten different meshes with $s=-1$ and $\alpha=0.95$.\label{Fig_sm1_alpha_0p95}}
\end{figure}

In Figure \ref{Fig_sm1_alpha_1}, we display similar results but this time with $\alpha=1$. We note the that the numerical solution does not converge when the mesh is refined. According to Theorem \ref{cas1}, which ensures that the operator $A: \mV_1(\Om) \rightarrow \mV_1(\Om)$ is not of Fredholm type in that case, this was somehow expected. In Figure \ref{Fig_sm1_singu}, we show another view of the numerical solution for one particular mesh. This allows us to illustrate the singular behaviour at the origin which is responsible for the ill-posedness of the problem (see (\ref{defBHsingu})).

\begin{figure}[!ht]
\centering
\includegraphics[width=3.2cm]{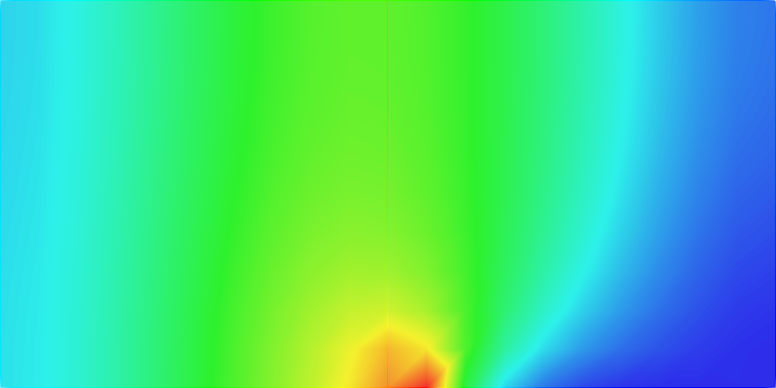}\,\includegraphics[width=3.2cm]{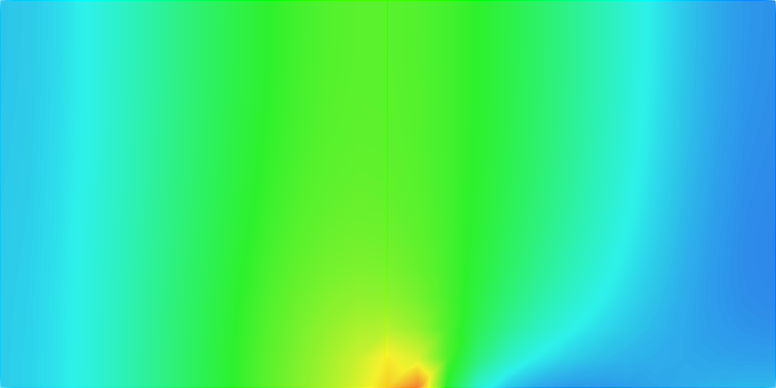}\,\includegraphics[width=3.2cm]{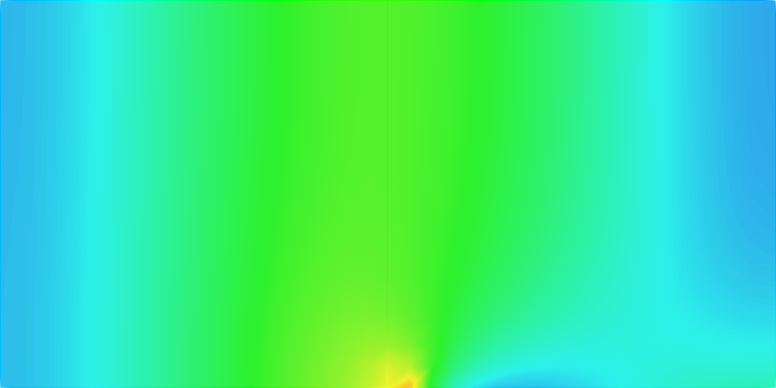}\,\includegraphics[width=3.2cm]{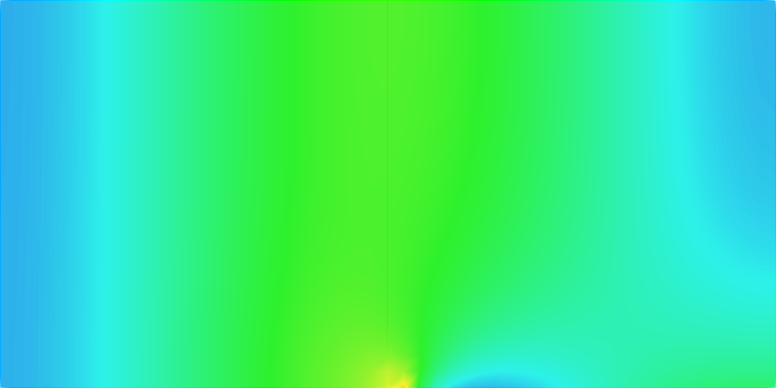}\,\includegraphics[width=3.2cm]{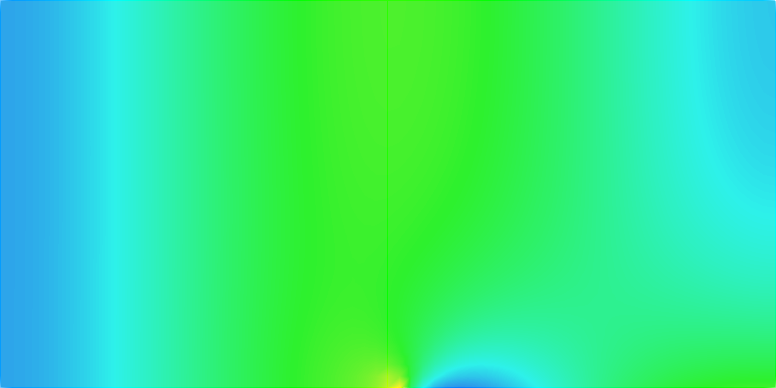}\\[3pt]
\includegraphics[width=3.2cm]{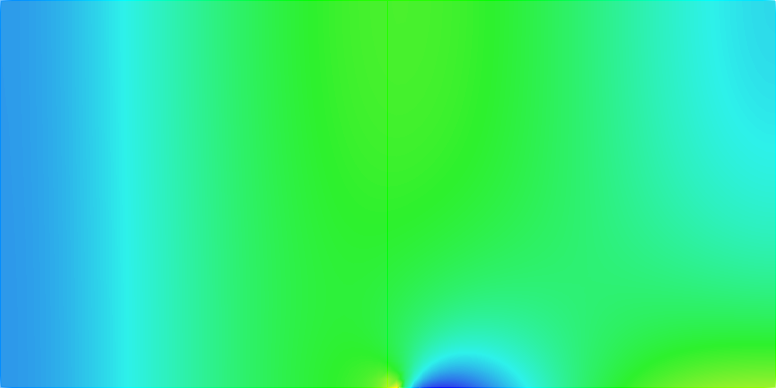}\,\includegraphics[width=3.2cm]{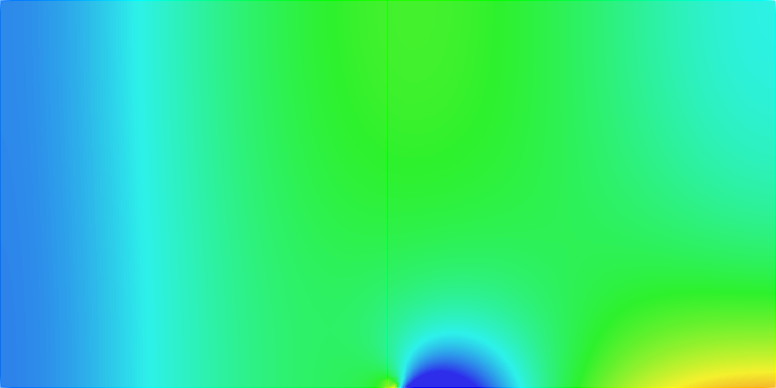}\,\includegraphics[width=3.2cm]{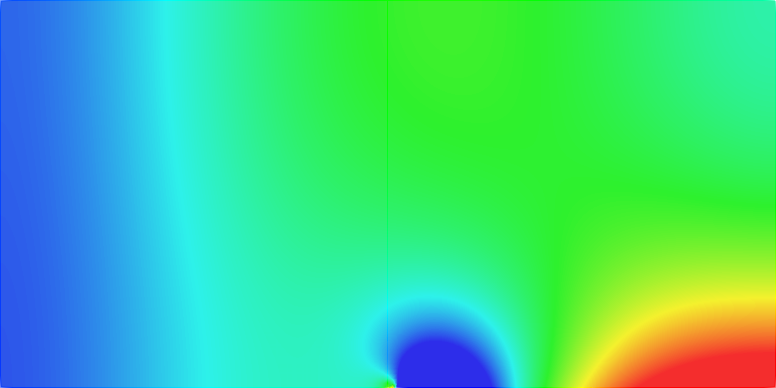}\,\includegraphics[width=3.2cm]{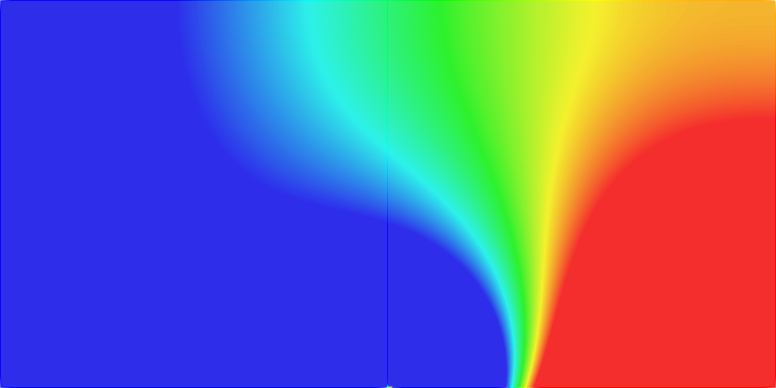}\,\includegraphics[width=3.2cm]{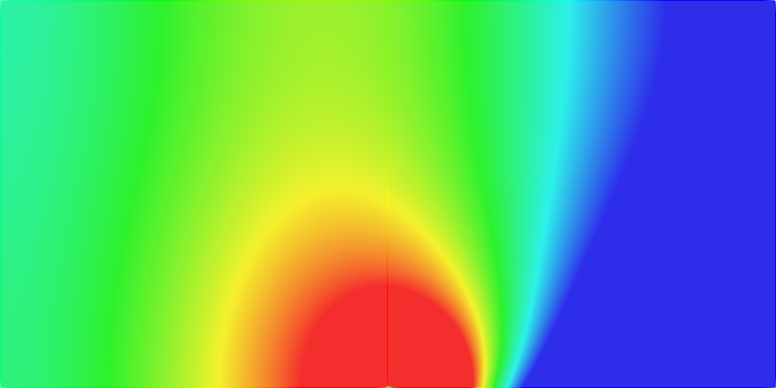}
\caption{Numerical resolution of (\ref{faible}) for ten different meshes with  $s=-1$ and $\alpha=1$.\label{Fig_sm1_alpha_1}}
\end{figure}

\begin{figure}[!ht]
\centering
\includegraphics[width=9cm]{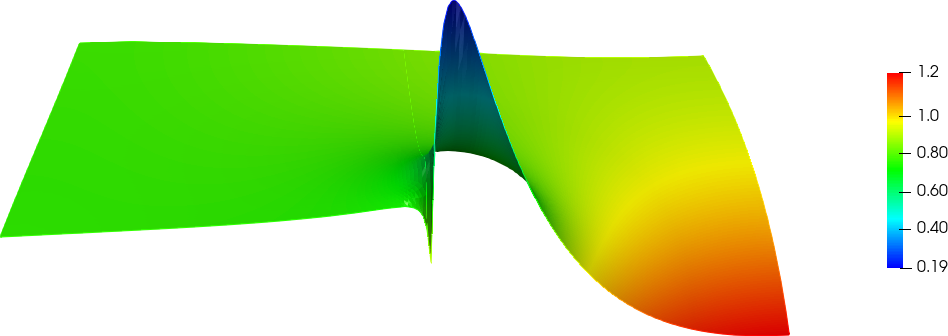}
\caption{Illustration of the singular behaviour of the numerical solution of (\ref{faible}) with  $s=-1$ and $\alpha=1$. Here the mesh is the same as in the 7th image of Figure \ref{Fig_sm1_alpha_1}.\label{Fig_sm1_singu}}
\end{figure}
\newpage

In Figure \ref{Fig_sm1_alpha_1p5}, we work with $\alpha=1.5$. Again, the numerical solution does not converge when we refine the mesh. This suggests that the problem at the continuous level is not well-posed in the Fredholm sense, a result that we have not been able to prove. 

\begin{figure}[!ht]
\centering
\includegraphics[width=3.2cm]{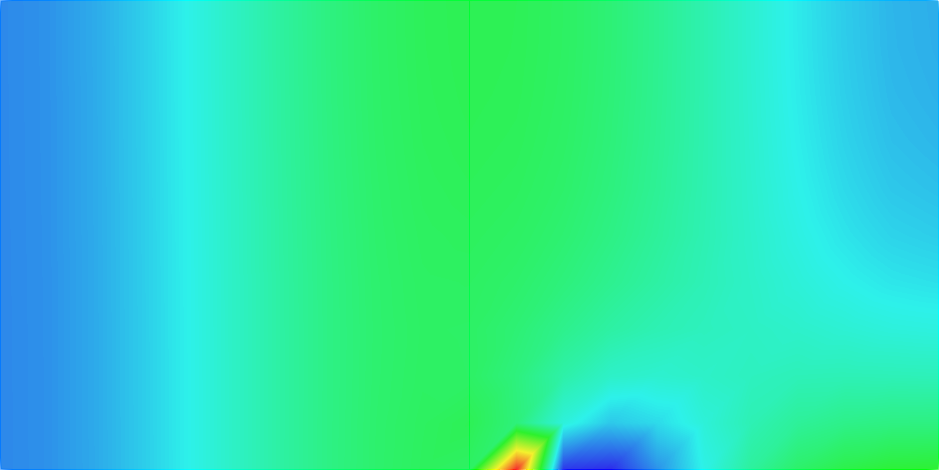}\,\includegraphics[width=3.2cm]{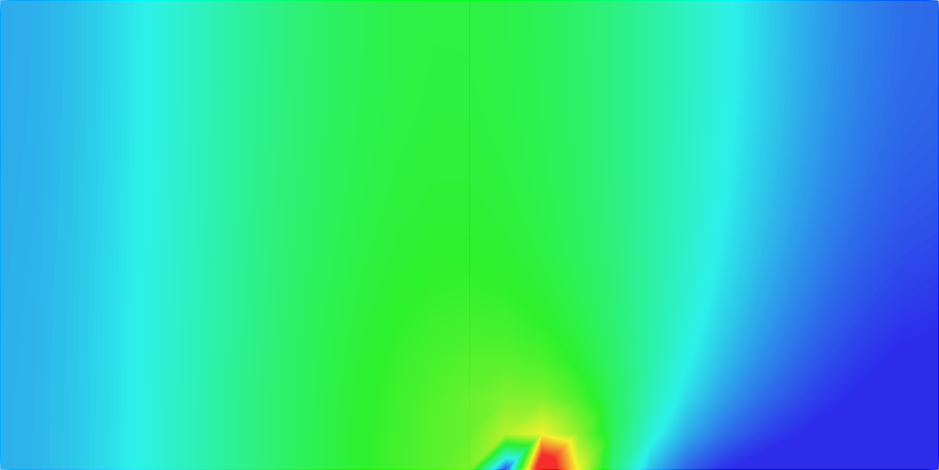}\,\includegraphics[width=3.2cm]{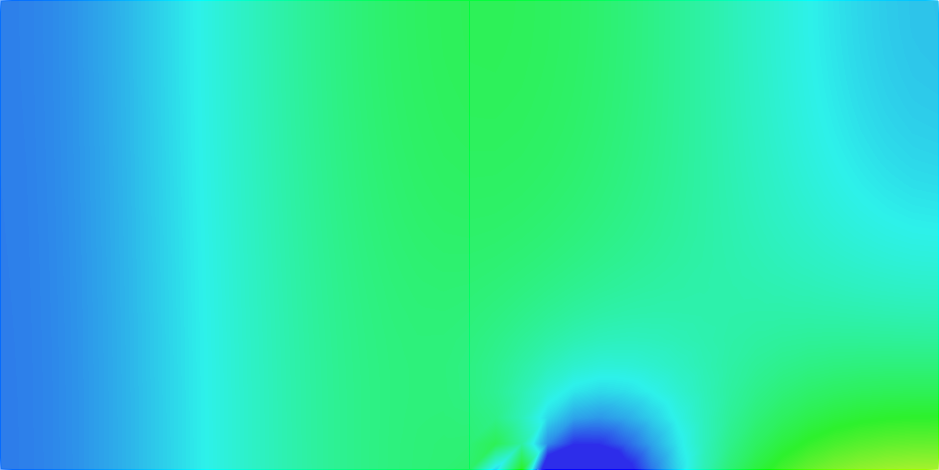}\,\includegraphics[width=3.2cm]{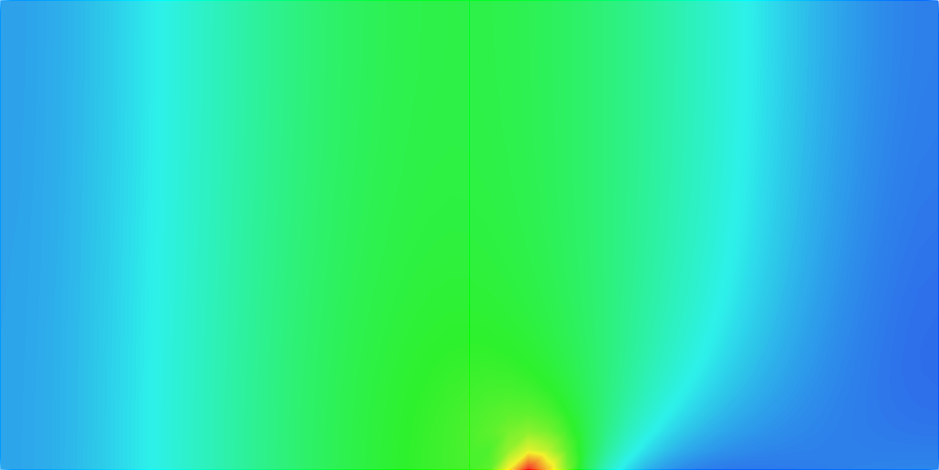}\,\includegraphics[width=3.2cm]{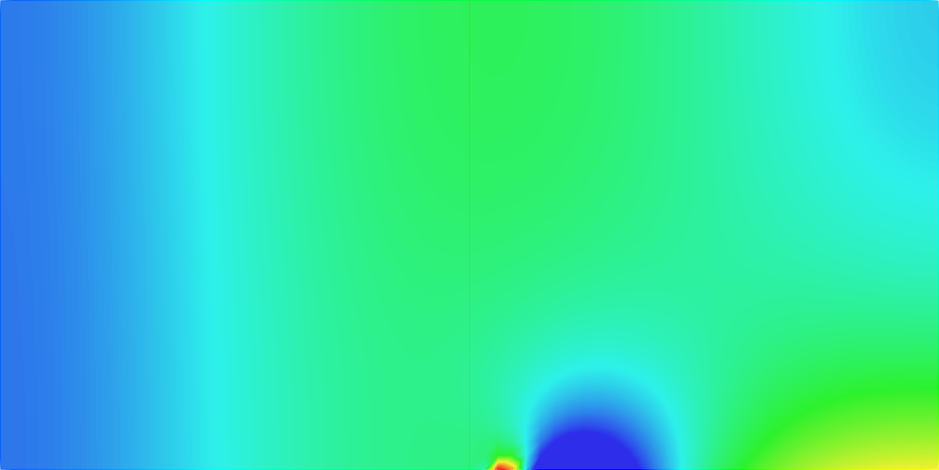}\\[3pt]
\includegraphics[width=3.2cm]{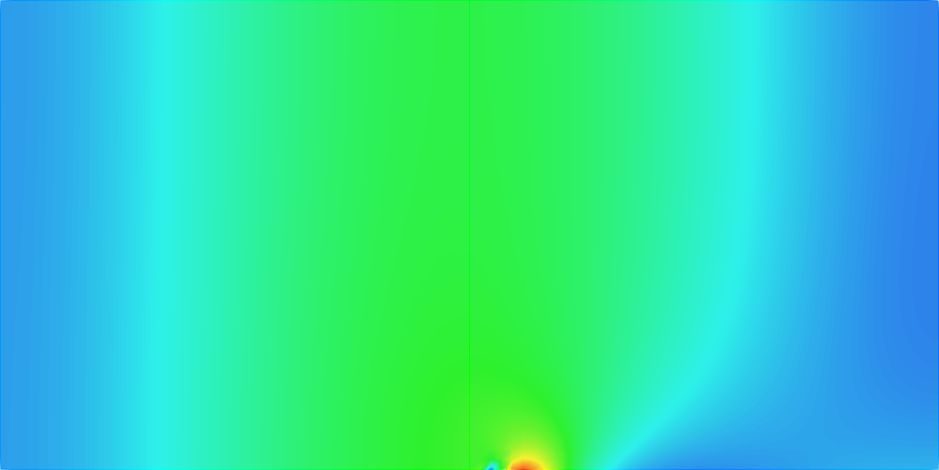}\,\includegraphics[width=3.2cm]{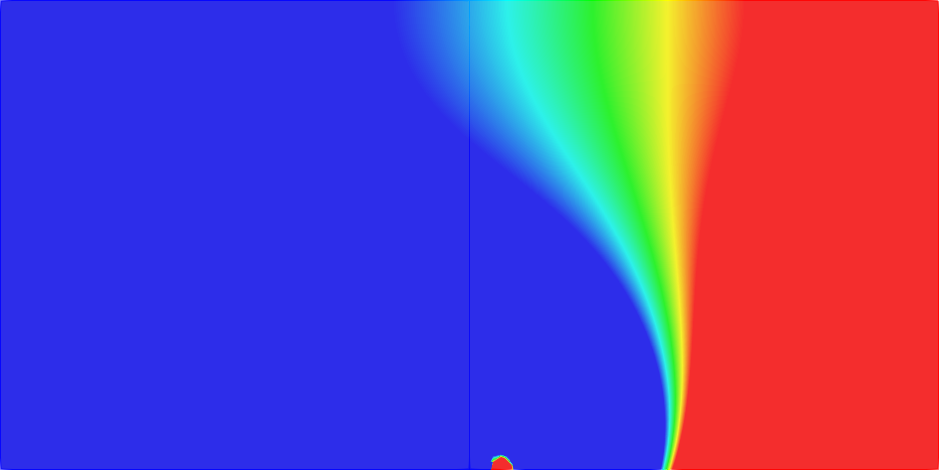}\,\includegraphics[width=3.2cm]{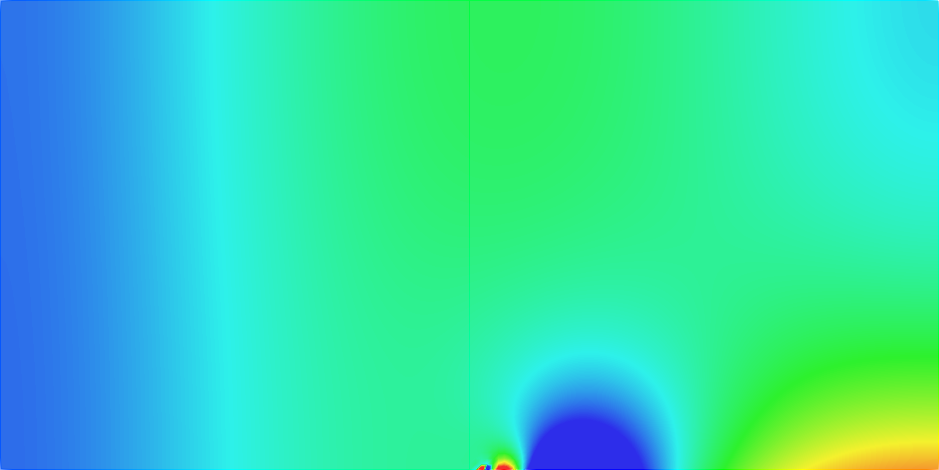}\,\includegraphics[width=3.2cm]{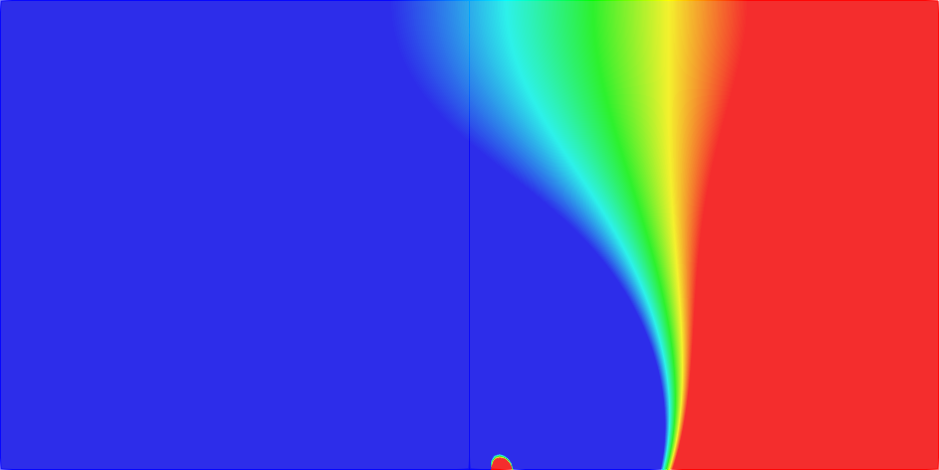}\,\includegraphics[width=3.2cm]{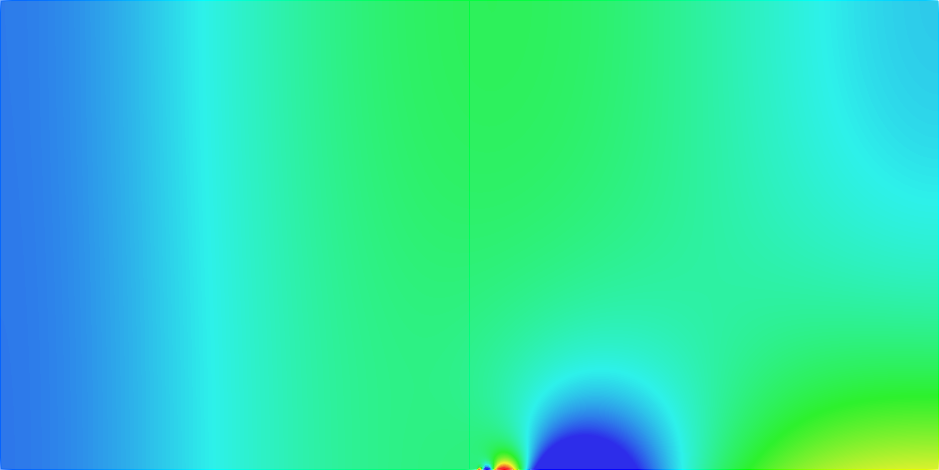}
\caption{Numerical resolution of (\ref{faible}) for ten different meshes with  $s=-1$ and $\alpha=1.5$.\label{Fig_sm1_alpha_1p5}}
\end{figure}

In Figure \ref{Fig_sp1_alpha_1p5}, we display results in the good sign case $s=1$ with $\alpha=1.5$. In agreement with Theorem \ref{good_sign}, which guarantees that (\ref{faible}) is well-posed, we observe that the numerical solution converges when the mesh is refined. In that situation, the corresponding sesquilinear form is coercive and we can apply the C\'ea's lemma. The question of the approximability of $\mV_\alpha(\Om)$ by usual Lagrange finite elements spaces however remains to be studied.

\begin{figure}[!ht]
\centering
\includegraphics[width=3.2cm]{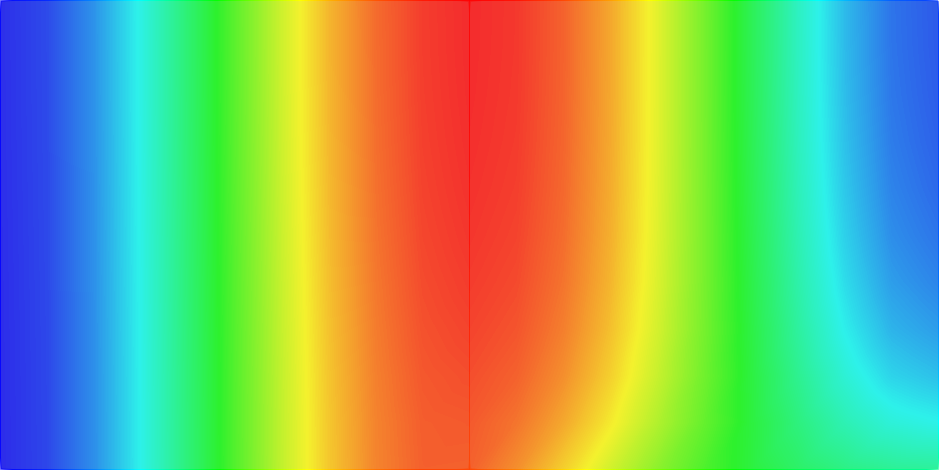}\,\includegraphics[width=3.2cm]{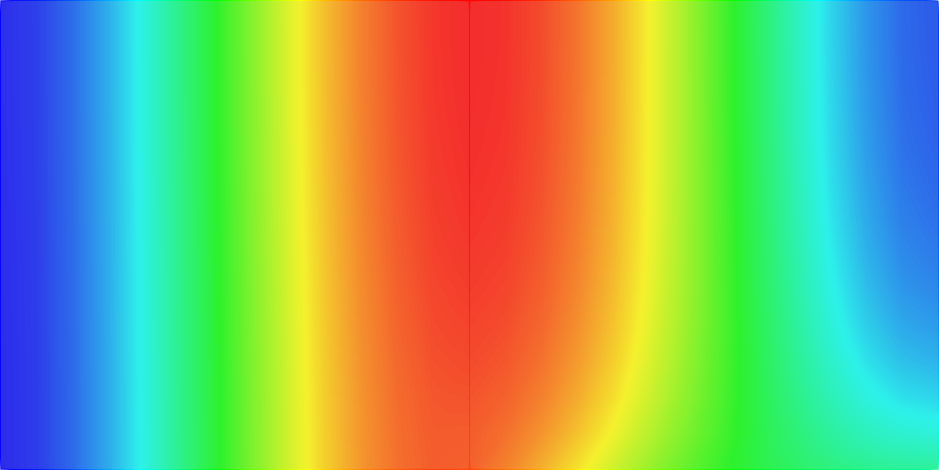}\,\includegraphics[width=3.2cm]{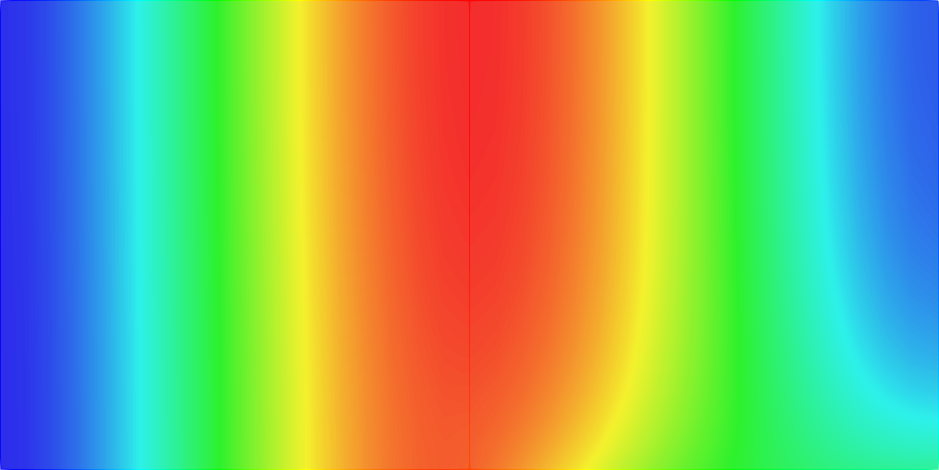}\,\includegraphics[width=3.2cm]{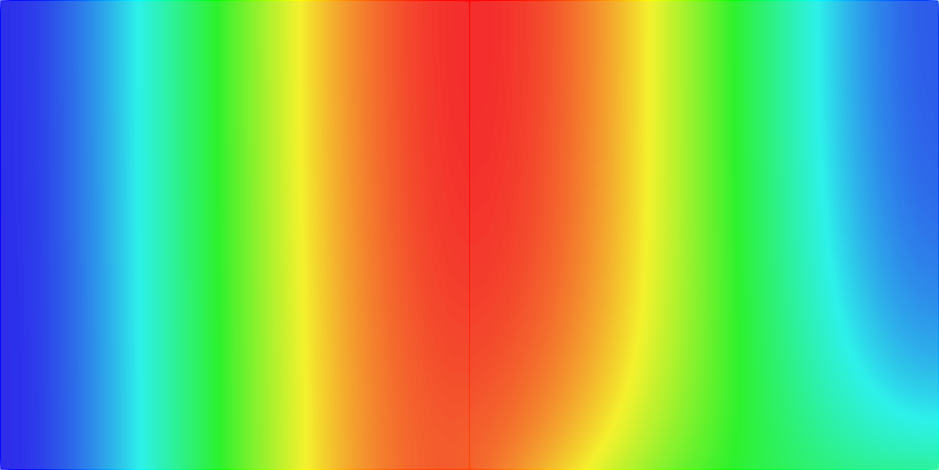}\,\includegraphics[width=3.2cm]{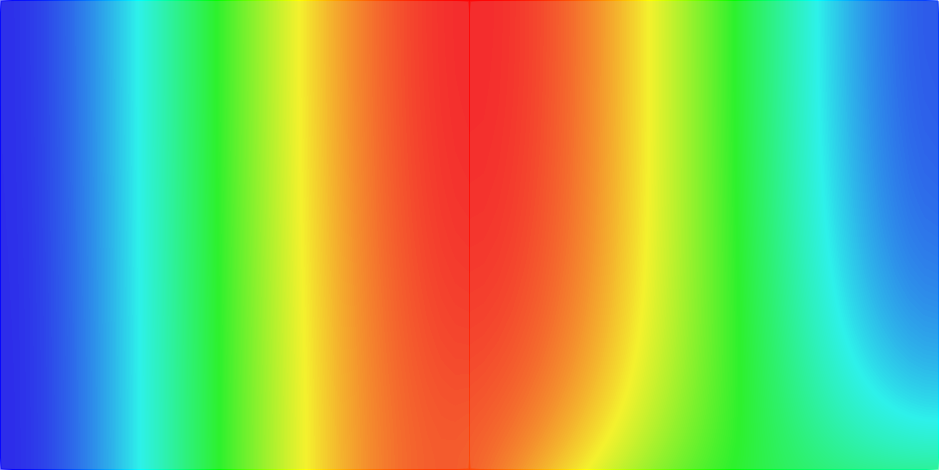}\\[3pt]
\includegraphics[width=3.2cm]{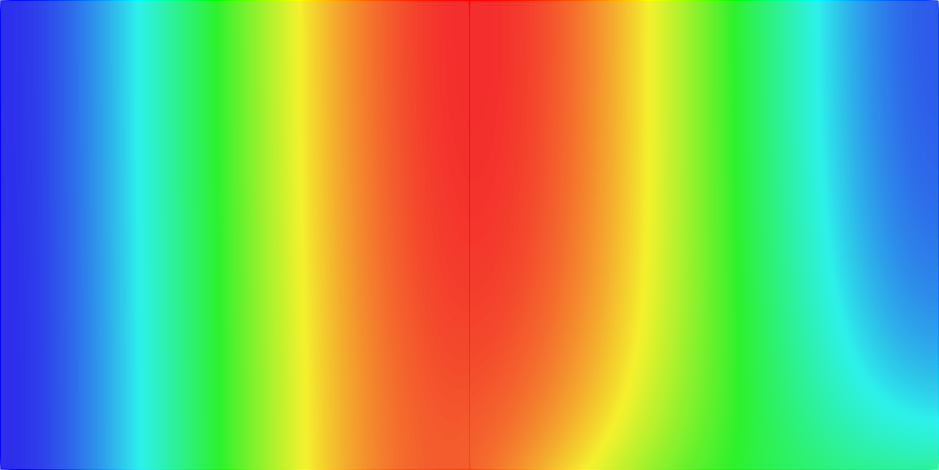}\,\includegraphics[width=3.2cm]{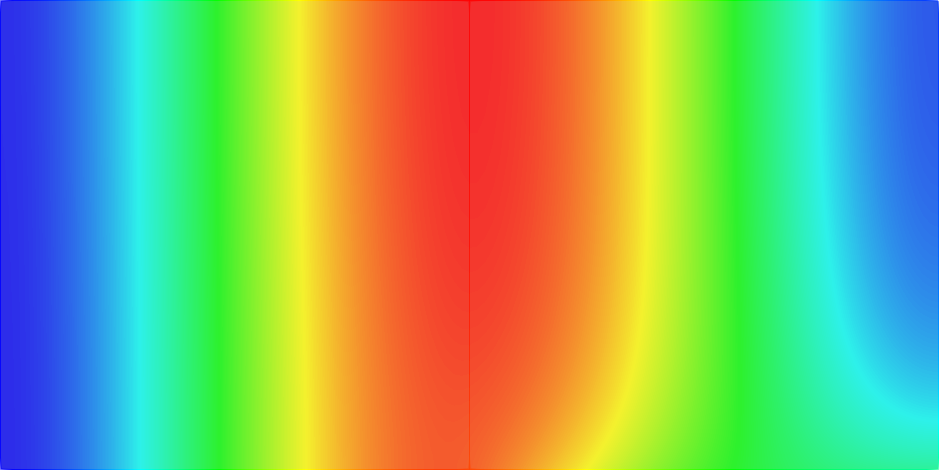}\,\includegraphics[width=3.2cm]{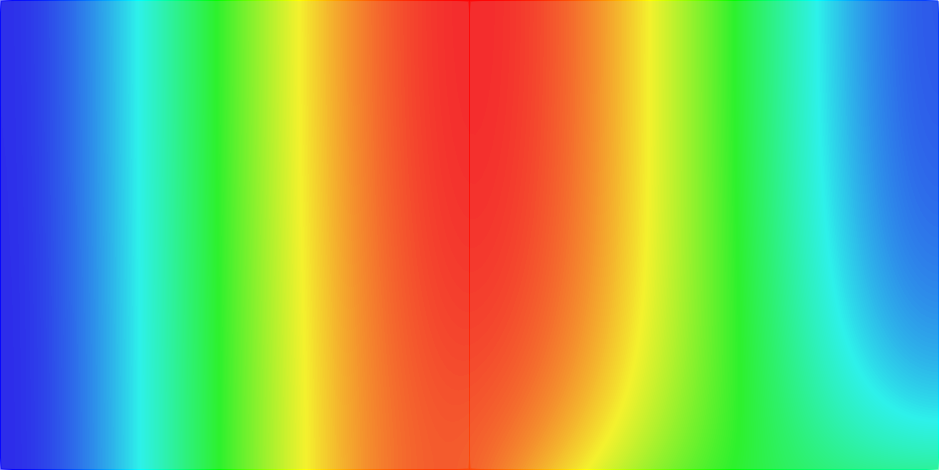}\,\includegraphics[width=3.2cm]{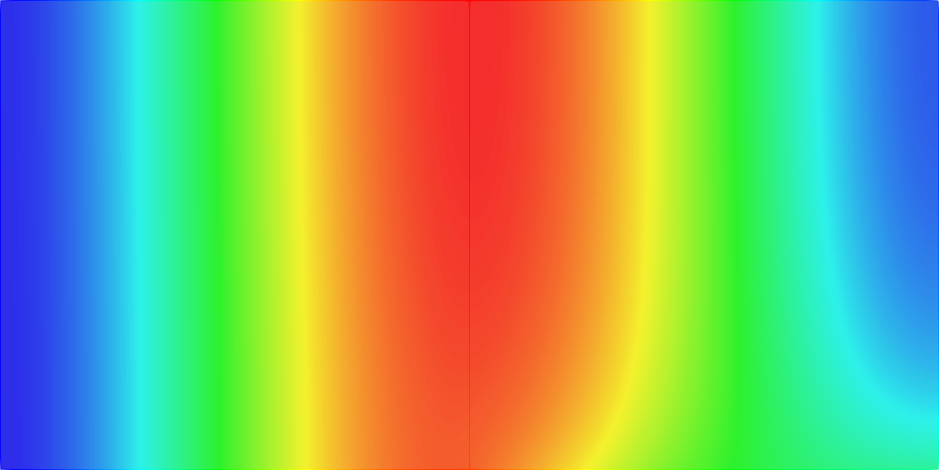}\,\includegraphics[width=3.2cm]{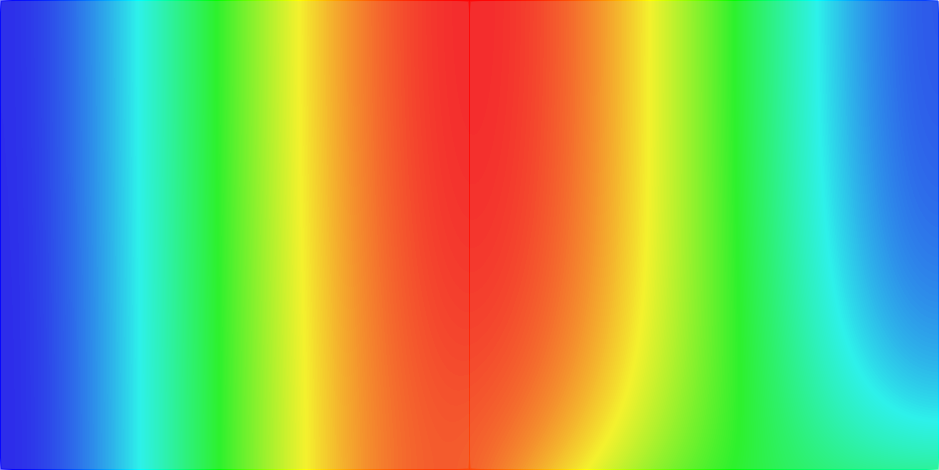}
\caption{Numerical resolution of (\ref{faible}) for ten different meshes with  $s=1$ and $\alpha=1.5$.\label{Fig_sp1_alpha_1p5}}
\end{figure}

Finally, in Figures \ref{Fig_Cha_alpha_0p5}--\ref{Fig_Cha_alpha_1p5}, we display results concerning Problem (\ref{faible bis}) which involves an impedance which is both vanishing and sign-changing. For $\alpha=0.5$, the numerical solution converges when the mesh is refined (Figure \ref{Fig_Cha_alpha_0p5}). This is not the case for $\alpha=1$ (Figure \ref{Fig_Cha_alpha_1}). These observations are in line with the statements of Theorems \ref{casgeneral_bis} and \ref{cas1_bis}. For $\alpha=1.5$ (Figure \ref{Fig_Cha_alpha_1p5}), the numerical solution does not converge either. Proving that the operator $B$ defined in (\ref{DefOpB}) is not of Fredholm type in that situation remains to be done.

\begin{figure}[!ht]
\centering
\includegraphics[width=3.2cm]{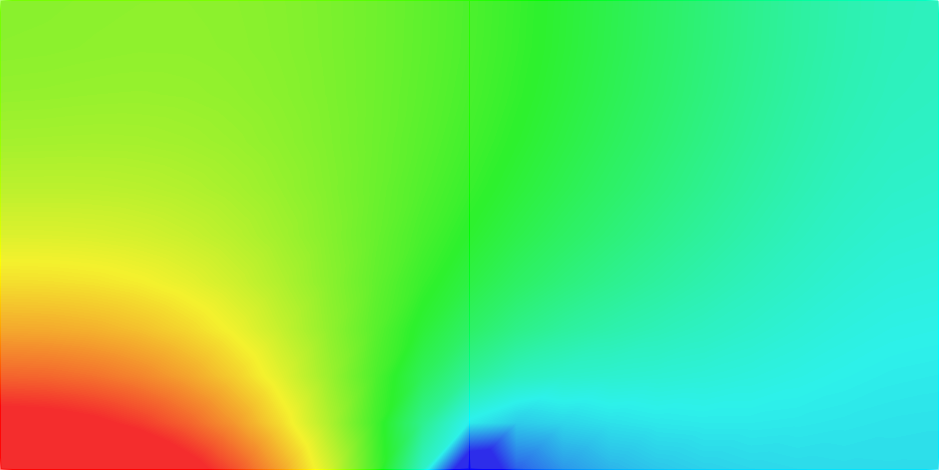}\,\includegraphics[width=3.2cm]{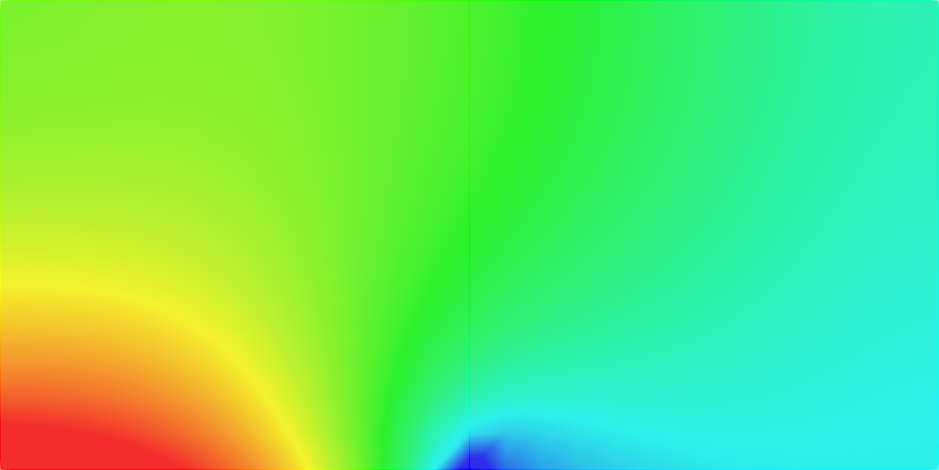}\,\includegraphics[width=3.2cm]{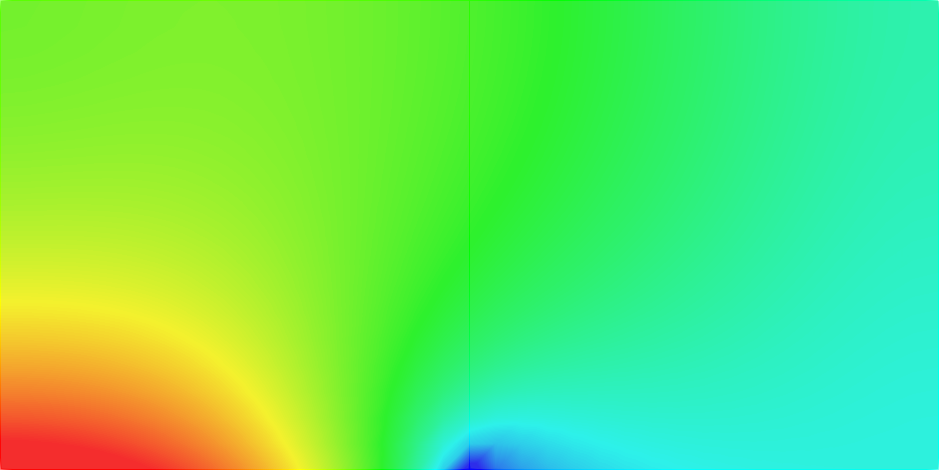}\,\includegraphics[width=3.2cm]{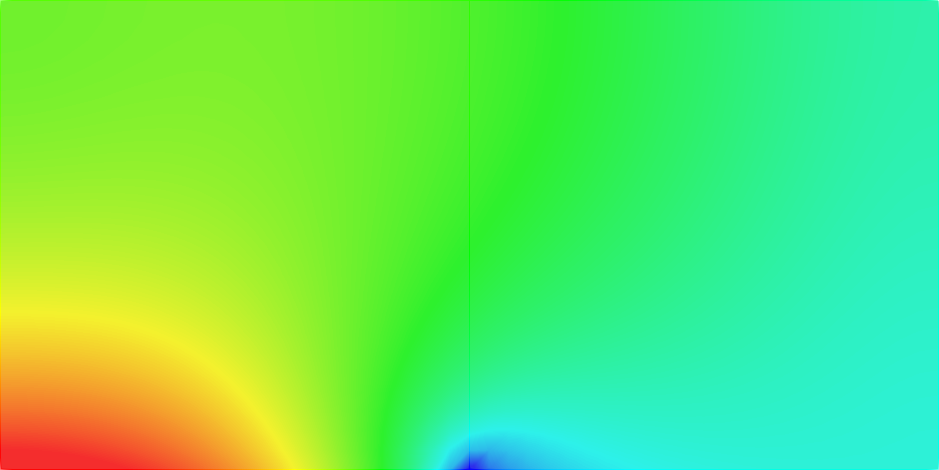}\,\includegraphics[width=3.2cm]{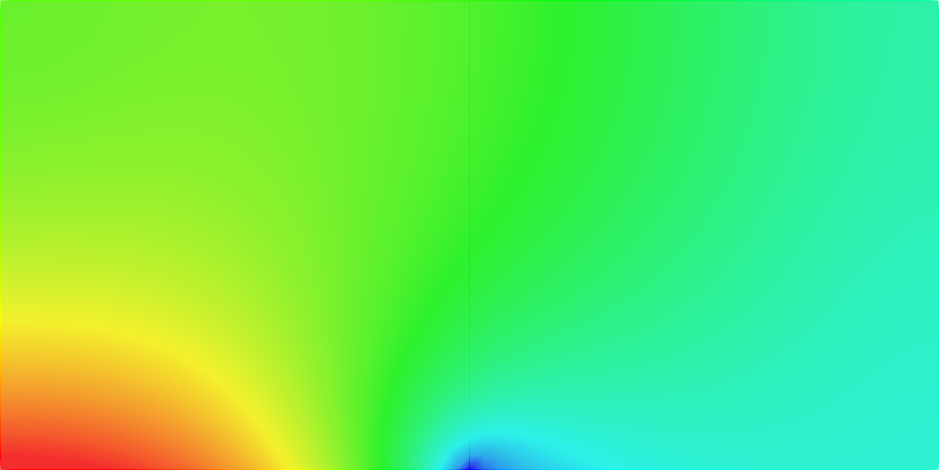}\\[3pt]
\includegraphics[width=3.2cm]{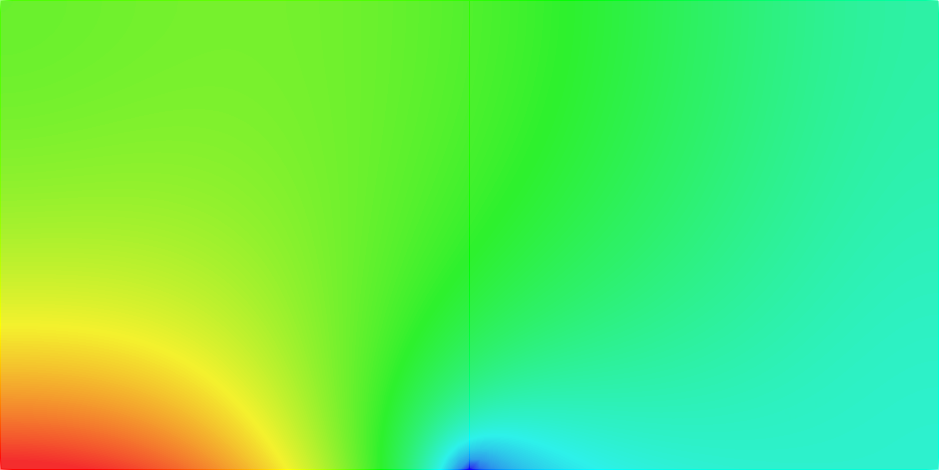}\,\includegraphics[width=3.2cm]{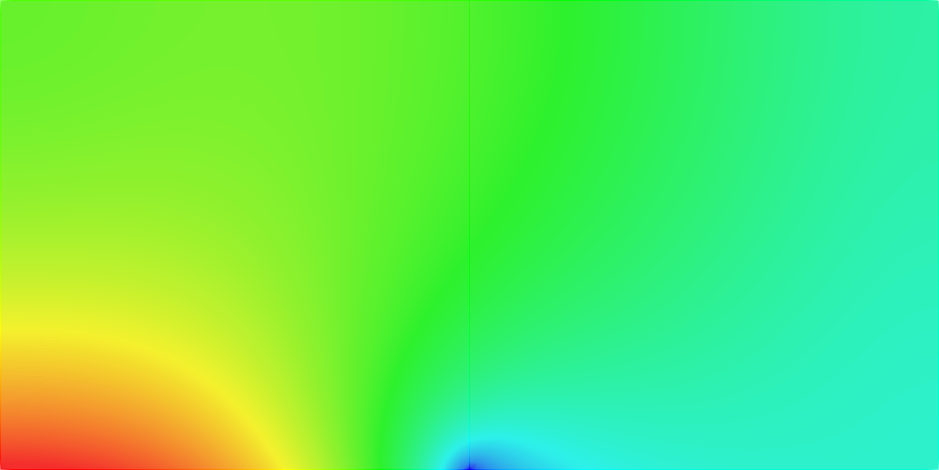}\,\includegraphics[width=3.2cm]{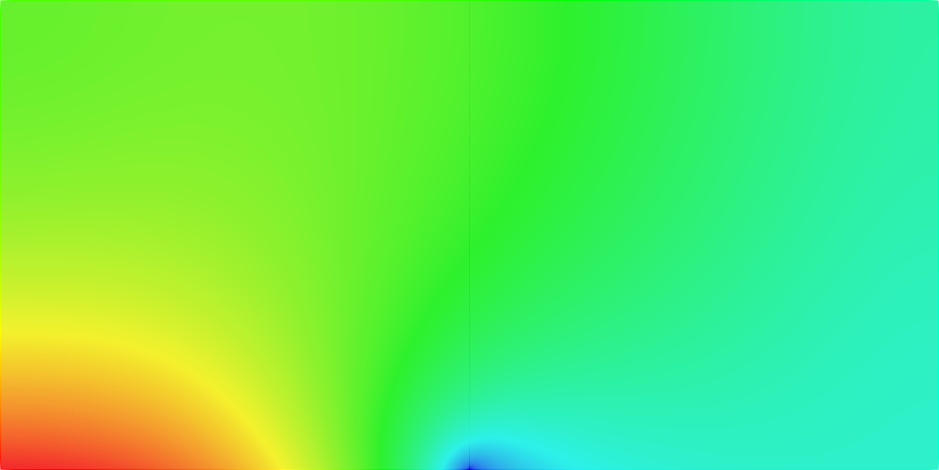}\,\includegraphics[width=3.2cm]{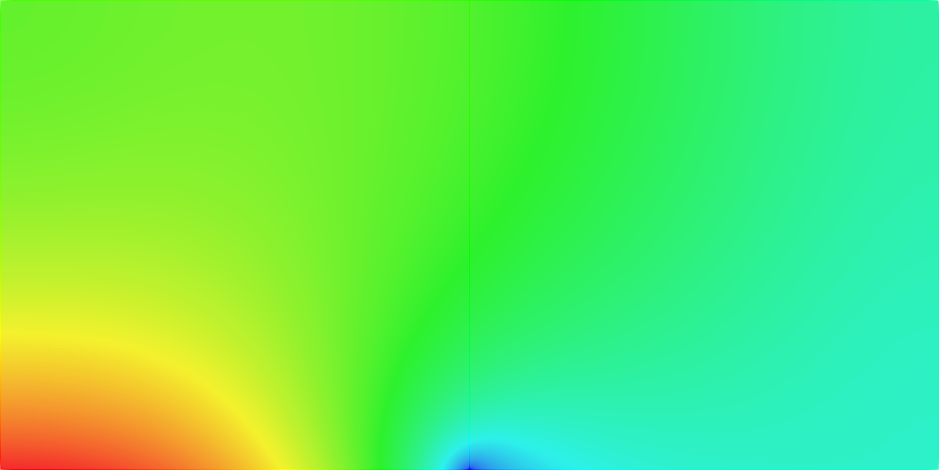}\,\includegraphics[width=3.2cm]{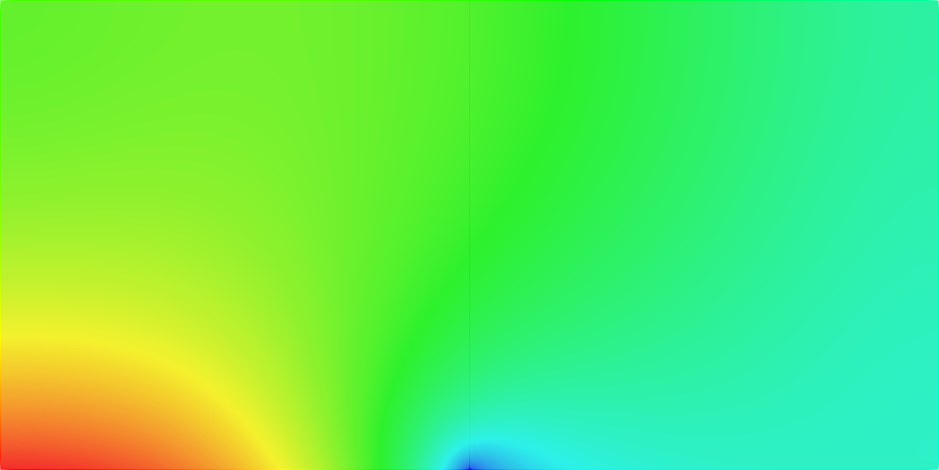}
\caption{Numerical resolution of (\ref{faible bis}) for ten different meshes with $\alpha=0.5$.\label{Fig_Cha_alpha_0p5}}
\end{figure}

\begin{figure}[!ht]
\centering
\includegraphics[width=3.2cm]{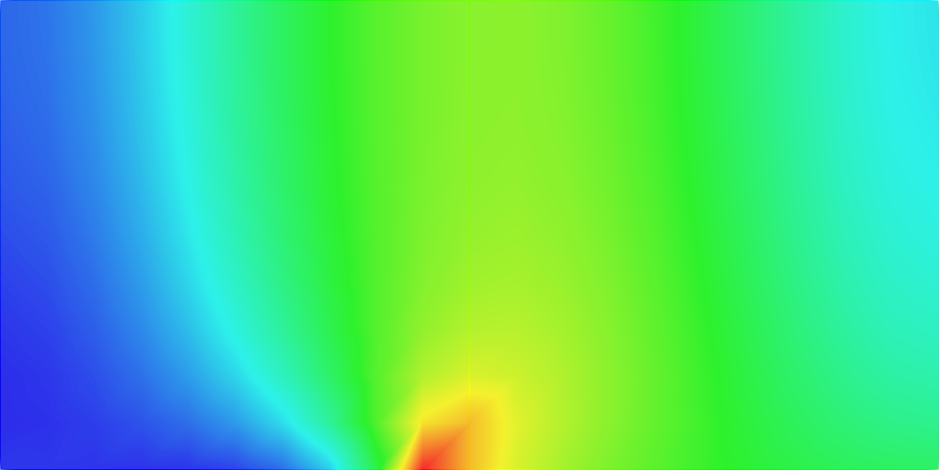}\,\includegraphics[width=3.2cm]{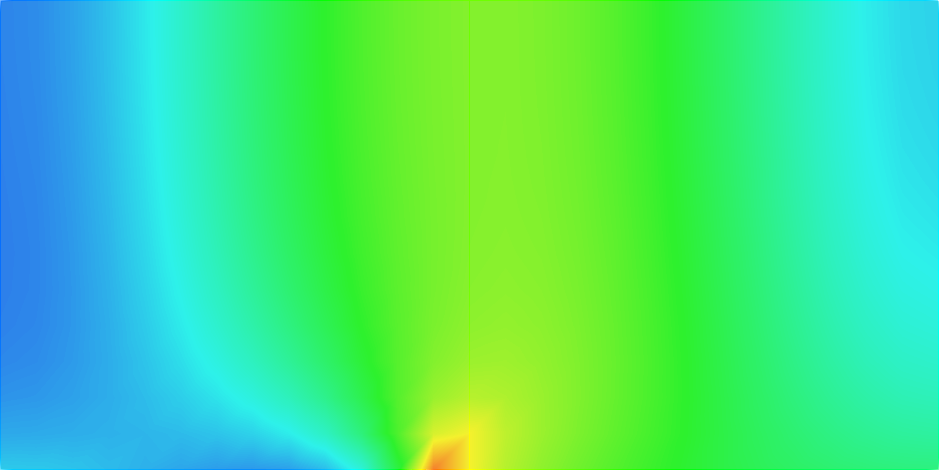}\,\includegraphics[width=3.2cm]{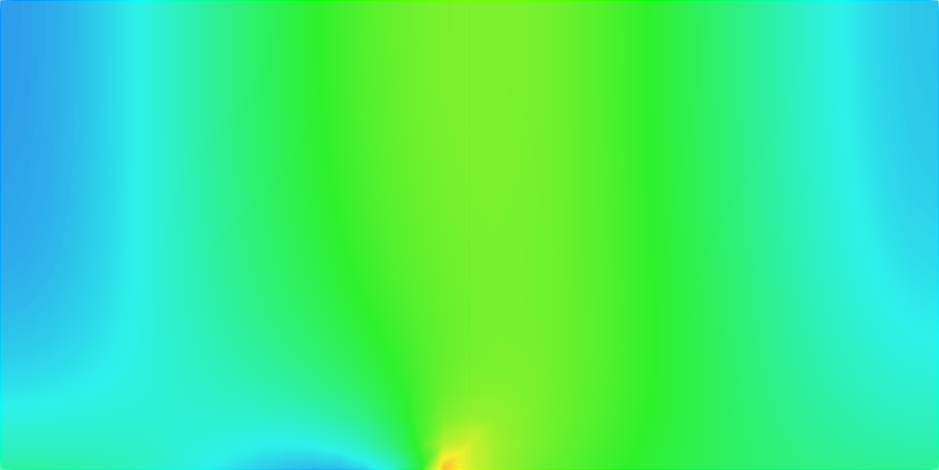}\,\includegraphics[width=3.2cm]{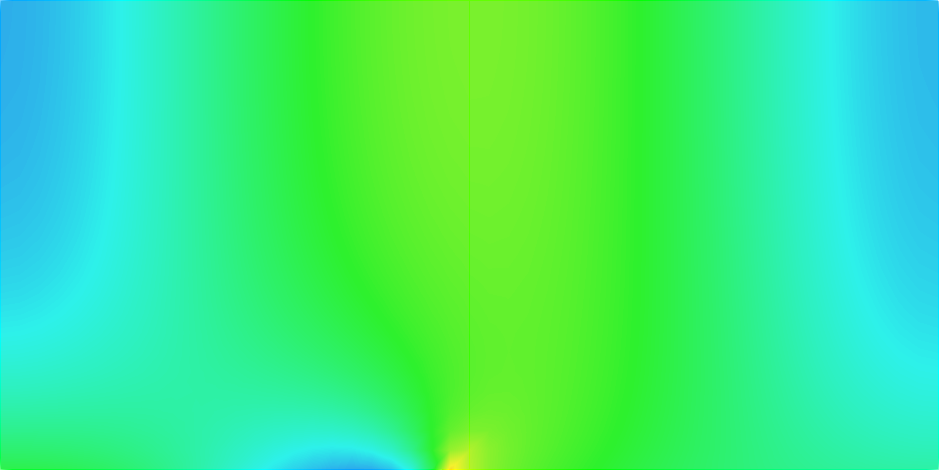}\,\includegraphics[width=3.2cm]{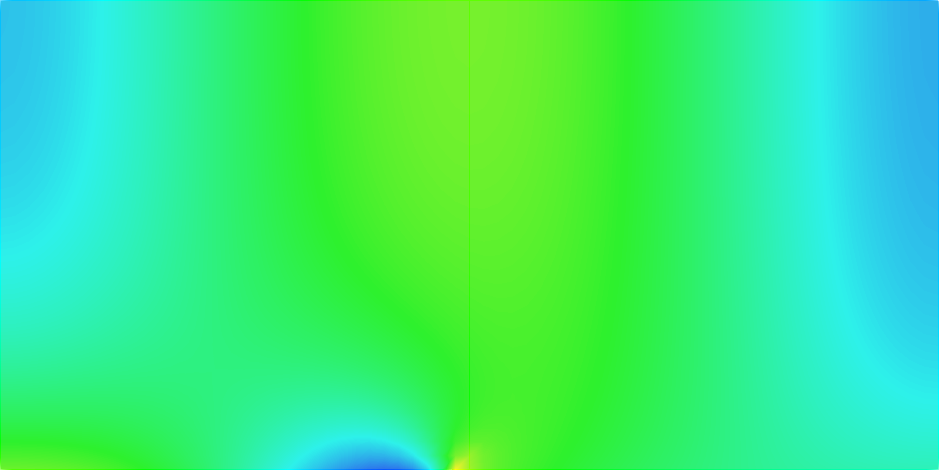}\\[3pt]
\includegraphics[width=3.2cm]{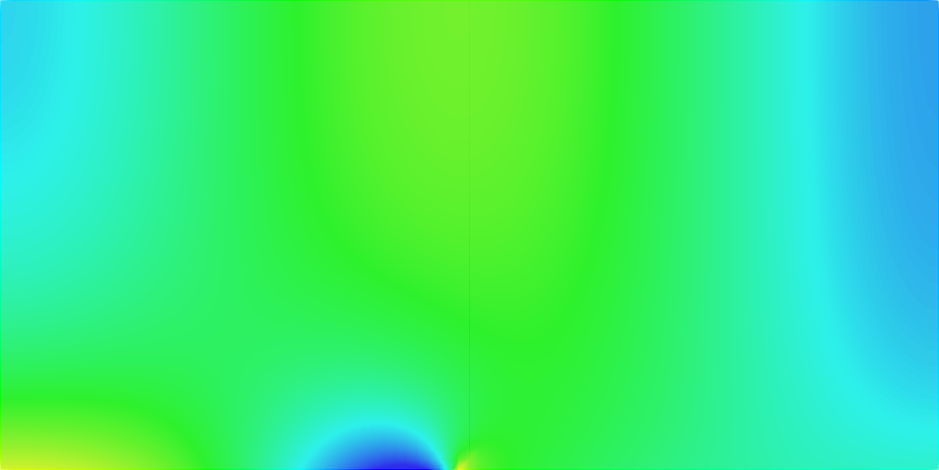}\,\includegraphics[width=3.2cm]{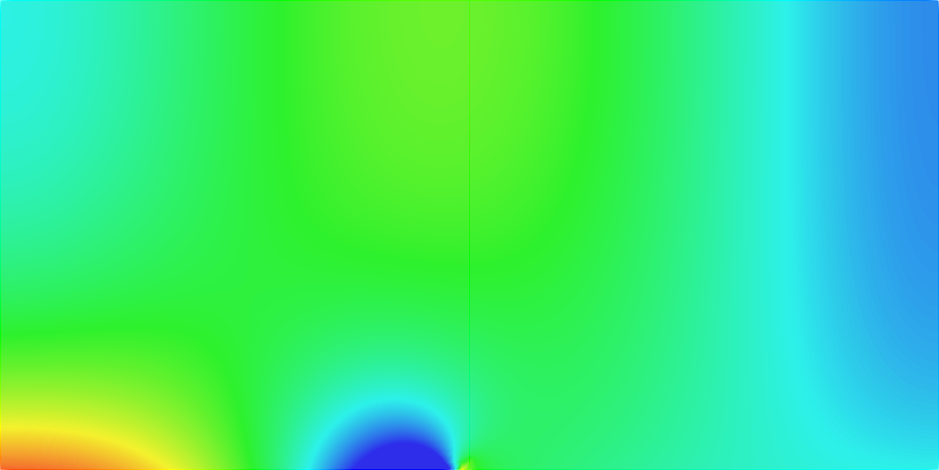}\,\includegraphics[width=3.2cm]{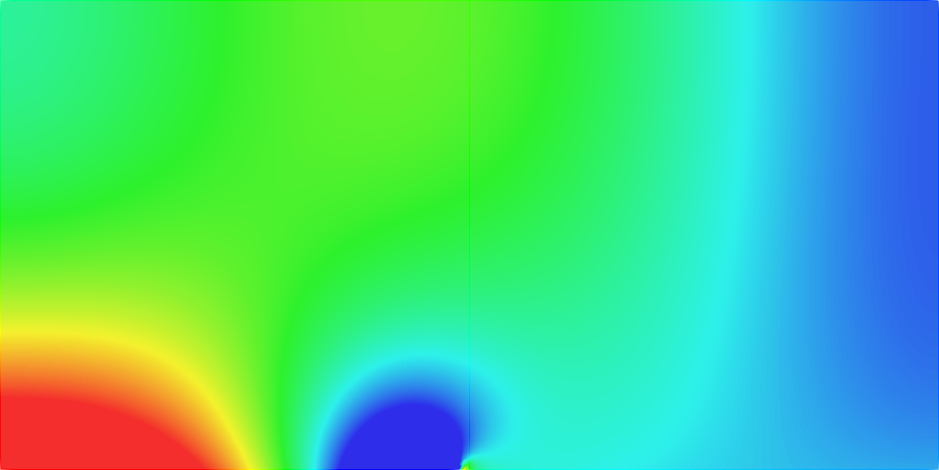}\,\includegraphics[width=3.2cm]{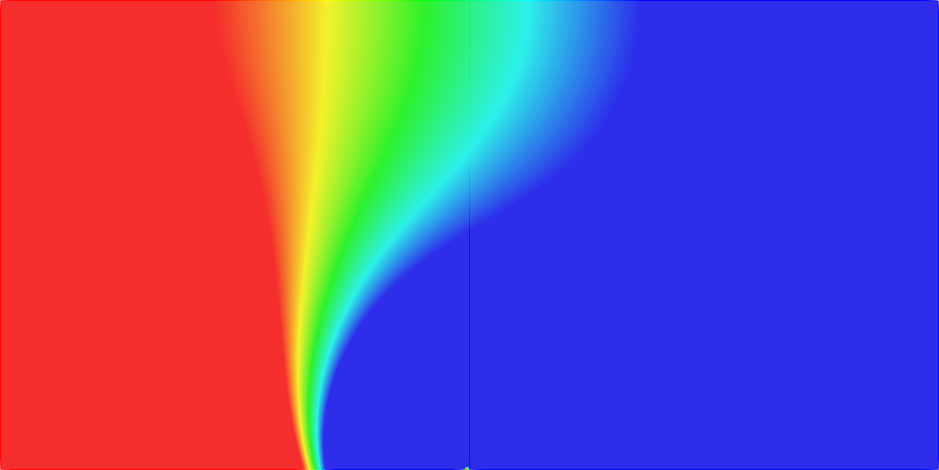}\,\includegraphics[width=3.2cm]{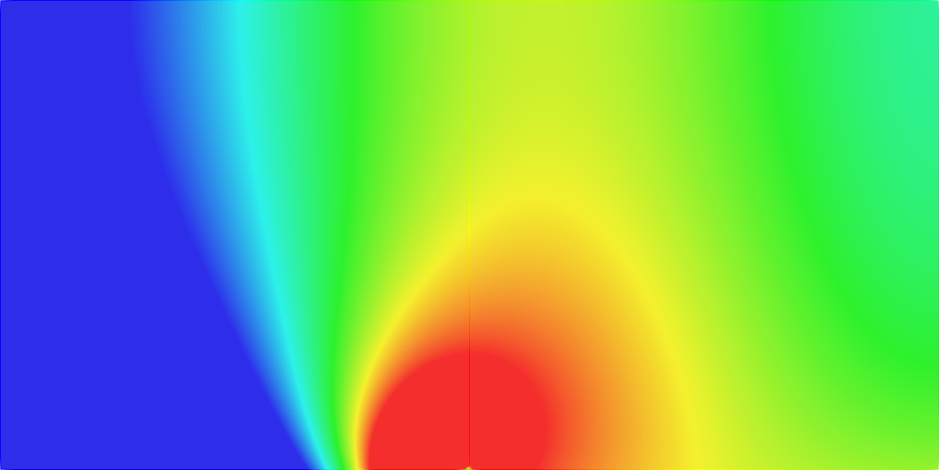}
\caption{Numerical resolution of (\ref{faible bis}) for ten different meshes with $\alpha=1$.\label{Fig_Cha_alpha_1}}
\end{figure}

\begin{figure}[!ht]
\centering
\includegraphics[width=3.2cm]{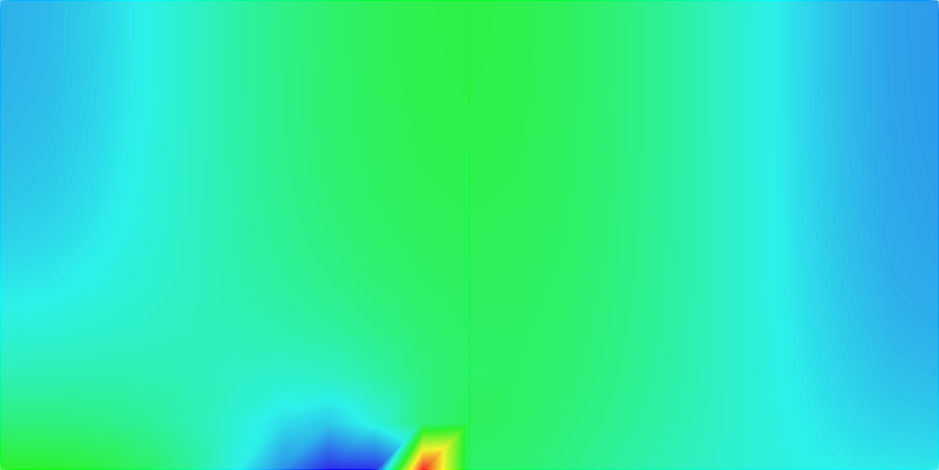}\,\includegraphics[width=3.2cm]{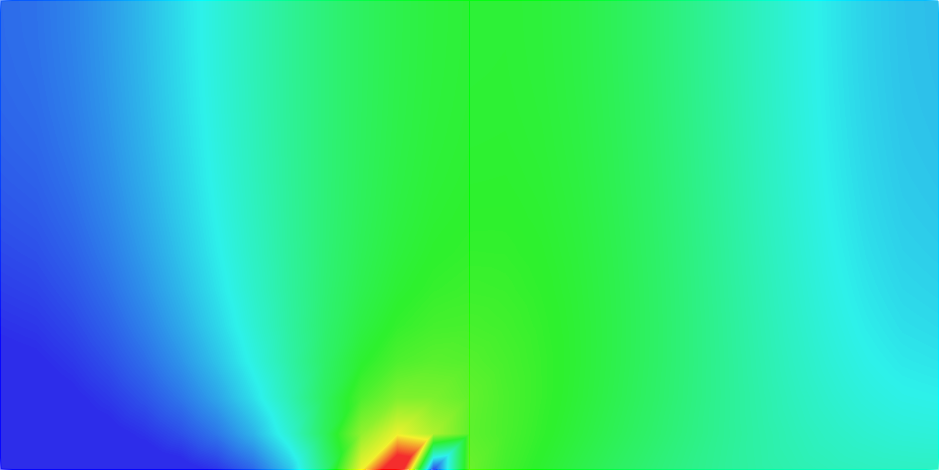}\,\includegraphics[width=3.2cm]{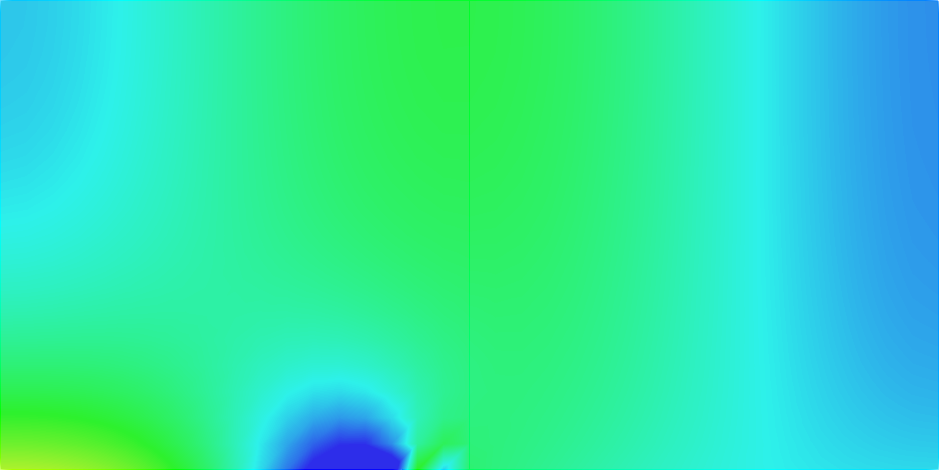}\,\includegraphics[width=3.2cm]{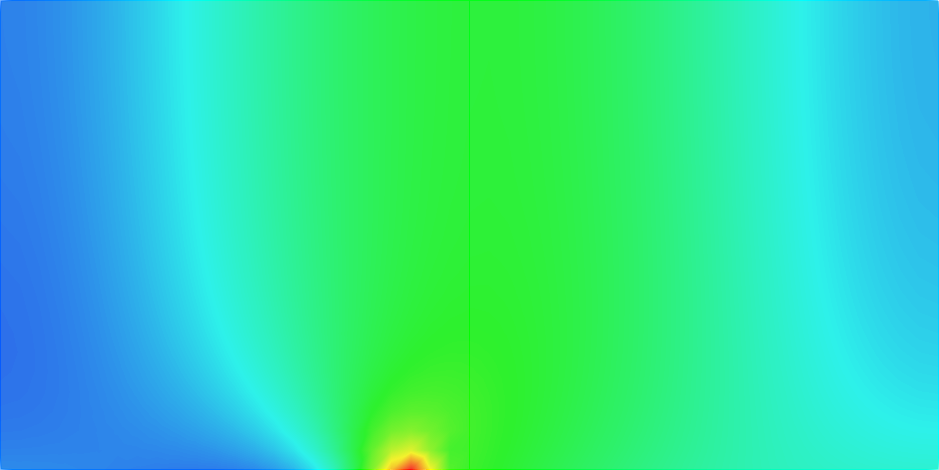}\,\includegraphics[width=3.2cm]{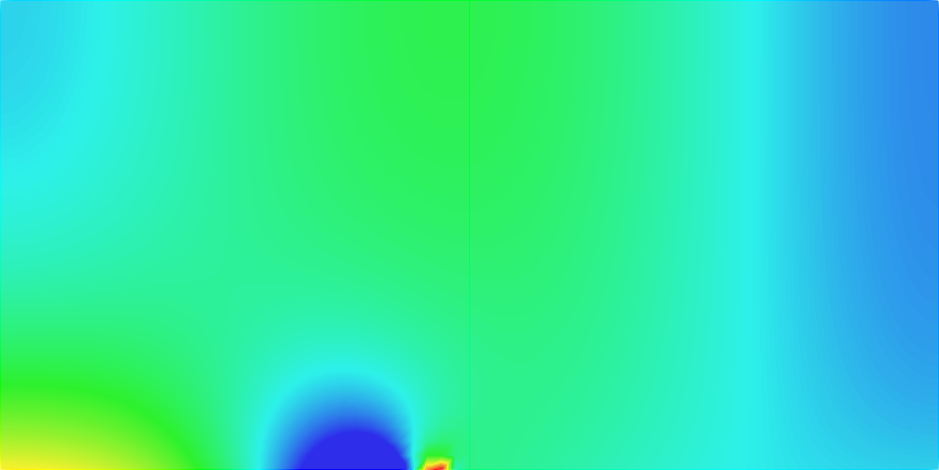}\\[3pt]
\includegraphics[width=3.2cm]{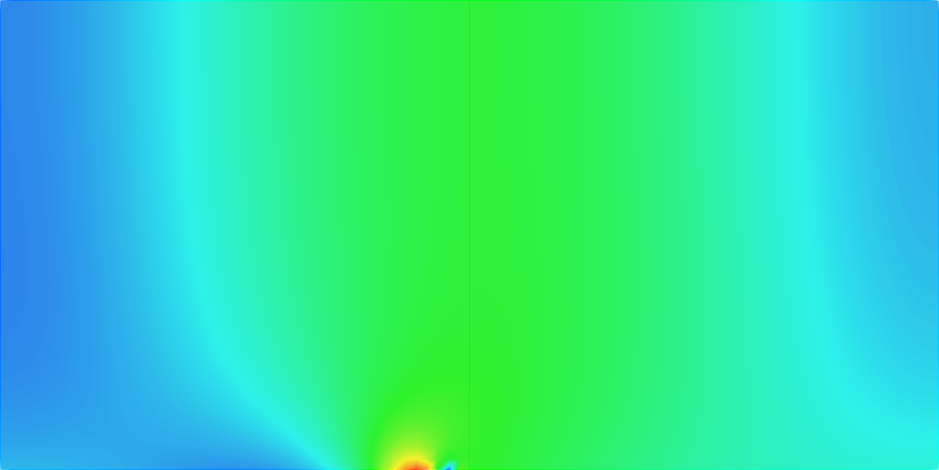}\,\includegraphics[width=3.2cm]{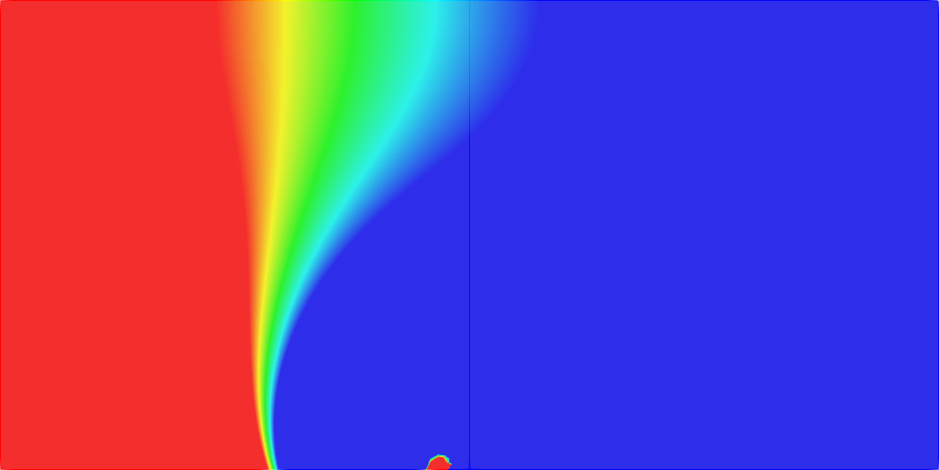}\,\includegraphics[width=3.2cm]{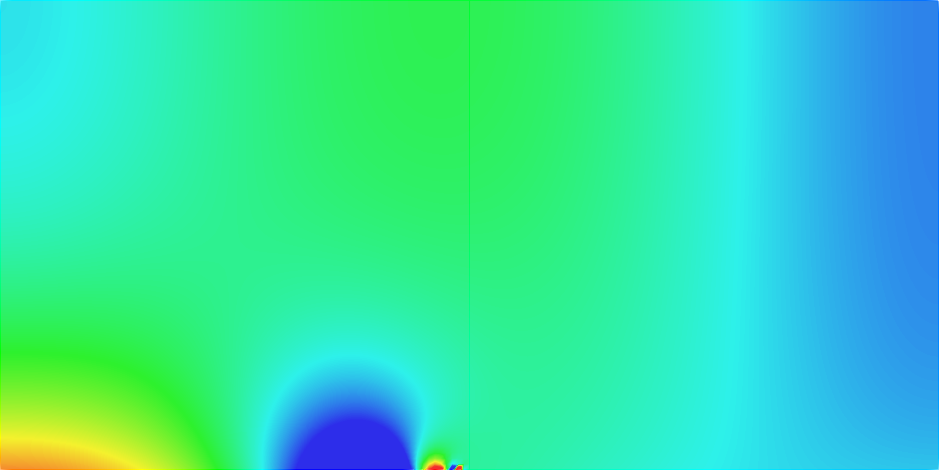}\,\includegraphics[width=3.2cm]{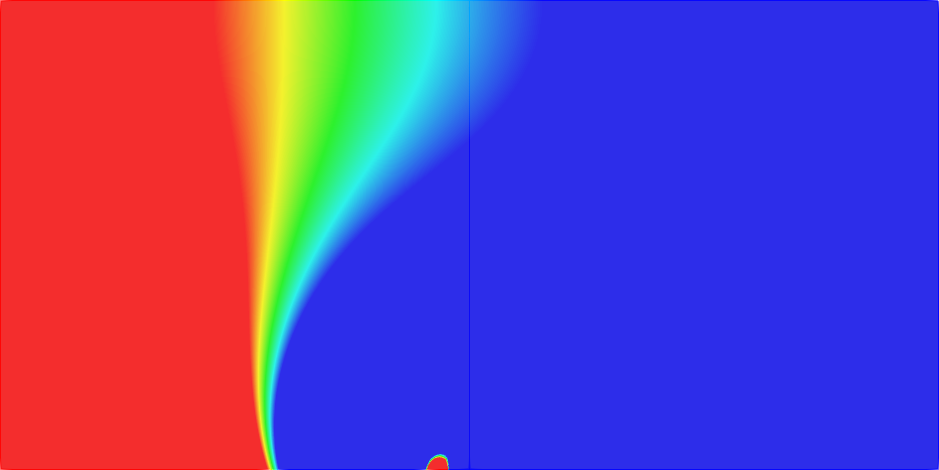}\,\includegraphics[width=3.2cm]{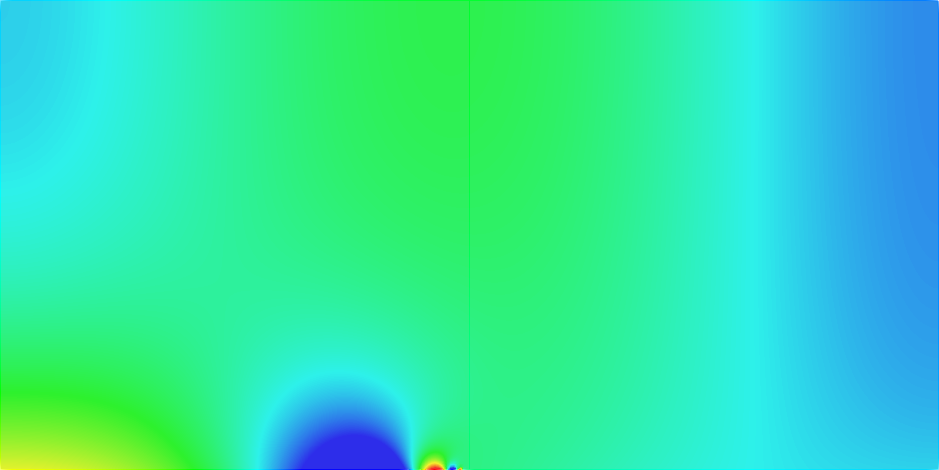}
\caption{Numerical resolution of (\ref{faible bis}) for ten different meshes with $\alpha=1.5$.\label{Fig_Cha_alpha_1p5}}
\end{figure}

\newpage
\section{Concluding remarks}\label{SectionConcludingRemarks}

Let us discuss here a few possible extensions and present some open questions following our study. In the bad sign case $s=-1$, we proved that the operators $A$, $B$ defined respectively in (\ref{DefOpA}), (\ref{DefOpB}) are Fredholm for $\alpha\in[0,1)$ and are not Fredholm for $\alpha=1$. For $s=-1$ and $\alpha>1$ (see Figure \ref{Fig_sm1_alpha_1p5}), we expect that $A$ and $B$ are not Fredholm but unfortunately, we are unable to establish this result. This is due to the fact that then there are no singularities with separate variables in polar coordinates at the origin. 

We worked with the equation $-\Delta u+u=f$ in $\Om$. For $\alpha\in[0,1]$, we could similarly consider the cases $-\Delta u=f$ or $-\Delta u-\om^2 u=f$ with $\om\in\Cplx$. This only induces compact perturbations and does not affect Fredholm properties for $A$ and $B$. Besides, we imposed generalized impedance boundary conditions on parts of the boundary which coincide with flat segments. This feature is used in particular to establish the compactness result of Proposition \ref{prop-compacite} (see (\ref{Def1DTo2D})) or to compute the singularities in (\ref{system}). When $\Gamma_\pm$ are smooth curves, we do not expect significant differences in the results compared to what we have obtained. However the proofs need to be written rigorously. 

In this study, for $s=-1$, we did not consider the question of injectivity of $A$ and $B$. For $\alpha\in[0,1)$, due do the Fredholm property, we know that these operators have a kernel of finite dimension. However showing that this kernel reduces to the null function is an open question. We do not expect that this occurs in all geometries but even finding a simple $\Om$ where we can prove that this is true reveals complications because we can not use separation of variables due to the form of the operators. We focused our attention on the 2D case. In 3D, the singularities are different and the analysis must be adapted. Additionally, in 3D one could consider situations where the impedance, which is then a function of two variables, vanishes on a line and not only at a point. The study of the problem in such a circumstance is completely open. Finally, in Section \ref{Section_Num} we presented simple numerical experiments whose results seem in accordance with our theorems. However there is no justification here. It would be interesting to establish results of convergence for our numerical methods when the mesh is refined and when $A$, $B$ are isomorphisms.

\bibliography{Biblio}
\bibliographystyle{plain}
\end{document}